\documentclass[10pt,letterpaper]{amsart}

\usepackage{amsmath}
\usepackage{amssymb}
\usepackage{amsthm}
\usepackage{amsfonts}
\usepackage{microtype}
\usepackage{mathrsfs}
\usepackage{graphicx}
\usepackage{color}
\usepackage{latexsym}
\usepackage{rotating}
\usepackage{xspace}
\usepackage[all]{xy}
\usepackage{longtable}
\usepackage{leftidx}
\usepackage{mathtools}
\usepackage{titlesec}
\usepackage{setspace}
\usepackage{tikz}
\usepackage{tikz-cd}
\usepackage{url}
\usepackage[autostyle]{csquotes}
\usetikzlibrary{calc}
\usetikzlibrary{arrows}
\usetikzlibrary{decorations.markings,intersections}
\usepackage{geometry}
\usepackage[final]{pdfpages}
\usepackage{fancyhdr}
\usepackage{tabularx}

\newtheorem{theorem}{Theorem}[section]
\numberwithin{equation}{theorem}
\newtheorem{lemma}[theorem]{Lemma}

\newtheorem{corollary}[theorem]{Corollary}

\theoremstyle{definition}

\newtheorem{remark}[theorem]{Remark}

\theoremstyle{conjecture}

\newtheorem*{theorem*}{Theorem}

\newcommand{\Ker}{\operatorname{Ker}}
\newcommand{\Coker}{\operatorname{Coker}}
\newcommand{\im}{\operatorname{Im}}

\newcommand{\Cone}{\operatorname{Cone}}
\newcommand{\Ext}{\operatorname{Ext}}

\newcommand{\Tor}{\operatorname{Tor}}
\newcommand{\Hom}{\operatorname{Hom}}

\newcommand{\Obj}{\operatorname{Obj}}
\newcommand{\Mor}{\operatorname{Mor}}

\newcommand{\op}{\operatorname{op}}
\newcommand{\Nor}{\operatorname{Nor}}
\newcommand{\DK}{\operatorname{DK}}

\newcommand{\Sym}{\operatorname{Sym}}
\newcommand{\Ho}{\operatorname{Ho}}

\newcommand{\Der}{\operatorname{Der}}

\newcommand{\ev}{\operatorname{ev}}
\newcommand{\Barr}{\operatorname{Bar}}
\newcommand{\Diag}{\operatorname{Diag}}

\newcommand{\suchthat}{\;\ifnum\currentgrouptype=16 \middle\fi|\;}

\makeatletter
\newcommand{\hocolim@}[2]{%
  \vtop{\m@th\ialign{##\cr
    \hfil$#1\operator@font holim$\hfil\cr
    \noalign{\nointerlineskip\kern1.5\ex@}#2\cr
    \noalign{\nointerlineskip\kern-\ex@}\cr}}%
}
\newcommand{\hocolim}{%
  \mathop{\mathpalette\hocolim@{\rightarrowfill@\textstyle}}\nmlimits@
}
\makeatother

\makeatletter
\newcommand{\holim@}[2]{%
  \vtop{\m@th\ialign{##\cr
    \hfil$#1\operator@font holim$\hfil\cr
    \noalign{\nointerlineskip\kern1.5\ex@}#2\cr
    \noalign{\nointerlineskip\kern-\ex@}\cr}}%
}
\newcommand{\holim}{%
  \mathop{\mathpalette\holim@{\leftarrowfill@\textstyle}}\nmlimits@
}
\makeatother

\makeatletter
\def\@secnumfont{\bfseries}
\def\section{\@startsection{section}{1}%
  \z@{.7\linespacing\@plus\linespacing}{.5\linespacing}%
  {\normalfont\Large\bfseries\filcenter}}
\def\subsection{\@startsection{subsection}{2}%
  \z@{.5\linespacing\@plus.7\linespacing}{-.5em}%
  {\normalfont\large\bfseries}}
\makeatother

\DeclareFontFamily{OT1}{pzc}{}
\DeclareFontShape{OT1}{pzc}{m}{it}{<-> s * [1.20] pzcmi7t}{}
\DeclareMathAlphabet{\mathpzc}{OT1}{pzc}{m}{it}

\makeatletter
\def\moverlay{\mathpalette\mov@rlay}
\def\mov@rlay#1#2{\leavevmode\vtop{%
   \baselineskip\z@skip \lineskiplimit-\maxdimen
   \ialign{\hfil$\m@th#1##$\hfil\cr#2\crcr}}}
\newcommand{\charfusion}[3][\mathord]{
    #1{\ifx#1\mathop\vphantom{#2}\fi
        \mathpalette\mov@rlay{#2\cr#3}
      }
    \ifx#1\mathop\expandafter\displaylimits\fi}
\makeatother

\makeatletter
\providecommand{\bigsqcap}{%
  \mathop{%
    \mathpalette\@updown\bigsqcup
  }%
}
\newcommand*{\@updown}[2]{%
  \rotatebox[origin=c]{180}{$\m@th#1#2$}%
}
\makeatother

\begin{document}

\author[Hossein Faridian]{Hossein Faridian}

\title[Quillen's Fundamental Spectral Sequences Revisited]
{Quillen's Fundamental Spectral Sequences Revisited}

\address{Hossein Faridian, School of Mathematical and Statistical Sciences, Clemson University, SC 29634, USA.}
\email{hfaridi@g.clemson.edu}



\begin{abstract}
Quillen's fundamental spectral sequences relate Andr\'{e}-Quillen homology and cohomology to Tor and Ext functors. The five-term exact sequences arising from these spectral sequences are leveraged to characterize regular and complete intersection local rings. Despite their immense importance in the theory of Andr\'{e}-Quillen homology and cohomology, these spectral sequences are not treated in the literature as they merit. A sketchy argument with gaps for a special case of the first spectral sequence appears in an unpublished manuscript of Quillen, while the second spectral sequence is only stated without proof in another paper of Quillen. A rigorous proof of the spectral sequences requires a delicate investigation of the subtle structures involved. The goal of this expository article is to present a comprehensive and detailed proof of Quillen's fundamental spectral sequences in a readable and well-structured manner.
\end{abstract}

\maketitle

\sloppy
\raggedbottom

\tableofcontents

\section{Introduction}

In the late 60's, M. Andr\'{e} and D. Quillen independently introduced a (co)homology theory for commutative algebras that now goes by the name of Andr\'{e}-Quillen (co)homology. The theory was later developed for schemes and more generally for ringed topoi by L. Illusie. Andr\'{e} uses ad hoc definitions and constructions, Quillen deploys his model category theory, and Illusie utilizes Grothendieck topos theory; see \cite{An}, \cite{Qu3}, and \cite{Il}. The theory of Andr\'{e}-Quillen (co)homology has countless applications in local algebra, algebraic geometry, and number theory: characterization of regular and complete intersection local rings, characterization of separable field extensions, characterization of unramified extensions of number fields, and deformation theory of schemes, to name a few. For more applications of the theory, refer to \cite{MR}, \cite{Av1}, \cite{Iy}, \cite{BI}, and \cite{GS}.

To get a flavor of the theory, let $R$ be a commutative ring, $A$ a commutative $R$-algebra, and $M$ an $A$-module. A derivation of $A$ over $R$ with coefficients in $M$ is an $R$-homomorphism $D:A\rightarrow M$ that satisfies the Leibniz rule, i.e. $D(ab)=aD(b)+ bD(a)$ for every $a,b \in A$. The set of all derivations of $A$ over $R$ with coefficients in $M$, denoted by $\Der_{R}(A,M)$, is an $A$-submodule of $\Hom_{R}(A,M)$. In addition, the induced functor $\Der_{R}(A,-):\mathcal{M}\mathpzc{od}(A)\rightarrow \mathcal{M}\mathpzc{od}(A)$ is representable, i.e. there exists an $A$-module $\Omega_{A|R}$, the so-called module of differentials of $A$ over $R$, and a universal derivation $d_{A|R}:A \rightarrow \Omega_{A|R}$ that induces a natural $A$-isomorphism $\Der_{R}(A,M) \cong \Hom_{A}\left(\Omega_{A|R},M \right)$. The module of differentials $\Omega_{A|R}$ can be concretely described as $\Omega_{A|R}=\mathfrak{k}/\mathfrak{k}^{2}$ where $\mathfrak{k}=\Ker(\mu)$ in which $\mu: A\otimes_{R}A \rightarrow A$ is given by $\mu(a\otimes b)=ab$ for every $a,b\in A$. In addition, the universal derivation $d_{A|R}:A \rightarrow \Omega_{A|R}$ is given by $d_{A|R}(a)=(1\otimes a - a\otimes 1) + \mathfrak{k}^{2}$ for every $a\in A$. Now assume that $R \rightarrow S \rightarrow T$ are two homomorphisms of commutative rings, and $M$ is a $T$-module. Then one has the following exact sequences:
$$\Omega_{S|R}\otimes_{S}M \rightarrow \Omega_{T|R}\otimes_{T}M \rightarrow \Omega_{T|S}\otimes_{T}M \rightarrow 0$$
and
$$0 \rightarrow \Der_{S}(T,M) \rightarrow \Der_{R}(T,M) \rightarrow \Der_{R}(S,M)$$
Historically speaking, people grew interested in extending these sequences beyond the above-displayed three terms. Grothendieck extended the second sequence to a six-term exact sequence by the Exalcomm functor, Lichtenbaum and Schlessinger extended it further to a nine-term exact sequence, and finally, Andr\'{e} and Quillen were able to extend the sequences indefinitely; see \cite{Gr}, \cite{LS}, \cite{An}, \cite{Qu1}, and \cite{Qu3}. Viewing derivations and differentials as functors of commutative algebras, one requires to have some sort of derived functors in order to extend the sequences. Applying the general theory of derived functors in model categories to the category of simplicial commutative algebras, one can construct the cotangent complex $\mathbb{L}^{A|R}\in \Obj\left(\mathcal{D}(A)\right)$ as the image of the identity map $1^{A}:A\rightarrow A$ under the left derived functor of the abelianization functor followed by the right derived functor of the normalization functor. The cotangent complex is in turn used to define the Andr\'{e}-Quillen homology and cohomology modules as $H_{i}^{AQ}\left(A|R;M \right)= \Tor_{i}^{A}\left(\mathbb{L}^{A|R},M\right)$ and $H_{AQ}^{i}\left(A|R;M \right)= \Ext_{A}^{i}\left(\mathbb{L}^{A|R},M\right)$ for every $i\geq 0$. These modules enable one to extend the above-mentioned sequences to the following exact sequences:
\begin{gather*}
\cdots \rightarrow H_{2}^{AQ}\left(S|R;M \right)\rightarrow H_{2}^{AQ}\left(T|R;M \right)\rightarrow H_{2}^{AQ}\left(T|S;M \right)\rightarrow H_{1}^{AQ}\left(S|R;M \right)\rightarrow \\
H_{1}^{AQ}\left(T|R;M \right) \rightarrow
H_{1}^{AQ}\left(T|S;M \right)\rightarrow \Omega_{S|R}\otimes_{S}M \rightarrow \Omega_{T|R}\otimes_{T}M \rightarrow \Omega_{T|S}\otimes_{T}M \rightarrow 0
\end{gather*}
and
\begin{gather*}
  0 \rightarrow \Der_{S}(T,M) \rightarrow \Der_{R}(T,M) \rightarrow \Der_{R}(S,M) \rightarrow H_{AQ}^{1}\left(T|S;M\right) \rightarrow H_{AQ}^{1}\left(T|R;M\right) \rightarrow \\
H_{AQ}^{1}\left(S|R;M\right) \rightarrow H_{AQ}^{2}\left(T|S;M\right) \rightarrow H_{AQ}^{2}\left(T|R;M\right) \rightarrow H_{AQ}^{2}\left(S|R;M\right) \rightarrow \cdots
\end{gather*}

Upon further development of the theory, Quillen constructed two spectral sequences that relate Andr\'{e}-Quillen homology and cohomology to Tor and Ext functors. The five-term exact sequences resulting from these spectral sequences are as follows:
$$\Tor_{3}^{R}(A,M)\rightarrow H_{3}^{AQ}(A|R;M) \rightarrow \textstyle\bigwedge_{A}^{2}\Tor_{1}^{R}(A,A)\otimes_{A}M \rightarrow \Tor_{2}^{R}(A,M) \rightarrow H_{2}^{AQ}(A|R;M) \rightarrow 0$$
and
$$0\rightarrow H_{AQ}^{2}(A|R;M)\rightarrow \Ext_{R}^{2}(A,M)\rightarrow \Hom_{A}\left(\textstyle\bigwedge_{A}^{2}\Tor_{1}^{R}(A,A),M\right)\rightarrow H_{AQ}^{3}(A|R;M)\rightarrow \Ext_{R}^{3}(A,M)$$
These sequences are used to characterize regular local rings as well as complete intersection local rings as follows.

\begin{theorem*}
Let $(R,\mathfrak{m},k)$ be a noetherian local ring. Then the following assertions are equivalent:
\begin{enumerate}
\item[(i)] $R$ is a regular ring.
\item[(ii)] $H_{i}^{AQ}(k|R;V)=0$ for every vector space $V$ over $k$ and every $i\geq 2$.
\item[(iii)] $H_{2}^{AQ}(k|R;k)=0$.
\item[(iv)] $H_{AQ}^{i}(k|R;V)=0$ for every vector space $V$ over $k$ and every $i\geq 2$.
\item[(v)] $H_{AQ}^{2}(k|R;k)=0$.
\end{enumerate}
\end{theorem*}

\begin{proof}
See \cite[Propositions 8.12]{Iy}.
\end{proof}

\begin{theorem*}
Let $(R,\mathfrak{m},k)$ be a noetherian local ring. Then the following assertions are equivalent:
\begin{enumerate}
\item[(i)] $R$ is a complete intersection ring.
\item[(ii)] $H_{i}^{AQ}(k|R;V)=0$ for every vector space $V$ over $k$ and every $i\geq 3$.
\item[(iii)] $H_{3}^{AQ}(k|R;k)=0$.
\item[(iv)] $H_{4}^{AQ}(k|R;k)=0$.
\item[(v)] $H_{i}^{AQ}(k|R;k)=0$ for every $i\gg 0$.
\item[(vi)] $H_{AQ}^{i}(k|R;V)=0$ for every vector space $V$ over $k$ and every $i\geq 3$.
\item[(vii)] $H_{AQ}^{3}(k|R;k)=0$.
\item[(viii)] $H_{AQ}^{4}(k|R;k)=0$.
\item[(ix)] $H_{AQ}^{i}(k|R;k)=0$ for every $i\gg 0$.
\end{enumerate}
\end{theorem*}

\begin{proof}
See \cite[Propositions 8.14]{Iy}, \cite[Chapter XVII, Th\'{e}or\`{e}me 13]{An}, and \cite[Theorem 1]{Av3}.
\end{proof}

\noindent
The latter theorem is leveraged by L. L. Avramov to prove the ascent-descent property of complete intersection local rings, and then to solve their long-sought localization problem; see \cite[Theorem 2 and Corollary 1]{Av2}.

Despite their focal role in the theory of Andr\'{e}-Quillen (co)homology, Quillen's fundamental spectral sequences are not treated the way they deserve. To shed some light on this, we point out that a sketchy argument with gaps for a special case of the first spectral sequence appears in an unpublished manuscript of Quillen (see \cite[Theorem 6.8]{Qu1}), while the second spectral sequence is only stated without proof in another paper of Quillen (see \cite[Theorem 6.8]{Qu3}). In order to construct the first spectral sequence, Quillen forms a filtration from which the spectral sequence is obtained. Moreover, he proves his Connectivity Lemma to establish its convergence. However, it is not obvious how the Connectivity Lemma implies convergence by looking at the standard theorems of convergence in the literature such as \cite[Theorem 2.6]{Mc}, \cite[Theorem 10.14]{Ro}, and \cite[Theorem 5.5.1]{We}. To address this defect, we prove a general theorem that allows one to use the Connectivity Lemma to settle the convergence of the spectral sequences. Also, a spectral sequence arising from a filtration is generally described by its first page, while Quillen's spectral sequences are given by their second pages. A delicate reparametrization is required to reconcile the two which is not addressed in \cite[Theorem 6.8]{Qu1}. The proof of the Connectivity Lemma in \cite[Lemma 6.5 and Theorem 8.8]{Qu1} is not well-articulated, so we use the ideas developed in \cite[Chapters XI, XII, and XIII]{An} and assemble them to provide an easier and more readable proof. When we get to the second spectral sequence, there is no proof in the literature. One might be tripped up to think that a straightforward dual argument works, but it is not clear what filtration should be used to form the second spectral sequence. We thus include a complete treatment of the second spectral sequence. After establishing the spectral sequences, we deal with their corresponding five-term exact sequences. However, one needs a generalized version of five-term exact sequences which is only stated in \cite[Corollaries 10.32 and 10.34]{Ro} whose statement is not even correct. We include the correct version and use it effectively. In the determination of the terms of the five-term exact sequences, there is a subtlety that should be noted about the fact that the symmetric power functor does not commute with the normalization functor, so one needs some sort of homotopy preservation result which we tackle. All in all, we present a complete, neat, and readable account of Quillen's fundamental spectral sequences.

\section{Spectral Sequences Arising from Filtrations}

It is folklore that filtrations yield spectral sequences; see \cite[Theorem 2.6]{Mc}, \cite[Theorem 10.14]{Ro}, and \cite[Theorem 5.5.1]{We}. In this regard, Quillen's fundamental spectral sequences arise from certain filtrations on chain complexes. However, one needs stronger versions of these theorems in order to establish the convergence of Quillen's fundamental spectral sequences. In this section, we prove a theorem on the convergence of spectral sequences arising from filtrations which generalizes the existing theorems in the literature and fits the bill for Quillen's fundamental spectral sequences. First we recall the basics of spectral sequences and chain complexes in order to set up the language and notation.

Let $R$ be a ring. A \textit{spectral sequence} $E$ of left $R$-modules consists of the following data: A left $R$-module $E_{p,q}^{r}$ for every $p,q\in \mathbb{Z}$ and $r\geq 1$; An $R$-homomorphism $d_{p,q}^{r}:E_{p,q}^{r} \rightarrow E_{p-r,q+r-1}^{r}$ for every $p,q\in \mathbb{Z}$ and $r\geq 1$ such that each composition $E_{p+r,q-r+1}^{r} \xrightarrow{d_{p+r,q-r+1}^{r}} E_{p,q}^{r} \xrightarrow{d_{p,q}^{r}} E_{p-r,q+r-1}^{r}$ is zero; and an $R$-isomorphism $E_{p,q}^{r+1} \cong  \frac{\Ker\left(d_{p,q}^{r}\right)}{\im\left(d_{p+r,q-r+1}^{r}\right)}$ for every $p,q\in \mathbb{Z}$ and $r\geq 1$.

A spectral sequence $E$ of left $R$-modules is said to be \textit{first quadrant} if $E_{p,q}^{2}=0$ for every $p<0$ or $q<0$, and is said to be \textit{third quadrant} if $E_{p,q}^{2}=0$ for every $p>0$ or $q>0$.

Given a spectral sequence $E$ of left $R$-modules, for any $p,q\in \mathbb{Z}$ and $r\geq 2$, there is a chain
$$B_{p,q}^{2}\subseteq B_{p,q}^{3}\subseteq \cdots \subseteq B_{p,q}^{r}\subseteq Z_{p,q}^{r}\subseteq \cdots \subseteq Z_{p,q}^{3}\subseteq Z_{p,q}^{2}$$
of submodules of $E_{p,q}^{1}$ such that $E_{p,q}^{r}\cong \frac{Z_{p,q}^{r}}{B_{p,q}^{r}}$; see \cite[5.2.8]{We} or \cite[paragraph after Example 10.12]{Ro}. Set $Z_{p,q}^{\infty}= \bigcap_{r=2}^{\infty}Z_{p,q}^{r}$, $B_{p,q}^{\infty}= \bigcup_{r=2}^{\infty}B_{p,q}^{r}$, and $E_{p,q}^{\infty}=\frac{Z_{p,q}^{\infty}}{B_{p,q}^{\infty}}$ for every $p,q\in \mathbb{Z}$. The spectral sequence $E$ is said to be \textit{convergent} to a family $H=\{H_{n}\}_{n\in \mathbb{Z}}$ of left $R$-modules if for any $n\in \mathbb{Z}$, $H_{n}$ has a finite filtration
$$0= U^{\alpha}\subseteq U^{\alpha+1}\subseteq \cdots \subseteq U^{\beta}= H_{n}$$
such that $\frac{U^{p}}{U^{p-1}}\cong E_{p,n-p}^{\infty}$ for every $p\in \mathbb{Z}$. When we speak of spectral sequences, we implicitly mean convergent spectral sequences and write $E_{p,q}^{1} \Rightarrow H_{p+q}$ or $E_{p,q}^{2} \Rightarrow H_{p+q}$ accordingly.

An $R$-\textit{complex} $X$ is a $\mathbb{Z}$-indexed sequence
$$X: \cdots \rightarrow X_{i+1} \xrightarrow{\partial_{i+1}^{X}} X_{i} \xrightarrow{\partial_{i}^{X}} X_{i-1} \rightarrow \cdots$$
of left $R$-modules and $R$-homomorphisms such that $\partial_{i}^{X}\partial_{i+1}^{X}=0$ for every $i\in \mathbb{Z}$. For any $i\in \mathbb{Z}$, then $i$th \textit{homology} of $X$ is the left $R$-module $H_{i}(X)=\frac{\Ker\left(\partial_{i}^{X}\right)}{\im\left(\partial_{i+1}^{X}\right)}$.

Given an $R$-complex $X$, a \textit{filtration} of $X$ is a sequence $\{F^{p}\}_{p\in \mathbb{Z}}$ of subcomplexes of $X$ such that $F^{p-1}$ is a subcomplex of $F^p$ for every $p\in \mathbb{Z}$. Given a filtration
$$\cdots \subseteq F^{p-1} \subseteq F^{p} \subseteq \cdots \subseteq X$$
of $X$, there is an induced filtration
$$\cdots \subseteq F_{n}^{p-1} \subseteq F_{n}^{p} \subseteq \cdots \subseteq X_{n}$$
of $X_{n}$ for every $n\in \mathbb{Z}$. Moreover, if for any $p,n\in \mathbb{Z}$, we set $U_{n}^{p}= \im\left(H_{n}(\iota^{p})\right)$ where $\iota^{p}:F^{p}\rightarrow X$ is the inclusion morphism, then we get an induced filtration
$$\cdots \subseteq U_{n}^{p-1} \subseteq U_{n}^{p} \subseteq \cdots \subseteq H_{n}(X)$$
of $H_{n}(X)$ for every $n\in \mathbb{Z}$; see \cite[paragraph after Lemma 10.13]{Ro}.

We need a general theorem on the convergence of the spectral sequence arising from a filtration that is strong enough to derive Quillen's fundamental spectral sequences. We first need a lemma.

\begin{lemma} \label{02.1}
Let $R$ be a ring, $M$ a left $R$-module, and $N$, $K$, and $L$ three submodules of $M$. If $L\subseteq K$, then we have:
$$\textstyle \frac{N+K}{N+L} \cong \frac{K}{(K\cap N)+L}$$
\end{lemma}

\begin{proof}
Since $L\subseteq K$, we have $K=K+L$, so we get $\textstyle \frac{N+K}{N+L} = \frac{N+K+L}{N+L} \cong \frac{K}{K\cap (N+L)}$. On the other hand, since $K\supseteq L$, we have $K\cap (N+L)=(K\cap N)+(K\cap L)=(K\cap N)+L$. As a result, we get $\textstyle \frac{N+K}{N+L} \cong \frac{K}{K\cap (N+L)} = \frac{K}{(K\cap N) +L}$.
\end{proof}

The following theorem and its corollaries generalize \cite[Theorem 2.6]{Mc}, \cite[Theorem 10.14]{Ro}, and \cite[Theorem 5.5.1]{We}. The author uses some ideas from \cite[Sections 12.23 and 12.24 on Spectral Sequences from Filtered Complexes]{St} in the sequel.

\begin{theorem} \label{02.2}
Let $R$ be a ring, and $X$ an $R$-complex equipped with a filtration
$$\cdots \subseteq F^{p-1} \subseteq F^{p} \subseteq \cdots \subseteq X$$
such that the induced filtration of $H_{n}(X)$ is finite for every $n\in \mathbb{Z}$. For any $r\geq 0$ and $p,q\in \mathbb{Z}$, set $I_{p,q}^{r}:= \left(\partial_{p+q}^{X}\right)^{-1}\left(F_{p+q-1}^{p-r}\right)\cap F_{p+q}^{p}$ and $J_{p,q}^{r}:= \partial_{p+q+1}^{X}\left(F_{p+q+1}^{p+r-1}\right)\cap F_{p+q}^{p}$. If
$$\left(\textstyle \bigcap_{r=0}^{\infty} I_{p,q}^{r}\right)+F_{p+q}^{p-1}= \left(\Ker\left(\partial_{p+q}^{X}\right)\cap F_{p+q}^{p}\right)+F_{p+q}^{p-1}$$
and
$$\left(\textstyle \bigcup_{r=0}^{\infty} J_{p,q}^{r}\right)+F_{p+q}^{p-1}= \left(\im\left(\partial_{p+q+1}^{X}\right)\cap F_{p+q}^{p}\right)+F_{p+q}^{p-1}$$
for every $p,q\in \mathbb{Z}$, then there exists a spectral sequence as follows:
$$\textstyle E_{p,q}^{1}=H_{p+q}\left(\frac{F^{p}}{F^{p-1}}\right) \Rightarrow H_{p+q}(X)$$
\end{theorem}

\begin{proof}
First, we construct the spectral sequence. Let $r\geq 0$ and $p,q\in \mathbb{Z}$. Since $F_{p+q-1}^{p-r}\subseteq F_{p+q-1}^{p-r+1}$, it is clear that $I_{p,q}^{r} \subseteq I_{p,q}^{r-1}$. Similarly, as $F_{p+q+1}^{p+r-1}\subseteq F_{p+q+1}^{p+r}$, we see that $J_{p,q}^{r} \subseteq J_{p,q}^{r+1}$. As a consequence, we have:
$$J_{p,q}^{0}\subseteq J_{p,q}^{1}\subseteq \cdots \subseteq \im\left(\partial_{p+q+1}^{X}\right)\cap F_{p+q}^{p} \subseteq \Ker\left(\partial_{p+q}^{X}\right)\cap F_{p+q}^{p} \subseteq \cdots \subseteq I_{p,q}^{1} \subseteq I_{p,q}^{0}= F_{p+q}^{p}$$
Let $r\geq 1$ and $p,q\in \mathbb{Z}$. We show that $J_{p,q}^{r}=\partial_{p+q+1}^{F^{p+r-1}}\left(I_{p+r-1,q-r+2}^{r-1}\right)$. To see this, let $y\in J_{p,q}^{r}= \partial_{p+q+1}^{X}\left(F_{p+q+1}^{p+r-1}\right)\cap F_{p+q}^{p}$. Then there is an element $x\in F_{p+q+1}^{p+r-1}$ such that $y=\partial_{p+q+1}^{X}(x)$. But then $\partial_{p+q+1}^{X}(x)=y\in F_{p+q}^{p}$, so $x\in \left(\partial_{p+q+1}^{X}\right)^{-1}\left(F_{p+q}^{p}\right)$. Hence $x\in \left(\partial_{p+q+1}^{X}\right)^{-1}\left(F_{p+q}^{p}\right)\cap F_{p+q+1}^{p+r-1}=I_{p+r-1,q-r+2}^{r-1}$. As a result, $y=\partial_{p+q+1}^{X}(x)= \partial_{p+q+1}^{F^{p+r-1}}(x)\in \partial_{p+q+1}^{F^{p+r-1}}\left(I_{p+r-1,q-r+2}^{r-1}\right)$. Conversely, let $y\in \partial_{p+q+1}^{F^{p+r-1}}\left(I_{p+r-1,q-r+2}^{r-1}\right)$. Then there is an element $x\in I_{p+r-1,q-r+2}^{r-1}$ such that $y=\partial_{p+q+1}^{F^{p+r-1}}(x)$. But $x\in I_{p+r-1,q-r+2}^{r-1}=\left(\partial_{p+q+1}^{X}\right)^{-1}\left(F_{p+q}^{p}\right)\cap F_{p+q+1}^{p+r-1}$, so $x\in F_{p+q+1}^{p+r-1}$ and $\partial_{p+q+1}^{X}(x)\in F_{p+q}^{p}$. As a result, $y=\partial_{p+q+1}^{F^{p+r-1}}(x)=\partial_{p+q+1}^{X}(x)\in \partial_{p+q+1}^{X}\left(F_{p+q+1}^{p+r-1}\right)\cap F_{p+q}^{p}=J_{p,q}^{r}$. Therefore, $J_{p,q}^{r}= \partial_{p+q+1}^{F^{p+r-1}}\left(I_{p+r-1,q-r+2}^{r-1}\right)$.

Let $r\geq 1$ and $p,q\in \mathbb{Z}$. Define:
$$\textstyle E_{p,q}^{r}:= \frac{I_{p,q}^{r}}{J_{p,q}^{r}+I_{p-1,q+1}^{r-1}}$$
Furthermore, define a map $\phi_{p,q}^{r}:I_{p,q}^{r}\rightarrow \frac{I_{p-r,q+r-1}^{r}}{J_{p-r,q+r-1}^{r}+I_{p-r-1,q+r}^{r-1}}$ by setting $\phi_{p,q}^{r}(x)=\partial_{p+q}^{X}(x)+J_{p-r,q+r-1}^{r}+I_{p-r-1,q+r}^{r-1}$ for every $x\in I_{p,q}^{r}$. For well-definedness, we note that if $x\in I_{p,q}^{r}=\left(\partial_{p+q}^{X}\right)^{-1}\left(F_{p+q-1}^{p-r}\right)\cap F_{p+q}^{p}$, then $\partial_{p+q}^{X}(x)\in F_{p+q-1}^{p-r}$, so since $\partial_{p+q-1}^{X}\left(\partial_{p+q}^{X}(x)\right)=0\in F_{p+q-2}^{p-2r}$, we get $\partial_{p+q}^{X}(x)\in \left(\partial_{p+q-1}^{X}\right)^{-1}\left(F_{p+q-2}^{p-2r}\right)\cap F_{p+q-1}^{p-r}= I_{p-r,q+r-1}^{r}$. Hence $\phi_{p,q}^{r}$ is well-defined. Moreover, it is clear that $\phi_{p,q}^{r}$ is an $R$-homomorphism. If $x\in J_{p,q}^{r}+I_{p-1,q+1}^{r-1}$, then we can write $x=u+v$ for some $u\in J_{p,q}^{r}$ and $v\in I_{p-1,q+1}^{r-1}$. Since $u\in J_{p,q}^{r}=\partial_{p+q+1}^{X}\left(F_{p+q+1}^{p+r-1}\right)\cap F_{p+q}^{p}$, we have $u=\partial_{p+q+1}^{X}(w)$ for some $w\in F_{p+q+1}^{p+r-1}$, so $\partial_{p+q}^{X}(u)=\partial_{p+q}^{X}\left(\partial_{p+q+1}^{X}(w)\right)=0$. As $v\in I_{p-1,q+1}^{r-1}= \left(\partial_{p+q}^{X}\right)^{-1}\left(F_{p+q-1}^{p-r}\right)\cap F_{p+q}^{p-1}$, we have $\partial_{p+q}^{X}(v)=\partial_{p+q}^{F^{p-1}}(v)\in \partial_{p+q}^{F^{p-1}}\left(I_{p-1,q+1}^{r-1}\right)=J_{p-r,q+r-1}^{r}$. Consequently, $\partial_{p+q}^{X}(x)=\partial_{p+q}^{X}(u)+\partial_{p+q}^{X}(v)=\partial_{p+q}^{X}(v)\in J_{p-r,q+r-1}^{r}$, so $\phi_{p,q}^{r}(x)=0$. This shows that $J_{p,q}^{r}+I_{p-1,q+1}^{r-1}\subseteq \Ker\left(\phi_{p,q}^{r}\right)$, so $\phi_{p,q}^{r}$ induces an $R$-homomorphism
$$\textstyle d_{p,q}^{r}:= \overline{\phi_{p,q}^{r}}: \frac{I_{p,q}^{r}}{J_{p,q}^{r}+I_{p-1,q+1}^{r-1}}\rightarrow \frac{I_{p-r,q+r-1}^{r}}{J_{p-r,q+r-1}^{r}+I_{p-r-1,q+r}^{r-1}}$$
given by
$$d_{p,q}^{r}\left(x+J_{p,q}^{r}+I_{p-1,q+1}^{r-1}\right)=\partial_{p+q}^{X}(x)+J_{p-r,q+r-1}^{r}+I_{p-r-1,q+r}^{r-1}$$
for every $x\in I_{p,q}^{r}$. Since $\partial_{p+q}^{X}\partial_{p+q+1}^{X}=0$, we see that $d_{p,q}^{r}d_{p+r,q-r+1}^{r}=0$.

Let $r\geq 1$ and $p,q\in \mathbb{Z}$. Considering $d_{p,q}^{r}:\frac{I_{p,q}^{r}}{J_{p,q}^{r}+I_{p-1,q+1}^{r-1}}\rightarrow \frac{I_{p-r,q+r-1}^{r}}{J_{p-r,q+r-1}^{r}+I_{p-r-1,q+r}^{r-1}}$, we show that:
$$\textstyle \Ker\left(d_{p,q}^{r}\right)=\frac{I_{p,q}^{r+1}+I_{p-1,q+1}^{r-1}}{J_{p,q}^{r}+I_{p-1,q+1}^{r-1}}$$
To see this, let $x+J_{p,q}^{r}+I_{p-1,q+1}^{r-1} \in \Ker\left(d_{p,q}^{r}\right)$. Then $x\in I_{p,q}^{r}$ with $\partial_{p+q}^{X}(x)\in J_{p-r,q+r-1}^{r}+I_{p-r-1,q+r}^{r-1}$, so we can write $\partial_{p+q}^{X}(x)=u+v$ for some $u\in J_{p-r,q+r-1}^{r}$ and $v\in I_{p-r-1,q+r}^{r-1}$. As $u\in J_{p-r,q+r-1}^{r}=\partial_{p+q}^{F^{p-1}}\left(I_{p-1,q+1}^{r-1}\right)$, we have $u=\partial_{p+q}^{F^{p-1}}(y)=\partial_{p+q}^{X}(y)$ for some $y\in I_{p-1,q+1}^{r-1}$. Now $x\in I_{p,q}^{r}\subseteq F_{p+q}^{p}$ and $y\in I_{p-1,q+1}^{r-1}\subseteq F_{p+q}^{p-1}\subseteq F_{p+q}^{p}$, so $x-y\in F_{p+q}^{p}$. Also, $\partial_{p+q}^{X}(x-y)=\partial_{p+q}^{X}(x)-u=v \in I_{p-r-1,q+r}^{r-1}\subseteq F_{p+q-1}^{p-r-1}$, so $x-y\in \left(\partial_{p+q}^{X}\right)^{-1}\left(F_{p+q-1}^{p-r-1}\right)$. Therefore, $x-y\in \left(\partial_{p+q}^{X}\right)^{-1}\left(F_{p+q-1}^{p-r-1}\right)\cap F_{p+q}^{p}= I_{p,q}^{r+1}$, so $x-y=z$ for some $z\in I_{p,q}^{r+1}$. Thus $x=z+y\in I_{p,q}^{r+1}+I_{p-1,q+1}^{r-1}$, so $x+J_{p,q}^{r}+I_{p-1,q+1}^{r-1}\in \frac{I_{p,q}^{r+1}+I_{p-1,q+1}^{r-1}}{J_{p,q}^{r}+I_{p-1,q+1}^{r-1}}$. Conversely, let $x\in I_{p,q}^{r+1}=\left(\partial_{p+q}^{X}\right)^{-1}\left(F_{p+q-1}^{p-r-1}\right)\cap F_{p+q}^{p}$. Then $\partial_{p+q}^{X}(x)\in F_{p+q-1}^{p-r-1}$, so since $\partial_{p+q-1}^{X}\left(\partial_{p+q}^{X}(x)\right)=0\in F_{p+q-2}^{p-2r}$, we get $\partial_{p+q}^{X}(x)\in \left(\partial_{p+q-1}^{X}\right)^{-1}\left(F_{p+q-2}^{p-2r}\right)\cap F_{p+q-1}^{p-r-1}=I_{p-r-1,q+r}^{r-1}$. As a result, $d_{p,q}^{r}\left(x+J_{p,q}^{r}+I_{p-1,q+1}^{r-1}\right)= \partial_{p+q}^{X}(x)+J_{p-r,q+r-1}^{r}+I_{p-r-1,q+r}^{r-1}=0$, so $x+J_{p,q}^{r}+I_{p-1,q+1}^{r-1}\in \Ker\left(d_{p,q}^{r}\right)$. Therefore, the claim is proved.

Similarly, considering $d_{p+r,q-r+1}^{r}:\frac{I_{p+r,q-r+1}^{r}}{J_{p+r,q-r+1}^{r}+I_{p+r-1,q-r+2}^{r-1}}\rightarrow \frac{I_{p,q}^{r}}{J_{p,q}^{r}+I_{p-1,q+1}^{r-1}}$,  we show that:
$$\textstyle \im\left(d_{p+r,q-r+1}^{r}\right)=\frac{J_{p,q}^{r+1}+I_{p-1,q+1}^{r-1}}{J_{p,q}^{r}+I_{p-1,q+1}^{r-1}}$$
To see this, let $y\in I_{p,q}^{r}$ be such that $y+J_{p,q}^{r}+I_{p-1,q+1}^{r-1}\in \im\left(d_{p+r,q-r+1}^{r}\right)$. Then there is an element $x\in I_{p+r,q-r+1}^{r}\subseteq F_{p+q+1}^{p+r}$ such that:
$$y+J_{p,q}^{r}+I_{p-1,q+1}^{r-1}= d_{p+r,q-r+1}^{r}\left(x+J_{p+r,q-r+1}^{r}+I_{p+r-1,q-r+2}^{r-1}\right)= \partial_{p+q+1}^{X}(x)+J_{p,q}^{r}+I_{p-1,q+1}^{r-1}$$
That is, $y-\partial_{p+q+1}^{X}(x)\in J_{p,q}^{r}+I_{p-1,q+1}^{r-1}$, so we can write $y-\partial_{p+q+1}^{X}(x)=u+v$ for some $u\in J_{p,q}^{r}\subseteq J_{p,q}^{r+1}$ and $v\in I_{p-1,q+1}^{r-1}$. But $x\in I_{p+r,q-r+1}^{r}$, so $\partial_{p+q+1}^{X}(x)=\partial_{p+q+1}^{F^{p+r}}(x)\in \partial_{p+q+1}^{F^{p+r}}\left(I_{p+r,q-r+1}^{r}\right)=J_{p,q}^{r+1}$. As a result, $\partial_{p+q+1}^{X}(x)+u\in J_{p,q}^{r+1}$, so $y=\partial_{p+q+1}^{X}(x)+u+v\in J_{p,q}^{r+1}+I_{p-1,q+1}^{r-1}$. Conversely, let $y\in J_{p,q}^{r+1}=\partial_{p+q+1}^{F^{p+r}}\left(I_{p+r,q-r+1}^{r}\right)$. Then $y=\partial_{p+q+1}^{F^{p+r}}(x)=\partial_{p+q+1}^{X}(x)$ for some $x\in I_{p+r,q-r+1}^{r}$. Thus we have:
$$y+J_{p,q}^{r}+I_{p-1,q+1}^{r-1}=\partial_{p+q+1}^{X}(x)+J_{p,q}^{r}+I_{p-1,q+1}^{r-1}=d_{p+r,q-r+1}^{r}\left(x+J_{p+r,q-r+1}^{r}+I_{p+r-1,q-r+2}^{r-1}\right)$$
Therefore, the claim is proved.

We now note that since $F_{p+q}^{p-1}\subseteq F_{p+q}^{p}$, we have:
$$I_{p,q}^{r}\cap F_{p+q}^{p-1}= \left(\partial_{p+q}^{X}\right)^{-1}\left(F_{p+q-1}^{p-r}\right)\cap F_{p+q}^{p}\cap F_{p+q}^{p-1}= \left(\partial_{p+q}^{X}\right)^{-1}\left(F_{p+q-1}^{p-r}\right)\cap F_{p+q}^{p-1}=I_{p-1,q+1}^{r-1}$$
Using this and the fact that $I_{p,q}^{r+1}\subseteq I_{p,q}^{r}$, we see that:
$$I_{p,q}^{r+1}\cap I_{p-1,q+1}^{r-1}=I_{p,q}^{r+1}\cap I_{p,q}^{r}\cap F_{p+q}^{p-1}= I_{p,q}^{r+1}\cap F_{p+q}^{p-1}= I_{p-1,q+1}^{r}$$
Thus in view of the fact that $J_{p,q}^{r+1}\subseteq I_{p,q}^{r+1}$, we can use Lemma \ref{02.1} to get the following isomorphisms:
$$\textstyle E_{p,q}^{r+1}= \frac{I_{p,q}^{r+1}}{J_{p,q}^{r+1}+I_{p-1,q+1}^{r}}= \frac{I_{p,q}^{r+1}}{J_{p,q}^{r+1}+\left(I_{p,q}^{r+1}\cap I_{p-1,q+1}^{r-1}\right)}\cong \frac{I_{p,q}^{r+1}+I_{p-1,q+1}^{r-1}}{J_{p,q}^{r+1}+I_{p-1,q+1}^{r-1}}\cong \frac{\frac{I_{p,q}^{r+1}+I_{p-1,q+1}^{r-1}}{J_{p,q}^{r}+I_{p-1,q+1}^{r-1}}}{\frac{J_{p,q}^{r+1}+I_{p-1,q+1}^{r-1}}{J_{p,q}^{r}+I_{p-1,q+1}^{r-1}}}= \frac{\Ker\left(d_{p,q}^{r}\right)}{\im\left(d_{p+r,q-r+1}^{r}\right)}$$

Let $p,q\in \mathbb{Z}$. We note that $I_{p,q}^{1}=\left(\partial_{p+q}^{X}\right)^{-1}\left(F_{p+q-1}^{p-1}\right)\cap F_{p+q}^{p}= \left(\partial_{p+q}^{F^p}\right)^{-1}\left(F_{p+q-1}^{p-1}\right)$ and $J_{p,q}^{1}=\partial_{p+q+1}^{F^p}\left(I_{p,q+1}^{0}\right)= \partial_{p+q+1}^{F^p}\left(F_{p+q+1}^{p}\right)$. Consider the following $R$-homomorphisms:
$$\textstyle \frac{F_{p+q+1}^{p}}{F_{p+q+1}^{p-1}} \xrightarrow{\partial_{p+q+1}^{\frac{F^{p}}{F^{p-1}}}} \frac{F_{p+q}^{p}}{F_{p+q}^{p-1}} \xrightarrow{\partial_{p+q}^{\frac{F^{p}}{F^{p-1}}}} \frac{F_{p+q-1}^{p}}{F_{p+q-1}^{p-1}}$$
Then using Lemma \ref{02.1}, we see that:
$$\textstyle H_{p+q}\left(\frac{F^{p}}{F^{p-1}}\right)= \frac{\Ker\left(\partial_{p+q}^{\frac{F^{p}}{F^{p-1}}}\right)}{\im\left(\partial_{p+q+1}^{\frac{F^{p}}{F^{p-1}}}\right)}= \frac{\frac{\left(\partial_{p+q}^{F^p}\right)^{-1}\left(F_{p+q-1}^{p-1}\right)+F_{p+q}^{p-1}}{F_{p+q}^{p-1}}} {\frac{\partial_{p+q+1}^{F^{p}}\left(F_{p+q+1}^{p}\right)+F_{p+q}^{p-1}}{F_{p+q}^{p-1}}}\cong \frac{I_{p,q}^{1}+F_{p+q}^{p-1}}{J_{p,q}^{1}+F_{p+q}^{p-1}}\cong \frac{I_{p,q}^{1}}{J_{p,q}^{1}+\left(I_{p,q}^{1}\cap F_{p+q}^{p-1}\right)}= \frac{I_{p,q}^{1}}{J_{p,q}^{1}+I_{p-1,q+1}^{0}}= E_{p,q}^{1}$$
Therefore, we obtain a spectral sequence with $E_{p,q}^{1}\cong H_{p+q}\left(\frac{F^{p}}{F^{p-1}}\right)$ for every $p,q\in \mathbb{Z}$.

Next, we establish the convergence of the spectral sequence. Let $p,q\in \mathbb{Z}$. By definition, we have $Z_{p,q}^{2}= \Ker\left(d_{p,q}^{1}\right)= \frac{I_{p,q}^{2}+I_{p-1,q+1}^{0}}{J_{p,q}^{1}+I_{p-1,q+1}^{0}}$ and $B_{p,q}^{2}=\im\left(d_{p+1,q}^{1}\right) =\frac{J_{p,q}^{2}+I_{p-1,q+1}^{0}}{J_{p,q}^1+I_{p-1,q+1}^{0}}$ as submodules of $E_{p,q}^{1}=\frac{I_{p,q}^{1}}{J_{p,q}^1+I_{p-1,q+1}^{0}}$. As we observed above, we have:
$$\textstyle E_{p,q}^{2}= \frac{I_{p,q}^{2}}{J_{p,q}^{2}+I_{p-1,q+1}^{1}}= \frac{I_{p,q}^{2}}{J_{p,q}^{2}+\left(I_{p,q}^{2}\cap I_{p-1,q+1}^{0}\right)}\cong \frac{I_{p,q}^{2}+I_{p-1,q+1}^{0}}{J_{p,q}^{2}+I_{p-1,q+1}^{0}}\cong \frac{\frac{I_{p,q}^{2}+I_{p-1,q+1}^{0}}{J_{p,q}^{1}+I_{p-1,q+1}^{0}}}{\frac{J_{p,q}^{2}+I_{p-1,q+1}^{0}}{J_{p,q}^{1}+I_{p-1,q+1}^{0}}}= \frac{\Ker\left(d_{p,q}^{1}\right)}{\im\left(d_{p+1,q}^{1}\right)}= \frac{Z_{p,q}^{2}}{B_{p,q}^{2}}$$
We note that since $I_{p,q}^{3}\cap F_{p+q}^{p-1}= I_{p-1,q+1}^{2}\subseteq I_{p-1,q+1}^{1}\subseteq I_{p-1,q+1}^{0}= F_{p+q}^{p-1}$, we have:
\begin{equation*}
\begin{split}
\left(I_{p,q}^{3}+I_{p-1,q+1}^{1}\right)\cap I_{p-1,q+1}^{0} & = \left(I_{p,q}^{3}\cap I_{p-1,q+1}^{0}\right)+ \left(I_{p-1,q+1}^{1}\cap I_{p-1,q+1}^{0}\right) = \left(I_{p,q}^{3}\cap F_{p+q}^{p-1}\right)+ I_{p-1,q+1}^{1} \\
 & = I_{p-1,q+1}^{2}+I_{p-1,q+1}^{1} = I_{p-1,q+1}^{1}
\end{split}
\end{equation*}
Now the submodule $\Ker\left(d_{p,q}^{2}\right)$ of $E_{p,q}^{2}$ is mapped under the above isomorphism as follows:
\begin{equation*}
\begin{split}
 \Ker\left(d_{p,q}^{2}\right) & = \textstyle \frac{I_{p,q}^{3}+I_{p-1,q+1}^{1}}{J_{p,q}^{2}+I_{p-1,q+1}^{1}} = \textstyle \frac{I_{p,q}^{3}+I_{p-1,q+1}^{1}}{J_{p,q}^{2}+\left(\left(I_{p,q}^{3}+I_{p-1,q+1}^{1}\right)\cap I_{p-1,q+1}^{0}\right)} \cong \textstyle \frac{I_{p,q}^{3}+I_{p-1,q+1}^{1}+I_{p-1,q+1}^{0}}{J_{p,q}^{2}+I_{p-1,q+1}^{0}} \cong \textstyle \frac{I_{p,q}^{3}+I_{p-1,q+1}^{0}}{J_{p,q}^{2}+I_{p-1,q+1}^{0}} \\
 & \cong \textstyle \frac{\frac{I_{p,q}^{3}+I_{p-1,q+1}^{0}}{J_{p,q}^{1}+I_{p-1,q+1}^{0}}}{\frac{J_{p,q}^{2}+I_{p-1,q+1}^{0}}{J_{p,q}^{1}+I_{p-1,q+1}^{0}}} = \textstyle \frac{Z_{p,q}^{3}}{B_{p,q}^{2}}
\end{split}
\end{equation*}
Hence $Z_{p,q}^{3}=\frac{I_{p,q}^{3}+I_{p-1,q+1}^{0}}{J_{p,q}^{1}+I_{p-1,q+1}^{0}}$. Similarly, since $J_{p,q}^{3}\cap F_{p+q}^{p-1}\subseteq I_{p,q}^{2}\cap F_{p+q}^{p-1}=I_{p-1,q+1}^{1}\subseteq I_{p-1,q+1}^{0}= F_{p+q}^{p-1}$, we have:
\begin{equation*}
\begin{split}
\left(J_{p,q}^{3}+I_{p-1,q+1}^{1}\right)\cap I_{p-1,q+1}^{0} & = \left(J_{p,q}^{3}\cap I_{p-1,q+1}^{0}\right)+ \left(I_{p-1,q+1}^{1}\cap I_{p-1,q+1}^{0}\right) = \left(J_{p,q}^{3}\cap F_{p+q}^{p-1}\right)+ I_{p-1,q+1}^{1} \\
 & = I_{p-1,q+1}^{1}
\end{split}
\end{equation*}
As a result, the submodule $\im\left(d_{p+2,q-1}^{2}\right)$ of $E_{p,q}^{2}$ is mapped under the above isomorphism as follows:
\begin{equation*}
\begin{split}
 \textstyle \im\left(d_{p+2,q-1}^{2}\right) & = \textstyle \frac{J_{p,q}^{3}+I_{p-1,q+1}^{1}}{J_{p,q}^{2}+I_{p-1,q+1}^{1}} = \frac{J_{p,q}^{3}+I_{p-1,q+1}^{1}}{J_{p,q}^{2}+\left(\left(J_{p,q}^{3}+I_{p-1,q+1}^{1}\right)\cap I_{p-1,q+1}^{0}\right)} \cong \frac{J_{p,q}^{3}+I_{p-1,q+1}^{1}+I_{p-1,q+1}^{0}}{J_{p,q}^{2}+I_{p-1,q+1}^{0}} \cong \frac{J_{p,q}^{3}+I_{p-1,q+1}^{0}}{J_{p,q}^{2}+I_{p-1,q+1}^{0}} \\
 & \textstyle \cong \frac{\frac{J_{p,q}^{3}+I_{p-1,q+1}^{0}}{J_{p,q}^{1}+I_{p-1,q+1}^{0}}}{\frac{J_{p,q}^{2}+I_{p-1,q+1}^{0}}{J_{p,q}^{1}+I_{p-1,q+1}^{0}}} = \frac{B_{p,q}^{3}}{B_{p,q}^{2}}
\end{split}
\end{equation*}
Hence $B_{p,q}^{3}=\frac{J_{p,q}^{3}+I_{p-1,q+1}^{0}}{J_{p,q}^{1}+I_{p-1,q+1}^{0}}$. Continuing in this way, we see that:
\[
 Z_{p,q}^{r}= \textstyle \frac{I_{p,q}^{r}+I_{p-1,q+1}^{0}}{J_{p,q}^{1}+I_{p-1,q+1}^{0}}=\frac{I_{p,q}^{r}+F_{p+q}^{p-1}}{J_{p,q}^{1}+F_{p+q}^{p-1}}
\quad \textrm{and} \quad
 B_{p,q}^{r}= \textstyle \frac{J_{p,q}^{r}+I_{p-1,q+1}^{0}}{J_{p,q}^{1}+I_{p-1,q+1}^{0}}= \frac{J_{p,q}^{r}+F_{p+q}^{p-1}}{J_{p,q}^{1}+F_{p+q}^{p-1}}
\]
On the other hand, using the hypothesis, we have:
$$\textstyle Z_{p,q}^{\infty}= \bigcap_{r=2}^{\infty}Z_{p,q}^{r}= \bigcap_{r=2}^{\infty}\frac{I_{p,q}^{r}+F_{p+q}^{p-1}}{J_{p,q}^{1}+F_{p+q}^{p-1}}= \frac{\left(\bigcap_{r=2}^{\infty}I_{p,q}^{r}\right)+F_{p+q}^{p-1}}{J_{p,q}^{1}+F_{p+q}^{p-1}}= \frac{\left(\Ker\left(\partial_{p+q}^{X}\right)\cap F_{p+q}^{p}\right)+F_{p+q}^{p-1}}{J_{p,q}^{1}+F_{p+q}^{p-1}}$$
and
$$\textstyle B_{p,q}^{\infty}= \bigcup_{r=2}^{\infty}B_{p,q}^{r}= \bigcup_{r=2}^{\infty}\frac{J_{p,q}^{r}+F_{p+q}^{p-1}}{J_{p,q}^{1}+F_{p+q}^{p-1}}= \frac{\left(\bigcup_{r=2}^{\infty}J_{p,q}^{r}\right)+F_{p+q}^{p-1}}{J_{p,q}^{1}+F_{p+q}^{p-1}}= \frac{\left(\im\left(\partial_{p+q+1}^{X}\right)\cap F_{p+q}^{p}\right)+F_{p+q}^{p-1}}{J_{p,q}^{1}+F_{p+q}^{p-1}}$$
Therefore, using Lemma \ref{02.1}, we obtain:
$$\textstyle E_{p,q}^{\infty} = \frac{Z_{p,q}^{\infty}}{B_{p,q}^{\infty}} \cong \frac{\left(\Ker\left(\partial_{p+q}^{X}\right)\cap F_{p+q}^{p}\right)+F_{p+q}^{p-1}}{\left(\im\left(\partial_{p+q+1}^{X}\right)\cap F_{p+q}^{p}\right)+F_{p+q}^{p-1}} \cong \frac{\Ker\left(\partial_{p+q}^{X}\right)\cap F_{p+q}^{p}}{\left(\Ker\left(\partial_{p+q}^{X}\right)\cap F_{p+q}^{p-1}\right)+ \left(\im\left(\partial_{p+q+1}^{X}\right)\cap F_{p+q}^{p}\right)}$$
Let $p,n\in \mathbb{Z}$. Let $\iota^{p}:F^{p}\rightarrow X$ be the inclusion morphism, and consider the induced $R$-homomorphism $H_{n}(\iota^{p}):H_{n}(F^{p})\rightarrow H_{n} (X)$. Then it is clear that:
$$\textstyle U_{n}^{p}:= \im\left(H_{n}(\iota^{p})\right)= \frac{\Ker\left(\partial_{n}^{F^{p}}\right)+\im\left(\partial_{n+1}^{X}\right)}{\im\left(\partial_{n+1}^{X}\right)}$$
By the hypothesis, there is a finite filtration $0=U_{n}^{\alpha}\subseteq U_{n}^{\alpha+1}\subseteq \cdots \subseteq U_{n}^{\beta}= H_{n}(X)$ of $H_{n}(X)$ where $\alpha$ and $\beta$ depend on $n$. Moreover, $\Ker\left(\partial_{n}^{F^{p-1}}\right)\subseteq \Ker\left(\partial_{n}^{F^p}\right)$, so using Lemma \ref{02.1}, we have:
\begin{equation*}
\begin{split}
 \textstyle \frac{U_{n}^{p}}{U_{n}^{p-1}} & \cong \textstyle \frac{\Ker\left(\partial_{n}^{F^{p}}\right)+\im\left(\partial_{n+1}^{X}\right)}{\Ker\left(\partial_{n}^{F^{p-1}}\right)+\im\left(\partial_{n+1}^{X}\right)}
 \cong \frac{\Ker\left(\partial_{n}^{F^{p}}\right)}{\Ker\left(\partial_{n}^{F^{p-1}}\right)+\left(\im\left(\partial_{n+1}^{X}\right)\cap \Ker\left(\partial_{n}^{F^{p}}\right)\right)} \textstyle \cong \frac{\Ker\left(\partial_{n}^{X}\right)\cap F_{n}^{p}}{\left(\Ker\left(\partial_{n}^{X}\right)\cap F_{n}^{p-1}\right)+\left(\im\left(\partial_{n+1}^{X}\right)\cap F_{n}^{p}\right)} \cong E_{p,n-p}^{\infty}
\end{split}
\end{equation*}
This means that the spectral sequence is convergent to $H_{n}(X)$.
\end{proof}

\begin{corollary} \label{02.3}
Let $R$ be a ring, $X$ an $R$-complex equipped with a filtration
$$\cdots \subseteq F^{p-1} \subseteq F^{p} \subseteq \cdots \subseteq X,$$
and $\iota^{p}:F^{p}\rightarrow X$ the inclusion morphism for every $p\in \mathbb{Z}$. If for any $n\in \mathbb{Z}$, there are integers $\alpha(n)$ and $\beta(n)$ such that $H_n(F^{p})=0$ for every $p\leq \alpha(n)$, and $H_{n}(\iota^{p}):H_{n}(F^{p})\rightarrow H_{n}(X)$ is an isomorphism for every $p\geq \beta(n)$, then there exists a spectral sequence as follows:
$$\textstyle E_{p,q}^{1}=H_{p+q}\left(\frac{F^{p}}{F^{p-1}}\right) \Rightarrow H_{p+q}(X)$$
\end{corollary}

\begin{proof}
Set $U_{n}^{p}=\im\left(H_{n}(\iota^{p})\right)$ for every $p,n\in \mathbb{Z}$. Let $n\in \mathbb{Z}$. The hypothesis implies that $U_{n}^{p}=0$ for every $p\leq \alpha(n)$, and $U_{n}^{p}=H_{n}(X)$ for every $p\geq \beta(n)$. Thus there is a finite filtration of $H_{n}(X)$ as follows:
$$0=U_{n}^{\alpha(n)} \subseteq U_{n}^{\alpha(n)+1} \subseteq \cdots \subseteq U_{n}^{\beta(n)} = H_{n}(X)$$
For any $r\geq 0$ and $p,q\in \mathbb{Z}$, define $I_{p,q}^{r}$ and $J_{p,q}^{r}$ as in Theorem \ref{02.2}. It is clear that $\left(\textstyle \bigcap_{r=0}^{\infty} I_{p,q}^{r}\right)+F_{p+q}^{p-1}\supseteq \left(\Ker\left(\partial_{p+q}^{X}\right)\cap F_{p+q}^{p}\right)+F_{p+q}^{p-1}$ and $\left(\textstyle \bigcup_{r=0}^{\infty} J_{p,q}^{r}\right)+F_{p+q}^{p-1}\subseteq \left(\im\left(\partial_{p+q+1}^{X}\right)\cap F_{p+q}^{p}\right)+F_{p+q}^{p-1}$. We establish the reverse inequalities. Fix $p,q\in \mathbb{Z}$ with $p+q=n$.

Let $x\in \bigcap_{r=0}^{\infty} I_{p,q}^{r}$. Then $x\in F_{n}^{p}$ and $\partial_{n}^{X}(x)\in F_{n-1}^{p-r}$ for every $r\geq 0$. There is an $r_{0}\geq 1$ such that $p-r_{0}\leq \alpha(n-1)$. Then the hypothesis implies that $H_{n-1}\left(F^{p-r_{0}}\right)=0$, so $\Ker\left(\partial_{n-1}^{F^{p-r_{0}}}\right)= \im\left(\partial_{n}^{F^{p-r_{0}}}\right)$. Consider the following $R$-homomorphisms:
$$F_{n}^{p-r_{0}} \xrightarrow{\partial_{n}^{F^{p-r_{0}}}} F_{n-1}^{p-r_{0}} \xrightarrow{\partial_{n-1}^{F^{p-r_{0}}}} F_{n-2}^{p-r_{0}}$$
We have $\partial_{n}^{X}(x)\in F_{n-1}^{p-r_{0}}$ and $\partial_{n-1}^{F^{p-r_{0}}}\left(\partial_{n}^{X}(x)\right)= \partial_{n-1}^{X}\left(\partial_{n}^{X}(x)\right)=0$, so $\partial_{n}^{X}(x)\in \Ker\left(\partial_{n-1}^{F^{p-r_{0}}}\right)= \im\left(\partial_{n}^{F^{p-r_{0}}}\right)$. Thus there is an element $y\in F_{n}^{p-r_{0}}$ such that $\partial_{n}^{X}(x)= \partial_{n}^{F^{p-r_{0}}}(y)= \partial_{n}^{X}(y)$. But then $x-y\in \Ker\left(\partial_{n}^{X}\right)$, so $x-y=z$ for some $z\in \Ker\left(\partial_{n}^{X}\right)$. Now, $x\in F_{n}^{p}$ and $y\in F_{n}^{p-r_{0}}\subseteq F_{n}^{p}$, so $z=x-y\in F_{n}^{p}$, whence $z\in \Ker\left(\partial_{n}^{X}\right)\cap F_{n}^{p}$. On the other hand, $y\in F_{n}^{p-r_{0}}\subseteq F_{n}^{p-1}$. Consequently, $x=z+y\in \left(\Ker\left(\partial_{n}^{X}\right)\cap F_{n}^{p}\right)+ F_{n}^{p-1}$. It follows that $\left(\textstyle \bigcap_{r=0}^{\infty} I_{p,q}^{r}\right)+F_{n}^{p-1}= \left(\Ker\left(\partial_{n}^{X}\right)\cap F_{n}^{p}\right)+F_{n}^{p-1}$.

Let $y\in \im\left(\partial_{n+1}^{X}\right)\cap F_{n}^{p}$. Then $y\in F_{n}^{p}$ and $y=\partial_{n+1}^{X}(x)$ for some $x\in X_{n+1}$. There is an $r_{0}\geq 1$ such that $p+r_{0}-1\geq \beta(n)$. Then the hypothesis implies that $H_{n}\left(\iota^{p+r_{0}-1}\right):H_{n}\left(F^{p+r_{0}-1}\right)\rightarrow H_{n}(X)$ is an isomorphism. We have $y\in F_{n}^{p}\subseteq F_{n}^{p+r_{0}-1}$ and $\partial_{n}^{F^{p+r_{0}-1}}(y)= \partial_{n}^{X}\left(\partial_{n+1}^{X}(x)\right)= 0$, so $y\in \Ker\left(\partial_{n}^{F^{p+r_{0}-1}}\right)$. But then $H_{n}\left(\iota^{p+r_{0}-1}\right)\left(y+\im\left(\partial_{n+1}^{F^{p+r_{0}-1}}\right)\right)= y+\im\left(\partial_{n+1}^{X}\right)=0$ as $y\in \im\left(\partial_{n+1}^{X}\right)$. Thus $y\in \im\left(\partial_{n+1}^{F^{p+r_{0}-1}}\right)$, so $y\in \partial_{n+1}^{X}\left(F_{n+1}^{p+r_{0}-1}\right)\cap F_{n}^{p} = J_{p,q}^{r_{0}} \subseteq \bigcup_{r=0}^{\infty}J_{p,q}^{r}$. It follows that $\left(\textstyle \bigcup_{r=0}^{\infty} J_{p,q}^{r}\right)+F_{n}^{p-1}= \left(\im\left(\partial_{n+1}^{X}\right)\cap F_{n}^{p}\right)+F_{n}^{p-1}$.

Therefore, the result follows from Theorem \ref{02.2}.
\end{proof}

\begin{corollary} \label{02.4}
Let $R$ be a ring, and $X$ an $R$-complex equipped with a filtration
$$\cdots \subseteq F^{p-1} \subseteq F^{p} \subseteq \cdots \subseteq X$$
such that the induced filtration of $H_{n}(X)$ is finite for every $n\in \mathbb{Z}$. If $\bigcap_{p\in \mathbb{Z}}F_{n}^{p}=0$ and $\bigcup_{p\in \mathbb{Z}}F_{n}^{p}=X_{n}$ for every $n\in \mathbb{Z}$, then there exists a spectral sequence as follows:
$$\textstyle E_{p,q}^{1}=H_{p+q}\left(\frac{F^{p}}{F^{p-1}}\right) \Rightarrow H_{p+q}(X)$$
\end{corollary}

\begin{proof}
For any $r\geq 0$ and $p,q\in \mathbb{Z}$, define $I_{p,q}^{r}$ and $J_{p,q}^{r}$ as in Theorem \ref{02.2}. Using the hypothesis, we have:
$$\textstyle \bigcap_{r=0}^{\infty}I_{p,q}^{r} = \bigcap_{r=0}^{\infty}\left(\left(\partial_{p+q}^{X}\right)^{-1}\left(F_{p+q-1}^{p-r}\right)\cap F_{p+q}^{p}\right)
 = \left(\partial_{p+q}^{X}\right)^{-1}\left(\bigcap_{r=0}^{\infty}F_{p+q-1}^{p-r}\right)\cap F_{p+q}^{p} = \Ker\left(\partial_{p+q}^{X}\right)\cap F_{p+q}^{p}$$
and
$$\textstyle \bigcup_{r=0}^{\infty}J_{p,q}^{r} = \bigcup_{r=0}^{\infty}\left(\partial_{p+q+1}^{X}\left(F_{p+q+1}^{p+r-1}\right)\cap F_{p+q}^{p}\right)
 = \partial_{p+q+1}^{X}\left(\bigcup_{r=0}^{\infty}F_{p+q+1}^{p+r-1}\right)\cap F_{p+q}^{p} = \im\left(\partial_{p+q+1}^{X}\right)\cap F_{p+q}^{p}$$
Therefore, the result follows from Theorem \ref{02.2}.
\end{proof}

\begin{corollary} \label{02.5}
Let $R$ be a ring, and $X$ an $R$-complex equipped with a filtration
$$\cdots \subseteq F^{p-1} \subseteq F^{p} \subseteq \cdots \subseteq X$$
such that the induced filtration of $X_{n}$ is finite for every $n\in \mathbb{Z}$. Then there exists a spectral sequence as follows:
$$\textstyle E_{p,q}^{1}=H_{p+q}\left(\frac{F^{p}}{F^{p-1}}\right) \Rightarrow H_{p+q}(X)$$
\end{corollary}

\begin{proof}
For any $p\in \mathbb{Z}$, let $\iota^{p}:F^{p}\rightarrow X$ be the inclusion morphism, and set $U_{n}^{p}= \im\left(H_{n}(\iota^{p})\right)$ for every $n\in \mathbb{Z}$. Let $n\in \mathbb{Z}$. By the hypothesis, $X_{n}$ has a finite filtration as follows:
$$0= F_{n}^{\alpha(n)}\subseteq F_{n}^{\alpha(n)+1}\subseteq \cdots \subseteq F_{n}^{\beta(n)}= X_{n}$$
If $p<\alpha(n)$, then $F_{n}^{p}=0$, so $H_{n}(F^{p})=0$, whence $U_{n}^{p}=0$. Let $\gamma(n)= \max\left\{\beta(n-1),\beta(n),\beta(n+1)\right\}$. If $p\geq \gamma(n)$, then $F_{n-1}^{p}=X_{n-1}$, $F_{n}^{p}=X_{n}$, and $F_{n+1}^{p}=X_{n+1}$, so $H_{n}(F^{p})=H_{n}(X)$, whence $U_{n}^{p}=H_{n}(X)$. This implies that there is a finite filtration of $H_{n}(X)$ as follows:
$$0= U_{n}^{\alpha(n)}\subseteq U_{n}^{\alpha(n)+1}\subseteq \cdots \subseteq U_{n}^{\gamma(n)}= H_{n}(X)$$
On the other hand, it is clear that $\bigcap_{p\in \mathbb{Z}}F_{n}^{p} =0$ and $\bigcup_{p\in \mathbb{Z}}F_{n}^{p} =X_{n}$ for every $n\in \mathbb{Z}$. Therefore, the result follows from Corollary \ref{02.4}.
\end{proof}

\section{Connectivity Lemma}

The key point to prove the convergence of Quillen's fundamental spectral sequences is a deep result of Quillen known as the \enquote{Connectivity Lemma} which allows one to deploy Corollary \ref{02.3}. This result appears in \cite[Theorem 6.12]{Qu3} without proof. Also, it is discussed in Quillen's unpublished notes \cite[Lemma 6.5 and Theorem 8.8]{Qu1}, where the argument is confusing and contains gaps to be filled. On the other hand, it is presented in \cite[Chapter XIII, Corollaire 4]{An} whose statement is a special case of Quillen's original version. In this section, we provide a proof of Quillen's original version using the ideas developed in \cite[Chapters XI, XII, and XIII]{An}. Our argument is complete, clear, and contains some new elements different from the references mentioned.

We first recall the notions of simplicial objects, their homotopy groups, simplicial resolutions, and the bar construction in the following remark.

\begin{remark} \label{03.1}
Let $\mathcal{C}$ be a category. A simplicial object over $\mathcal{C}$ is a family $A=\{A_{n}\}_{n\geq 0}$ of objects of $\mathcal{C}$ together with face morphisms $d_{n,i}^{A}:A_{n} \rightarrow A_{n-1}$ for every $n\geq 1$ and $0\leq i\leq n$, and degeneracy morphisms $s_{n,i}^{A}:A_{n} \rightarrow A_{n+1}$ for every $n\geq 0$ and $0\leq i\leq n$, that satisfy the following relations for every $n\geq 0$ and $0\leq i,j \leq n$:
\begin{enumerate}
\item[(i)] $d_{n,i}^{A}d_{n+1,j}^{A}=d_{n,j-1}^{A}d_{n+1,i}^{A}$ if $i<j$.
\item[(ii)] $s_{n,i}^{A}s_{n-1,j}^{A}=s_{n,j+1}^{A}s_{n-1,i}^{A}$ if $i\leq j$.
\item[(iii)]
 \label{eqn:damage piecewise}
 $d_{n,i}^{A}s_{n-1,j}^{A}=
      \begin{cases}
        s_{n-2,j-1}^{A}d_{n-1,i}^{A} & \text{if } i<j \\
        1^{A_{n-1}} & \text{if } i=j \text{ or } j+1 \\
        s_{n-2,j}^{A}d_{n-1,i-1}^{A} & \text{if } i>j+1
      \end{cases}$
\end{enumerate}

\noindent
Given two simplicial objects $A=\{A_{n}\}_{n\geq 0}$ and $B=\{B_{n}\}_{n\geq 0}$ over $\mathcal{C}$, a morphism $\tau:A \rightarrow B$ of simplicial objects is a collection $\tau=(\tau_{n})_{n\geq 0}$ of morphisms $\tau_{n}:A_{n} \rightarrow B_{n}$ of $\mathcal{C}$ for every $n \geq 0$, such that the following diagram is commutative for every $n\geq 1$ and $0\leq i\leq n$:
\begin{equation*}
  \begin{tikzcd}[column sep=3.5em,row sep=2em]
  A_{n-1} \arrow{r}{s_{n-1,i}^{A}} \arrow{d}{\tau_{n-1}}
  & A_{n} \arrow{r}{d_{n,i}^{A}} \arrow{d}{\tau_{n}}
  & A_{n-1} \arrow{d}{\tau_{n-1}}
  \\
  B_{n-1} \arrow{r}{s_{n-1,i}^{B}}
  & B_{n} \arrow{r}{d_{n,i}^{B}}
  & B_{n-1}
\end{tikzcd}
\end{equation*}

\noindent
It is immediate that we have a category $\mathpzc{s}\mathcal{C}$ of simplicial objects over $\mathcal{C}$. As a matter of fact, $\mathpzc{s}\mathcal{C}$ is a functor category. Let $\Delta$ denote the simplex category described by
$$\Obj(\Delta)=\left\{[n]=\{0,1,2,...,n\} \suchthat n\geq 0\right\}$$
and
$$\Mor_{\Delta}\left([n],[m]\right) = \left\{f\in \Mor_{\mathcal{S}\mathpzc{et}}\left([n],[m]\right) \suchthat f \textrm{ is order-preserving}\right\}$$
for every $[n],[m] \in \Obj(\Delta)$. Then it is folklore that $\mathpzc{s}\mathcal{C}=\mathcal{C}^{\Delta^{\op}}$. In other words, a simplicial object over $\mathcal{C}$ is nothing but a contravariant functor $\Delta \rightarrow \mathcal{C}$. In practice, we speak of simplicial sets, simplicial modules, or simplicial commutative algebras when we consider simplicial objects over the categories of sets, modules, or commutative algebras.

Given a ring $R$, one has the normalization functor $\Nor:\mathpzc{s}\mathcal{M}\mathpzc{od}(R) \rightarrow \mathcal{C}_{\geq 0}(R)$ and the Dold-Kan functor $\DK:\mathcal{C}_{\geq 0}(R) \rightarrow \mathpzc{s}\mathcal{M}\mathpzc{od}(R)$; see \cite[Definition 8.3.6, and 8.4.4]{We} or \cite[Definitions 4.5 and 4.9]{Fa}. The Dold-Kan Correspondence Theorem asserts that there is an adjoint equivalence of categories $(\DK,\Nor):\mathcal{C}_{\geq 0}(R) \leftrightarrows \mathpzc{s}\mathcal{M}\mathpzc{od}(R)$; see \cite[Theorem 8.4.1 and Exercise 8.4.2]{We}.

Given a simplicial set that satisfies the so-called Kan condition, one can define its homotopy groups; see \cite[Chapter 8, Section 8.3]{We}. It is folklore that simplicial modules and simplicial commutative algebras satisfy the Kan condition, so one can define their homotopy groups. In particular, if $R$ is a commutative ring and $M$ is a simplicial $R$-module, then for any $n\geq 0$, the $n$th homotopy group of $M$ is given by $\pi_{n}(M)= H_{n}\left(\Nor_{R}(M)\right) \cong H_{n}\left(C_{R}(M)\right)$ where $C_{R}(M)$ is the Moore complex associated with $M$. The same description applies to the homotopy groups of a simplicial commutative $R$-algebra; see \cite[Definition 8.3.6 and Lemma 8.3.7]{We} and \cite[Remark 4.6]{Fa}. We note that it follows from the Dold-Kan Correspondence Theorem that the normalization functor is exact, so if
$$0\rightarrow M'\rightarrow M\rightarrow M''\rightarrow 0$$
is a short exact sequence of simplicial $R$-modules, then we get the following natural exact sequence of homotopy groups:
$$\cdots \rightarrow \pi_{i}(M')\rightarrow \pi_{i}(M)\rightarrow \pi_{i}(M'')\rightarrow \pi_{i-1}(M')\rightarrow \pi_{i-1}(M)\rightarrow \pi_{i-1}(M'')\rightarrow $$$$ \cdots \rightarrow \pi_{0}(M')\rightarrow \pi_{0}(M)\rightarrow \pi_{0}(M'')\rightarrow 0$$
We further note that there is an adjoint pair $(\pi_{0},\Diag):\mathpzc{s}\mathcal{M}\mathpzc{od}(R) \leftrightarrows \mathcal{M}\mathpzc{od}(R)$ of functors in which $\Diag:\mathcal{M}\mathpzc{od}(R) \rightarrow \mathpzc{s}\mathcal{M}\mathpzc{od}(R)$ is the diagonal functor that assigns to each left $R$-module $M$, the constant simplicial left $R$-module $M$, and to each $R$-homomorphism $f:M\rightarrow N$, the morphism $f:M\rightarrow N$ of constant simplicial left $R$-modules.

Given a commutative ring $R$, a deep result of Quillen states that the category of simplicial commutative $R$-algebras is a model category; see \cite[Chapter II, Section 3]{Qu2}, also see \cite{Fa} for a purely algebraic proof. Accordingly, given a commutative $R$-algebra $A$, there exists a cofibrant replacement $A_{\textrm{c}} \rightarrow A$ for $A$ in the category of simplicial commutative $R$-algebras, which is referred to as a \enquote{simplicial algebra resolution} of $A$. In particular, there is an $R$-algebra epimorphism $\varepsilon:(A_{\textrm{c}})_{0}\rightarrow A$ with $\varepsilon d_{1,0}^{A_{\textrm{c}}} = \varepsilon d_{1,1}^{A_{\textrm{c}}}$ that induces \begin{equation*}
 \label{eqn:damage piecewise}
\pi_{n}(A_{\textrm{c}})=
 \begin{dcases}
  A & \textrm{if } n=0 \\
  0 & \textrm{if } n>0.
 \end{dcases}
\end{equation*}
More explicitly, there exists a concrete construction of a simplicial algebra resolution $R\langle X\rangle \rightarrow A$ for $A$ in which $R\langle X\rangle$ is a polynomial algebra in each degree; see \cite[Section 4]{Iy}. We make extensive use of such a resolution in the sequel.

Given a commutative ring $R$, an $R$-algebra $A$, a right $A$-module $M$, and a left $A$-module $N$, the bar construction $\Barr_{R}(M,A,N)$ is a simplicial $R$-module such that for any $n\geq 0$, we have
$$\Barr_{R}(M,A,N)_{n}= M\otimes_{R}T^{n}_{R}(A)\otimes_{R}N$$
where $T^{n}_{R}(A)$ is the $n$-fold tensor product of $A$, and for any $0\leq i\leq n$, the face morphism $d_{n,i}^{\Barr_{R}(M,A,N)}:\Barr_{R}(M,A,N)_{n}\rightarrow \Barr_{R}(M,A,N)_{n-1}$ is given by
$$d_{n,i}^{\Barr_{R}(M,A,N)}(a_{0}\otimes \cdots \otimes a_{n+1})=a_{0}\otimes \cdots \otimes a_{i}a_{i+1} \otimes \cdots \otimes a_{n+1}$$
and the degeneracy morphism $s_{n,i}^{\Barr_{R}(M,A,N)}:\Barr_{R}(M,A,N)_{n}\rightarrow \Barr_{R}(M,A,N)_{n+1}$ is given by
$$s_{n,i}^{\Barr_{R}(M,A,N)}(a_{0}\otimes \cdots \otimes a_{n+1})=a_{0}\otimes \cdots \otimes a_{i}\otimes 1 \otimes a_{i+1} \otimes \cdots \otimes a_{n+1}$$
for every $a_{0}\otimes \cdots \otimes a_{n+1}\in \Barr_{R}(M,A,N)_{n}$. If $A$ is a flat $R$-algebra, and either $M$ or $N$ is a flat $R$-module, then one has $\Tor_{i}^{A}(M,N)\cong \pi_{i}\left(\Barr_{R}(M,A,N)\right)$ for every $i\geq 0$; see \cite[8.6.12 and Example 8.7.2]{We} for a more general version.
\end{remark}

We also need the following lemmas.

\begin{lemma} \label{03.2}
Let $R$ be a ring, $\{a_{\alpha}\}_{\alpha \in \Phi}$ a family of elements of $R$ such that any finitely many distinct elements of it form a regular sequence on $R$ in some order, and $\mathfrak{a}=\langle a_{\alpha} \suchthat \alpha\in \Phi \rangle$. Then $\frac{\mathfrak{a}^{s}}{\mathfrak{a}^{s+1}}$ is a free $\frac{R}{\mathfrak{a}}$-module with a basis $\mathfrak{B}= \{a_{\alpha_{1}} \cdots a_{\alpha_{s}}+\mathfrak{a}^{s+1} \suchthat \alpha_{1},...,\alpha_{s}\in \Phi\}$ for every $s\geq 1$. As a result, there is a natural isomorphism $\Sym_{\frac{R}{\mathfrak{a}}}^{s}\left(\frac{\mathfrak{a}}{\mathfrak{a}^{2}}\right)\cong \frac{\mathfrak{a}^{s}}{\mathfrak{a}^{s+1}}$ for every $s\geq 1$.
\end{lemma}

\begin{proof}
Let $s\geq 1$. It is clear that $\mathfrak{a}^{s}=\langle a_{\alpha_{1}} \cdots a_{\alpha_{s}} \suchthat \alpha_{1},...,\alpha_{s}\in \Phi \rangle$, so $\frac{\mathfrak{a}^{s}}{\mathfrak{a}^{s+1}}$ is generated by $\mathfrak{B}$ as an $R$-module, hence as an $\frac{R}{\mathfrak{a}}$-module. On the other hand, if $\sum_{j=1}^{m}\left(r_{j}+\mathfrak{a}\right)\left(a_{\alpha_{1j}} \cdots a_{\alpha_{sj}} + \mathfrak{a}^{s+1}\right) = 0$ for some $r_{1},...,r_{m}\in R$, then $\sum_{j=1}^{m}r_{j}a_{\alpha_{1j}} \cdots a_{\alpha_{sj}} \in \mathfrak{a}^{s+1}$. Let $a_{1},...,a_{n}$ be the distinct elements that appear among $a_{\alpha_{ij}}$'s for every $1\leq i\leq s$ and $1\leq j\leq m$, and also among the generators of $\mathfrak{a}^{s+1}$ in the expression of $\sum_{j=1}^{m}r_{j}a_{\alpha_{1j}} \cdots a_{\alpha_{sj}}$. Allowing some coefficients and exponents to be zero if necessary, we can rewrite the latter inclusion as $\sum_{i_{1}+ \cdots +i_{n}=s} r_{i_{1},...,i_{n}} a_{1}^{i_{1}} \cdots a_{n}^{i_{n}} \in \langle a_{1},...,a_{n} \rangle^{s+1} \subseteq \mathfrak{a}^{s+1}$ where every $r_{i_{1},...,i_{n}} \in R$. Assuming $a_{1},...,a_{n}$ are in the right order to form a regular sequence on $R$, we deduce from \cite[Theorem 1.1.7]{BH} that every $r_{i_{1},...,i_{n}} \in \langle a_{1},...,a_{n} \rangle \subseteq \mathfrak{a}$. But it is obvious that $r_{1},...,r_{m}$ are among $r_{i_{1},...,i_{n}}$'s, so $r_{1},...,r_{m}\in \mathfrak{a}$. That is, $\mathfrak{B}$ is linearly independent.

For the next part, we note that by using the universal property of symmetric powers, we can define an $\frac{R}{\mathfrak{a}}$-epimorphism $\zeta:\Sym_{\frac{R}{\mathfrak{a}}}^{s}\left(\frac{\mathfrak{a}}{\mathfrak{a}^{2}}\right) \rightarrow \frac{\mathfrak{a}^{s}}{\mathfrak{a}^{s+1}}$ by setting $\zeta\left(\left(a_{1}+\mathfrak{a}^{2}\right)\ast \cdots \ast \left(a_{s}+\mathfrak{a}^{2}\right)\right)= a_{1}\cdots a_{s}+ \mathfrak{a}^{s+1}$ for every $a_{1},...,a_{s}\in \mathfrak{a}$. On the other hand, since $\frac{\mathfrak{a}^{s}}{\mathfrak{a}^{s+1}}$ is a free $\frac{R}{\mathfrak{a}}$-module, we can define an $\frac{R}{\mathfrak{a}}$-homomorphism $\xi:\frac{\mathfrak{a}^{s}}{\mathfrak{a}^{s+1}} \rightarrow \Sym_{\frac{R}{\mathfrak{a}}}^{s}\left(\frac{\mathfrak{a}}{\mathfrak{a}^{2}}\right)$ by setting $\xi\left(a_{\alpha_{1}}\cdots a_{\alpha_{s}}+ \mathfrak{a}^{s+1}\right)= \left(a_{\alpha_{1}}+\mathfrak{a}^{2}\right) \ast \cdots \ast \left(a_{\alpha_{s}}+\mathfrak{a}^{2}\right)$ for every $\alpha_{1},...,\alpha_{s}\in \Phi$. By checking on the basis elements, we observe that $\xi\zeta=1^{\Sym_{\frac{R}{\mathfrak{a}}}^{s}\left(\frac{\mathfrak{a}}{\mathfrak{a}^{2}}\right)}$, so $\zeta$ is injective, hence an isomorphism.
\end{proof}

\begin{lemma} \label{03.3}
Let $R$ be a ring, $M$ a simplicial right $R$-module, and $N$ a simplicial left $R$-module. Then the following natural isomorphism of $\mathbb{Z}$-modules holds for every $i\geq 0$:
$$\pi_{i}(M\otimes_{R}N) \cong H_{i}\left(\Nor_{R}(M)\otimes_{R}\Nor_{R}(N)\right)$$
\end{lemma}

\begin{proof}
In light of \cite[Corollaries 5.7 and 5.10]{Fa}, we have following natural isomorphisms for every $i\geq 0$:
$$\pi_{i}(M\otimes_{R}N) = H_{i}\left(\Nor_{R}(M\otimes_{R}N)\right) \cong H_{i}\left(\Nor_{R}(M)\boxtimes_{R}\Nor_{R}(N)\right) \cong H_{i}\left(\Nor_{R}(M)\otimes_{R}\Nor_{R}(N)\right)$$
\end{proof}

\begin{corollary} \label{03.4}
Let $R$ be a ring, $M$ a simplicial right $R$-module, and $N$ a simplicial left $R$-module. Suppose that either $M_{i}$ is a flat right $R$-module for every $i\geq 0$, or $N_{i}$ is a flat left $R$-module for every $i\geq 0$. Then the following natural isomorphism of $\mathbb{Z}$-modules holds:
$$\pi_{0}(M\otimes_{R}N) \cong \pi_{0}(M)\otimes_{R}\pi_{0}(N)$$
\end{corollary}

\begin{proof}
Suppose that $M_{i}$ is a flat right $R$-module for every $i\geq 0$. For any $i\geq 0$, $\Nor_{R}(M)_{i}$ is a direct summand of $M_{i}$ by \cite[Lemma 8.3.7]{We}, so we conclude that $\Nor_{R}(M)_{i}$ is a flat right $R$-module. As a result, by \cite[Theorem 10.90]{Ro}, there is a first quadrant spectral sequence as follows:
$$E_{p,q}^{2}= \bigoplus_{i+j=q}\Tor_{p}^{R}\left(H_{i}\left(\Nor_{R}(M)\right),H_{j}\left(\Nor_{R}(N)\right)\right) \Rightarrow H_{p+q}\left(\Nor_{R}(M)\otimes_{R}\Nor_{R}(N)\right)$$
As a result, we have the low-degree isomorphism $E_{0,0}^{2} \cong H_{0}$. However, since $H_{i}\left(\Nor_{R}(M)\right)=0$ for every $i<0$ and $H_{j}\left(\Nor_{R}(N)\right)=0$ for every $j<0$, we have:
$$E_{0,0}^{2} = \bigoplus_{i+j=0}\Tor_{0}^{R}\left(H_{i}\left(\Nor_{R}(M)\right),H_{j}\left(\Nor_{R}(N)\right)\right) \cong H_{0}\left(\Nor_{R}(M)\right)\otimes_{R}H_{0}\left(\Nor_{R}(N)\right) = \pi_{0}(M)\otimes_{R}\pi_{0}(N)$$
On the other hand, by Lemma \ref{03.3}, we have $H_{0}=H_{0}\left(\Nor_{R}(M)\otimes_{R}\Nor_{R}(N)\right)\cong \pi_{0}(M\otimes_{R}N)$. Therefore, the result follows. The other case is similarly established.
\end{proof}

The following corollary generalizes \cite[Chapter VIII, Lemme 18]{An} by replacing \enquote{projectivity} with \enquote{flatness}. The argument here is different as well.

\begin{corollary} \label{03.5}
Let $R$ be a ring, $M$ a simplicial right $R$-module, $N$ a simplicial left $R$-module, and $m,n\geq 0$. Suppose that either $M_{i}$ is a flat right $R$-module for every $i\geq 0$, or $N_{i}$ is a flat left $R$-module for every $i\geq 0$. If $\pi_{i}(M)=0$ for every $i\leq m$ and $\pi_{i}(N)=0$ for every $i\leq n$, then $\pi_{i}(M\otimes_{R}N)=0$ for every $i\leq m+n+1$.
\end{corollary}

\begin{proof}
Suppose that $M_{i}$ is a flat right $R$-module for every $i\geq 0$. As in the proof of Corollary \ref{03.4}, there is a first quadrant spectral sequence as follows:
$$E_{p,q}^{2}= \bigoplus_{i+j=q}\Tor_{p}^{R}\left(H_{i}\left(\Nor_{R}(M)\right),H_{j}\left(\Nor_{R}(N)\right)\right) \Rightarrow H_{p+q}\left(\Nor_{R}(M)\otimes_{R}\Nor_{R}(N)\right)$$
If $q\leq m+n+1$, then in view of $q=i+j$, we must either have $i\leq m$ or $j\leq n$. If $i\leq m$, then $H_{i}\left(\Nor_{R}(M)\right)=\pi_{i}(M)=0$, and if $j\leq n$, then $H_{j}\left(\Nor_{R}(N)\right)=\pi_{j}(N)=0$. Thus in any case, $E_{p,q}^{2}=0$. That is, $E_{p,q}^{2}=0$ for every $p,q\in \mathbb{Z}$ with $q\leq m+n+1$, so in view of Lemma \ref{03.3}, we infer that $\pi_{i}(M\otimes_{R}N)\cong H_{i}\left(\Nor_{R}(M)\otimes_{R}\Nor_{R}(N)\right)=0$ for every $i\leq m+n+1$. The other case is similarly established.
\end{proof}

Now we present Quillen's Connectivity Lemma.

\begin{theorem} \label{03.6}
Let $R$ be a commutative ring, and $A$ a commutative $R$-algebra such that the natural $R$-algebra epimorphism $\mu_{A}:A\otimes_{R}A\rightarrow A$, given by $\mu_{A} (a\otimes b)=ab$ for every $a,b\in A$, is an isomorphism. Let $\rho:R\langle X\rangle \rightarrow A$ be a simplicial algebra resolution of $A$, and consider the morphism $\varphi:R\langle X\rangle \otimes_{R}A \rightarrow A$ of simplicial $A$-algebras, given by $\varphi_{n}(p\otimes a)=\rho_{n}(p)a$ for every $n\geq 0$, $p\in R\langle X\rangle_{n}$, and $a\in A$. If $\mathfrak{K}= \Ker(\varphi)$ is regarded as a simplicial ideal in the simplicial $A$-algebra $R\langle X\rangle \otimes_{R}A$, then $\pi_{k}(\mathfrak{K}^{s})=0$ for every $s\geq 0$ and $k\leq s-1$.
\end{theorem}

\begin{proof}
Let $\upsilon_{A}:R\rightarrow A$ be the structure map of $A$, and $\xi:R\langle X\rangle \otimes_{R}A \rightarrow A\langle X\rangle$ the isomorphism of simplicial $A$-algebras. If for any $n\geq 0$, we set $R\langle X \rangle_{n}= R[X_n{}]$ where $X_{n}= \left\{X_{n,\alpha} \suchthat \alpha\in \Phi_{n}\right\}$, then we see that $\xi_{n}(p\otimes a)= \sum_{i_{1},...,i_{k}\geq 0}a\upsilon_{A}(r_{i_{1}...i_{k}})X_{n,\alpha_{1}}^{i_{1}}\cdots X_{n,\alpha_{k}}^{i_{k}}$ for every $p= p\left(X_{n,\alpha_{1}},\ldots, X_{n,\alpha_{k}}\right)= \sum_{i_{1},...,i_{k}\geq 0}r_{i_{1}...i_{k}}X_{n,\alpha_{1}}^{i_{1}}\cdots X_{n,\alpha_{k}}^{i_{k}}\in R\langle X \rangle_{n}$ and $a\in A$. Now set $a_{n,\alpha}=\rho_{n}\left(X_{n,\alpha}\right)\in A$ for every $n\geq 0$ and $\alpha\in \Phi_{n}$. For any $n\geq 0$, consider the evaluation map $\ev_{n}:A\langle X\rangle_{n}\rightarrow A$ given by $\ev_{n}\left(X_{n,\alpha}\right)=a_{n,\alpha}$ for every $\alpha\in \Phi_{n}$, and set $\mathfrak{H}_{n}= \Ker(\ev_{n})= \left\langle X_{n,\alpha}-a_{n,\alpha} \suchthat \alpha\in \Phi_{n} \right\rangle$. Then we have the following commutative diagram of simplicial $A$-modules with exact rows:
\begin{equation*}
  \begin{tikzcd}
  0 \arrow{r} & \mathfrak{K} \arrow{r} \arrow{d}{\cong} & R\langle X\rangle \otimes_{R}A \arrow{r}{\varphi} \arrow{d}{\cong}[swap]{\xi} & A \arrow{r} \ar[equal]{d} & 0
  \\
  0 \arrow{r} & \mathfrak{H} \arrow{r} & A\langle X\rangle \arrow{r}{\ev} & A \arrow{r} & 0
\end{tikzcd}
\end{equation*}
The isomorphism $\mathfrak{K} \cong \mathfrak{H}$ is induced by $\xi$, so we notice that $\mathfrak{K}^{s} \cong \mathfrak{H}^{s}$ for every $s\geq 0$, whence $\pi_{k}(\mathfrak{K}^{s}) \cong \pi_{k}(\mathfrak{H}^{s})$ for every $s,k\geq 0$. Thus it suffices to show that $\pi_{k}(\mathfrak{H}^{s})=0$ for every $s\geq 0$ and $k\leq s-1$.

Let $n\geq 0$. Since $\mathfrak{H}_{n}= \left\langle X_{n,\alpha}-a_{n,\alpha} \suchthat \alpha\in \Phi_{n} \right\rangle$, and any finitely many distinct elements of these generators form a regular sequence on $A[X_{n}]$, Lemma \ref{03.2} implies that $\frac{\mathfrak{H}_{n}^{s}}{\mathfrak{H}_{n}^{s+1}}$ is a free $\frac{A[X_{n}]}{\mathfrak{H}_{n}}$-module for every $s\geq 1$. But $\frac{A[X_{n}]}{\mathfrak{H}_{n}}\cong A$, so $\frac{\mathfrak{H}_{n}^{s}}{\mathfrak{H}_{n}^{s+1}}$ is a free $A$-module for every $s\geq 1$. Moreover, by Lemma \ref{03.2}, we have $\Sym_{A}^{s}\left(\frac{\mathfrak{H}_{n}}{\mathfrak{H}_{n}^{2}}\right)\cong \Sym_{\frac{A[X_{n}]}{\mathfrak{H}_{n}}}^{s}\left(\frac{\mathfrak{H}_{n}}{\mathfrak{H}_{n}^{2}}\right)\cong \frac{\mathfrak{H}_{n}^{s}}{\mathfrak{H}_{n}^{s+1}}$ for every $s\geq 1$. Also, since $\frac{\mathfrak{H}_{n}}{\mathfrak{H}_{n}^{2}}$ is a free $A$-module with a basis $\left\{X_{n,\alpha}-a_{n,\alpha}+ \mathfrak{H}_{n}^{2} \suchthat \alpha\in \Phi_{n}\right\}$, we infer that $\Sym_{A}\left(\frac{\mathfrak{H}_{n}}{\mathfrak{H}_{n}^{2}}\right)\cong A[X_{n}]= A\langle X \rangle_{n}$. The short exact sequence
$$0\rightarrow \mathfrak{H}_{n}\rightarrow A\langle X \rangle_{n}\rightarrow A\rightarrow 0$$
of $A$-modules is split, so $A\langle X \rangle_{n} \cong \mathfrak{H}_{n}\oplus A$, implying that $\mathfrak{H}_{n}$ is a projective $A$-module. Furthermore, the short exact sequence
$$\textstyle 0\rightarrow \mathfrak{H}_{n}^{2}\rightarrow \mathfrak{H}_{n}\rightarrow \frac{\mathfrak{H}_{n}}{\mathfrak{H}_{n}^{2}}\rightarrow 0$$
of $A$-modules is split, so $\mathfrak{H}_{n} \cong \mathfrak{H}_{n}^{2}\oplus \frac{\mathfrak{H}_{n}}{\mathfrak{H}_{n}^{2}}$, whence $\mathfrak{H}_{n}^{2}$ is a projective $A$-module. Continuing in this way, we see that for any $s\geq 1$, the short exact sequence
$$\textstyle 0\rightarrow \mathfrak{H}_{n}^{s+1}\rightarrow \mathfrak{H}_{n}^{s}\rightarrow \frac{\mathfrak{H}_{n}^{s}}{\mathfrak{H}_{n}^{s+1}}\rightarrow 0$$
of $A$-modules is split, implying that $\mathfrak{H}_{n}^{s}$ is a projective $A$-module.

We now argue by induction on $s$ to show that $\pi_{k}(\mathfrak{H}^{s})=0$ for every $s\geq 0$ and $k\leq s-1$. The result is clear for $s=0$. Let $s=1$. Considering the differential $\partial_{1}^{\Nor_{R}\left(R\langle X\rangle\right)}:\Nor_{R}\left(R\langle X\rangle\right)_{1} \rightarrow \Nor_{R}\left(R\langle X\rangle\right)_{0}$ and the quotient map $\varpi:\Nor_{R}\left(R\langle X\rangle\right)_{0} \rightarrow \Coker\left(\partial_{1}^{\Nor_{R}\left(R\langle X\rangle\right)}\right)$, the induced $A$-homomorphism
$$\pi_{0}\left(R\langle X\rangle\otimes_{R}A\right) = \Coker\left(\partial_{1}^{\Nor_{R}\left(R\langle X\rangle\right)}\otimes_{R}A\right) \xrightarrow{\overline{\varpi\otimes_{R}A}} \Coker\left(\partial_{1}^{\Nor_{R}\left(R\langle X\rangle\right)}\right)\otimes_{R}A = \pi_{0}\left(R\langle X\rangle\right)\otimes_{R}A$$
is an isomorphism. Now consider the following commutative diagram:
\begin{equation*}
  \begin{tikzcd}[column sep=4.5em,row sep=2em]
  \pi_{0}\left(R\langle X\rangle\otimes_{R}A\right) \arrow{r}{\pi_{0}(\varphi)} \arrow{d}{\cong}[swap]{\overline{\varpi\otimes_{R}A}} & A
  \\
  \pi_{0}\left(R\langle X\rangle\right)\otimes_{R}A \arrow{r}{\pi_{0}(\rho)\otimes_{R}A} & A\otimes_{R}A \arrow{u}{\cong}[swap]{\mu_{A}}
\end{tikzcd}
\end{equation*}
Since $\rho$ is a cofibrant replacement, it is a weak equivalence, so $\Nor_{R}(\rho)$ is a quasi-isomorphism; see \cite[Theorem 4.17]{GS} or \cite[Theorem 6.6]{Fa}. In particular, $\pi_{0}(\rho)= H_{0}\left(\Nor_{R}(\rho)\right)$ is an isomorphism, so $\pi_{0}(\rho)\otimes_{R}A$ is an isomorphism. Thus the above diagram shows that $\pi_{0}(\varphi)$ is an isomorphism. Now in light of Remark \ref{03.1}, the short exact sequence
$$0\rightarrow \mathfrak{K}\rightarrow R\langle X\rangle\otimes_{R}A \xrightarrow{\varphi} A \rightarrow 0$$
yields the following exact sequence:
$$0=\pi_{1}(A) \rightarrow \pi_{0}\left(\mathfrak{K}\right) \rightarrow \pi_{0}\left(R\langle X\rangle\otimes_{R}A\right) \xrightarrow{\pi_{0}(\varphi)} \pi_{0}(A)=A \rightarrow 0$$
As $\pi_{0}(\varphi)$ is an isomorphism, we conclude that $\pi_{0}\left(\mathfrak{H}\right)\cong \pi_{0}\left(\mathfrak{K}\right)=0$.

Before moving on to the inductive step, we need to prove a few facts. Firstly, we show that if $n\geq 0$, $s\geq 1$, and $\kappa_{s}:\frac{A\langle X\rangle_{n}}{\mathfrak{H}_{n}^{s}} \rightarrow \frac{A\langle X\rangle_{n}}{\mathfrak{H}_{n}^{s-1}}$ is the natural $A$-epimorphism, then
$$\textstyle \Tor_{i}^{A\langle X\rangle_{n}}(A,\kappa_{s}):\Tor_{i}^{A\langle X\rangle_{n}}\left(A,\frac{A\langle X\rangle_{n}}{\mathfrak{H}_{n}^{s}}\right) \rightarrow \Tor_{i}^{A\langle X\rangle_{n}}\left(A,\frac{A\langle X\rangle_{n}}{\mathfrak{H}_{n}^{s-1}}\right)$$
is the zero map for every $i\geq 1$. Let $L$ be a free $A$-module, and $M$ a $\Sym_{A}(L)$-module. We construct a connective $A$-complex $K^{A}(L,M)$ by setting $K^{A}(L,M)_{i}=\bigwedge_{A}^{i}L\otimes_{A}M$ for every $i\geq 0$, and defining $\partial_{i}^{K^{A}(L,M)}:K^{A}(L,M)_{i}\rightarrow K^{A}(L,M)_{i-1}$ by
$$\partial_{i}^{K^{A}(L,M)}\left(\left(x_{1}\wedge \cdots \wedge x_{i}\right)\otimes y\right)= \sum_{k=1}^{i}(-1)^{k+1}\left(x_{1}\wedge \cdots \wedge \widehat{x_{k}} \wedge \cdots \wedge x_{i}\right)\otimes x_{k}y$$
for every $i\geq 1$, $x_{1},...,x_{i}\in L$, and $y\in M$. It is easily verified that $K^{A}(L,M)$ is a connective $A$-complex. Now we consider the special case $K^{A}\left(L,\Sym_{A}(L)\right)$. Since $L$ is a free $A$-module, $\bigwedge_{A}^{i}L$ is a free $A$-module for every $i\geq 0$, so $K^{A}\left(L,\Sym_{A}(L)\right)_{i}= \bigwedge_{A}^{i}L\otimes_{A}\Sym_{A}(L)$ is a free $\Sym_{A}(L)$-module for every $i\geq 0$. Considering the $A$-homomorphism
$$L\otimes_{A}\Sym_{A}(L) =K^{A}\left(L,\Sym_{A}(L)\right)_{1} \xrightarrow{\partial_{1}^{K^{A}(L,M)}} K^{A}\left(L,\Sym_{A}(L)\right)_{0} = A\otimes_{A}\Sym_{A}(L)$$
where $\partial_{1}^{K^{A}(L,M)}(x\otimes y)=1\otimes(x\ast y)$ for every $x\in L$ and $y\in \Sym_{A}(L)$, we see that:
\begin{equation*}
\begin{split}
 H_{0}\left(K^{A}\left(L,\Sym_{A}(L)\right)\right) & = \textstyle \frac{K^{A}\left(L,\Sym_{A}(L)\right)_{0}}{\im\left(\partial_{1}^{K^{A}\left(L,\Sym_{A}(L)\right)}\right)}
 = \frac{A\otimes_{A}\Sym_{A}(L)}{A\otimes_{A}L \Sym_{A}(L)} \cong \frac{\Sym_{A}(L)}{L \Sym_{A}(L)} = \frac{A\oplus L\oplus \Sym_{A}^{2}(L)\oplus \cdots}{L\left(A\oplus L\oplus \Sym_{A}^{2}(L)\oplus \cdots \right)} \\
 & = \textstyle \frac{A\oplus L\oplus \Sym_{A}^{2}(L)\oplus \cdots}{L\oplus \Sym_{A}^{2}(L)\oplus \cdots}\cong A \\
\end{split}
\end{equation*}
We next show that $H_{i}\left(K^{A}\left(L,\Sym_{A}(L)\right)\right)=0$ for every $i\geq 1$. First assume that $L=A^{m}$ for some $m\geq 1$. We argue by induction on $m$ to show that $H_{i}\left(K^{A}\left(A^{m},\Sym_{A}(A^{m})\right)\right)=0$ for every $i\geq 1$. Let $m=1$. Then $K^{A}\left(A,\Sym_{A}(A)\right)$ is as follows:
$$0 \rightarrow A\otimes_{A}\Sym_{A}(A) \xrightarrow{\partial_{1}^{K^{A}\left(A,\Sym_{A}(A)\right)}} A\otimes_{A}\Sym_{A}(A) \rightarrow 0$$
We note that $\partial_{1}^{K^{A}\left(A,\Sym_{A}(A)\right)}(a\otimes y)= 1\otimes ay= a\otimes y$ for every $a\in A$ and $y\in \Sym_{A}(A)$, so it follows that $\partial_{1}^{K^{A}\left(A,\Sym_{A}(A)\right)}=1^{A\otimes_{A}\Sym_{A}(A)}$, whence $H_{1}\left(K^{A}\left(A,\Sym_{A}(A)\right)\right)=0$. Therefore, $H_{i}\left(K^{A}\left(A,\Sym_{A}(A)\right)\right)=0$ for every $i\geq 1$. Now suppose that $m\geq 2$, and the result holds for $m-1$. Let $x\in A^{m}$, and define an $A$-homomorphism $s_{i}:K^{A}\left(A^{m},\Sym_{A}(A^{m})\right)_{i} \rightarrow K^{A}\left(A^{m},\Sym_{A}(A^{m})\right)_{i+1}$ by setting $s_{i}=0$ for every $i\leq -1$, $s_{0}(1\otimes y)=x\otimes y$, and $s_{i}\left((x_{1}\wedge \cdots \wedge x_{i})\otimes y\right)= (x\wedge x_{1}\wedge \cdots \wedge x_{i})\otimes y$ for every $i\geq 1$, $x_{1},...,x_{i} \in A^{m}$, and $y\in \Sym_{A}(A^{m})$. Then one can see by inspection that
$$\partial_{i+1}^{K^{A}\left(A^{m},\Sym_{A}(A^{m})\right)}s_{i} + s_{i-1}\partial_{i}^{K^{A}\left(A^{m},\Sym_{A}(A^{m})\right)} = x1^{K^{A}\left(A^{m},\Sym_{A}(A^{m})\right)_{i}}$$
for every $i\in \mathbb{Z}$, i.e. $x1^{K^{A}\left(A^{m},\Sym_{A}(A^{m})\right)} \sim 0$. As a result, for any $i\geq 0$, we have $x1^{H_{i}\left(K^{A}\left(A^{m},\Sym_{A}(A^{m})\right)\right)}= H_{i}\left(x1^{K^{A}\left(A^{m},\Sym_{A}(A^{m})\right)}\right) = H_{i}(0) = 0$, so $xH_{i}\left(K^{A}\left(A^{m},\Sym_{A}(A^{m})\right)\right) = 0$. Now specialize to $x=(0,...,0,1)\in A^{m}$, and consider the following short exact sequence of $A$-modules:
$$\textstyle 0 \rightarrow \Sym_{A}(A^{m}) \xrightarrow{x} \Sym_{A}(A^{m}) \rightarrow \frac{\Sym_{A}(A^{m})}{x \Sym_{A}(A^{m})} \rightarrow 0$$
Then we get the following short exact sequence of $A$-modules:
$$\textstyle 0 \rightarrow \bigwedge_{A}^{i}A^{m}\otimes_{A} \Sym_{A}(A^{m}) \xrightarrow{x} \bigwedge_{A}^{i}A^{m}\otimes_{A} \Sym_{A}(A^{m}) \rightarrow \bigwedge_{A}^{i}A^{m}\otimes_{A} \frac{\Sym_{A}(A^{m})}{x \Sym_{A}(A^{m})} \rightarrow 0$$
We further note that:
$$\textstyle \frac{\Sym_{A}(A^{m})}{x \Sym_{A}(A^{m})} \cong \frac{A[Z_{1},...,Z_{m}]}{\langle Z_{m} \rangle} \cong A[Z_{1},...,Z_{m-1}] \cong \Sym_{A}(A^{m-1})$$
and
\begin{equation*}
\begin{split}
\textstyle \bigwedge_{A}^{i}A^{m} & \cong \textstyle \bigwedge_{A}^{i}\left(A^{m-1}\oplus A\right) \cong \bigoplus_{u+v=i}\left(\bigwedge_{A}^{u}A^{m-1} \otimes_{A} \bigwedge_{A}^{v}A\right) \cong \left(\bigwedge_{A}^{i}A^{m-1}\otimes_{A}A\right) \oplus \left(\bigwedge_{A}^{i-1}A^{m-1}\otimes_{A}A\right) \\
& \textstyle \cong \bigwedge_{A}^{i}A^{m-1} \oplus \bigwedge_{A}^{i-1}A^{m-1}
\end{split}
\end{equation*}
Thus we get:
\begin{equation*}
\begin{split}
\textstyle \bigwedge_{A}^{i}A^{m}\otimes_{A}\frac{\Sym_{A}(A^{m})}{x \Sym_{A}(A^{m})} & \textstyle \cong \left(\bigwedge_{A}^{i}A^{m-1} \oplus \textstyle \bigwedge_{A}^{i-1}A^{m-1}\right)\otimes_{A}\Sym_{A}(A^{m-1}) \\
& \textstyle \cong \left(\bigwedge_{A}^{i}A^{m-1}\otimes_{A}\Sym_{A}(A^{m-1})\right) \oplus \left(\textstyle \bigwedge_{A}^{i-1}A^{m-1}\otimes_{A}\Sym_{A}(A^{m-1})\right) \\
& = K^{A}\left(A^{m-1},\Sym_{A}(A^{m-1})\right)_{i} \oplus K^{A}\left(A^{m-1},\Sym_{A}(A^{m-1})\right)_{i-1}
\end{split}
\end{equation*}
Hence we obtain the following short exact sequence of $A$-modules for every $i\geq 0$:
\footnotesize
$$0 \rightarrow K^{A}\left(A^{m},\Sym_{A}(A^{m})\right)_{i} \xrightarrow{x} K^{A}\left(A^{m},\Sym_{A}(A^{m})\right)_{i} \rightarrow K^{A}\left(A^{m-1},\Sym_{A}(A^{m-1})\right)_{i} \oplus K^{A}\left(A^{m-1},\Sym_{A}(A^{m-1})\right)_{i-1} \rightarrow 0$$
\normalsize
It can be easily checked that the homomorphisms in the above sequence are compatible with the differentials, so we get the following short exact sequence of $A$-complexes:
\small
$$0 \rightarrow K^{A}\left(A^{m},\Sym_{A}(A^{m})\right) \xrightarrow{x} K^{A}\left(A^{m},\Sym_{A}(A^{m})\right) \rightarrow K^{A}\left(A^{m-1},\Sym_{A}(A^{m-1})\right) \oplus \Sigma K^{A}\left(A^{m-1},\Sym_{A}(A^{m-1})\right) \rightarrow 0$$
\normalsize
This gives the following exact sequence for every $i\geq 0$:
$$H_{i+1}\left(K^{A}\left(A^{m-1},\Sym_{A}(A^{m-1})\right)\right) \oplus H_{i}\left(K^{A}\left(A^{m-1},\Sym_{A}(A^{m-1})\right)\right) \rightarrow H_{i}\left(K^{A}\left(A^{m},\Sym_{A}(A^{m})\right)\right) $$$$ \xrightarrow{x} H_{i}\left(K^{A}\left(A^{m},\Sym_{A}(A^{m})\right)\right)$$
But $xH_{i}\left(K^{A}\left(A^{m},\Sym_{A}(A^{m})\right)\right) = 0$ for every $i\geq 0$, so we get the following exact sequence for every $i\geq 0$:
$$H_{i+1}\left(K^{A}\left(A^{m-1},\Sym_{A}(A^{m-1})\right)\right) \oplus H_{i}\left(K^{A}\left(A^{m-1},\Sym_{A}(A^{m-1})\right)\right) \rightarrow H_{i}\left(K^{A}\left(A^{m},\Sym_{A}(A^{m})\right)\right) \rightarrow 0$$
By the induction hypothesis, $H_{i}\left(K^{A}\left(A^{m-1},\Sym_{A}(A^{m-1})\right)\right) = 0$ for every $i\geq 1$, so the above exact sequence implies that $H_{i}\left(K^{A}\left(A^{m},\Sym_{A}(A^{m})\right)\right)= 0$ for every $i\geq 1$. Now suppose that $L$ is an arbitrary free $A$module. By Lazard's Theorem \cite[Theorem 5.40]{Ro}, one has $L \cong \underset{\alpha\in \Phi}{\varinjlim} L_{\alpha}$ where $\Phi$ is a filtered poset and $L_{\alpha}$ is a finitely generated free $A$-module for every $\alpha \in \Phi$. Using what was proved above, we have for every $i\geq 1$:
\begin{equation*}
\begin{split}
H_{i}\left(K^{A}\left(L,\Sym_{A}(L)\right)\right) & \cong H_{i}\left(K^{A}\left(\underset{\alpha\in \Phi}{\varinjlim} L_{\alpha},\Sym_{A}\left(\underset{\alpha\in \Phi}{\varinjlim} L_{\alpha}\right)\right)\right) \cong H_{i}\left(\underset{\alpha\in \Phi}{\varinjlim} K^{A}\left(L_{\alpha},\Sym_{A}(L_{\alpha})\right)\right) \\
& \cong \underset{\alpha\in \Phi}{\varinjlim} H_{i}\left(K^{A}\left(L_{\alpha},\Sym_{A}(L_{\alpha})\right)\right) = 0
\end{split}
\end{equation*}
All in all, we infer that $K^{A}\left(L,\Sym_{A}(L)\right)$ is a projective resolution of $A$ regarded as a $\Sym_{A}(L)$-module via the canonical projection $\Sym_{A}(L)\rightarrow A$. As a result, there is a natural isomorphism
$$\Tor_{i}^{\Sym_{A}(L)}(A,-) \cong H_{i}\left(K^{A}\left(L,\Sym_{A}(L)\right)\otimes_{\Sym_{A}(L)}-\right)$$
of functors for every $i\geq 0$. In particular, we have the following natural isomorphisms for every $i\geq 0$:
$$\Tor_{i}^{\Sym_{A}(L)}(A,M) \cong H_{i}\left(K^{A}\left(L,\Sym_{A}(L)\right)\otimes_{\Sym_{A}(L)}M\right) \cong H_{i}\left(K^{A}(L,M)\right)$$
Now suppose that $M=\bigoplus_{j=0}^{\infty} M_{j}$ where $M_{j}$ is an $A$-submodule of $M$ for every $j\geq 0$ with the property that $\Sym_{A}^{i}(L)M_{j} \subseteq M_{i+j}$ for every $i,j\geq 0$. Then we have for every $i\geq 0$:
$$\textstyle K^{A}(L,M)_{i} = \bigwedge_{A}^{i}L \otimes_{A}M = \bigwedge_{A}^{i}L \otimes_{A}\left(\bigoplus_{j=0}^{\infty}M_{j}\right) \cong \bigoplus_{j=0}^{\infty}\left(\bigwedge_{A}^{i}L\otimes_{A}M_{j}\right)$$
Setting $K^{A}(L,M)_{i,j}:= \bigwedge_{A}^{i}L \otimes_{A}M_{j}$ for every $i,j\geq 0$, we have
$$\textstyle K^{A}(L,M)_{i}\cong \bigoplus_{j=0}^{\infty}K^{A}(L,M)_{i,j}$$
for every $i\geq 0$. If $i\geq 1$, $x_{1},...,x_{i}\in L$, and $y\in M_{j}$ for some $j\geq 0$, then since $x_{k}\in L= \Sym_{A}^{1}(L)$ for every $1\leq k \leq i$, we have $x_{k}y\in M_{j+1}$, so we have:
\small
$$\textstyle \partial_{i}^{K^{A}(L,M)}\left(\left(x_{1}\wedge \cdots \wedge x_{i}\right)\otimes y\right)= \sum_{k=1}^{i}(-1)^{k+1}\left(x_{1}\wedge \cdots \wedge \widehat{x_{k}} \wedge \cdots \wedge x_{i}\right)\otimes x_{k}y \in \bigwedge_{A}^{i-1}L \otimes_{A}M_{j+1}= K^{A}(L,M)_{i-1,j+1}$$
\normalsize
It follows that $\partial_{i}^{K^{A}(L,M)}$ induces an $A$-homomorphism $\partial_{i,j}^{K^{A}(L,M)}:K^{A}(L,M)_{i,j}\rightarrow K^{A}(L,M)_{i-1,j+1}$ for every $i,j\geq 0$. Consider the following differentials for every $i,j\geq 0$:
$$K^{A}(L,M)_{i+1,j-1} \xrightarrow{\partial_{i+1,j-1}^{K^{A}(L,M)}} K^{A}(L,M)_{i,j} \xrightarrow{\partial_{i,j}^{K^{A}(L,M)}} K^{A}(L,M)_{i-1,j+1}$$
Setting $H_{i,j}\left(K^{A}(L,M)\right):= \frac{\Ker\left(\partial_{i,j}^{K^{A}(L,M)}\right)}{\im\left(\partial_{i+1,j-1}^{K^{A}(L,M)}\right)}$  for every $i,j\geq 0$, we see that
$$\textstyle H_{i}\left(K^{A}(L,M)\right)\cong \bigoplus_{j=0}^{\infty}H_{i,j}\left(K^{A}(L,M)\right)$$
for every $i\geq 0$. Moreover, it is clear that $H_{i,j}\left(K^{A}(L,M)\right)$ is completely determined by $M_{j-1}$, $M_{j}$, and $M_{j+1}$ for every $i,j\geq 0$. Now let $n\geq 0$, and consider the two $\Sym_{A}\left(\frac{\mathfrak{H}_{n}}{\mathfrak{H}_{n}^{2}}\right)$-modules $M=\Sym_{A}\left(\frac{\mathfrak{H}_{n}}{\mathfrak{H}_{n}^{2}}\right)= \bigoplus_{j=0}^{\infty}\Sym_{A}^{j}\left(\frac{\mathfrak{H}_{n}}{\mathfrak{H}_{n}^{2}}\right)$ and $M^{(s)}= \bigoplus_{j=0}^{s-1}\Sym_{A}^{j}\left(\frac{\mathfrak{H}_{n}}{\mathfrak{H}_{n}^{2}}\right)$ for every $s\geq 1$. Let $s\geq 1$. As we just observed, $H_{i,j}\left(K^{A} \left(\frac{\mathfrak{H}_{n}}{\mathfrak{H}_{n}^{2}},M\right)\right)$ is completely determined by $M_{j-1}$, $M_{j}$, and $M_{j+1}$, and $H_{i,j}\left(K^{A} \left(\frac{\mathfrak{H}_{n}}{\mathfrak{H}_{n}^{2}},M^{(s)}\right)\right)$ is completely determined by $M_{j-1}^{(s)}$, $M_{j}^{(s)}$, and $M_{j+1}^{(s)}$. However, $M_{j}^{(s)}= \Sym_{A}^{j}\left(\frac{\mathfrak{H}_{n}}{\mathfrak{H}_{n}^{2}}\right)= M_{j}$ for every $j<s$. Thus we deduce that $H_{i,j}\left(K^{A} \left(\frac{\mathfrak{H}_{n}}{\mathfrak{H}_{n}^{2}},M^{(s)}\right)\right)= H_{i,j}\left(K^{A} \left(\frac{\mathfrak{H}_{n}}{\mathfrak{H}_{n}^{2}},M\right)\right)$ if $j+1<s$, i.e. $j<s-1$. But we have
$$\textstyle 0= \Tor_{i}^{\Sym_{A}\left(\frac{\mathfrak{H}_{n}}{\mathfrak{H}_{n}^{2}}\right)}(A,M)\cong H_{i}\left(K^{A}\left(\frac{\mathfrak{H}_{n}}{\mathfrak{H}_{n}^{2}},M\right)\right)\cong \bigoplus_{j=0}^{\infty}H_{i,j}\left(K^{A}\left(\frac{\mathfrak{H}_{n}}{\mathfrak{H}_{n}^{2}},M\right)\right)$$
for every $i\geq 1$, so we infer that $H_{i,j}\left(K^{A}\left(\frac{\mathfrak{H}_{n}}{\mathfrak{H}_{n}^{2}},M\right)\right)=0$ for every $i\geq 1$ and $j\geq 0$. Hence $H_{i,j}\left(K^{A}\left(\frac{\mathfrak{H}_{n}}{\mathfrak{H}_{n}^{2}},M^{(s)}\right)\right)=0$ for every $i\geq 1$ and $0\leq j< s-1$. On the other hand, $M_{j}^{(s)}=0$ for every $j>s-1$, so $H_{i,j}\left(K^{A}\left(\frac{\mathfrak{H}_{n}}{\mathfrak{H}_{n}^{2}},M^{(s)}\right)\right)=0$ for every $i\geq 0$ and $j>s-1$. It follows that
$$\textstyle H_{i}\left(K^{A}\left(\frac{\mathfrak{H}_{n}}{\mathfrak{H}_{n}^{2}},M^{(s)}\right)\right) \cong H_{i,s-1}\left(K^{A}\left(\frac{\mathfrak{H}_{n}}{\mathfrak{H}_{n}^{2}},M^{(s)}\right)\right)$$
for every $i\geq 1$. We observed above that:
$$\textstyle A\langle X\rangle_{n} \cong \Sym_{A}\left(\frac{\mathfrak{H}_{n}}{\mathfrak{H}_{n}^{2}}\right) = \bigoplus_{j=0}^{\infty}\Sym_{A}^{j}\left(\frac{\mathfrak{H}_{n}}{\mathfrak{H}_{n}^{2}}\right) \cong \bigoplus_{j=0}^{\infty}\frac{\mathfrak{H}_{n}^{j}}{\mathfrak{H}_{n}^{j+1}}$$
When we form $\frac{A\langle X\rangle_{n}}{\mathfrak{H}_{n}^{s}}$, we kill polynomials of degrees greater than or equal to $s$, so we notice that $\frac{A\langle X\rangle_{n}}{\mathfrak{H}_{n}^{s}}\cong \bigoplus_{j=0}^{s-1}\Sym_{A}^{j}\left(\frac{\mathfrak{H}_{n}}{\mathfrak{H}_{n}^{2}}\right) = M^{(s)}$. Let $\lambda_{s}:M^{(s)}\rightarrow M^{(s-1)}$ be the canonical projection. Then we have the following commutative diagram:
\begin{equation*}
  \begin{tikzcd}
  \frac{A\langle X\rangle_{n}}{\mathfrak{H}_{n}^{s}} \arrow{r}{\kappa_{s}} \arrow{d}[swap]{\cong} & \frac{A\langle X\rangle_{n}}{\mathfrak{H}_{n}^{s-1}} \arrow{d}{\cong}
  \\
  M^{(s)} \arrow{r}{\lambda_{s}} & M^{(s-1)}
\end{tikzcd}
\end{equation*}
This gives the following commutative diagram for every $i\geq 1$:
\begin{equation*}
  \begin{tikzcd}[column sep=9em,row sep=2em]
  \Tor_{i}^{A\langle X\rangle_{n}}\left(A,\frac{A\langle X\rangle_{n}}{\mathfrak{H}_{n}^{s}}\right) \arrow{r}{\Tor_{i}^{A\langle X\rangle_{n}}\left(A,\kappa_{s}\right)} \arrow{d}[swap]{\cong} & \Tor_{i}^{A\langle X\rangle_{n}}\left(A,\frac{A\langle X\rangle_{n}}{\mathfrak{H}_{n}^{s-1}}\right) \arrow{d}{\cong}
  \\
  \Tor_{i}^{\Sym_{A}\left(\frac{\mathfrak{H}_{n}}{\mathfrak{H}_{n}^{2}}\right)}\left(A,M^{(s)}\right) \arrow{r}{\Tor_{i}^{\Sym_{A}\left(\frac{\mathfrak{H}_{n}}{\mathfrak{H}_{n}^{2}}\right)}\left(A,\lambda_{s}\right)} \arrow{d}[swap]{\cong} & \Tor_{i}^{\Sym_{A}\left(\frac{\mathfrak{H}_{n}}{\mathfrak{H}_{n}^{2}}\right)}\left(A,M^{(s-1)}\right) \arrow{d}{\cong}
  \\
  H_{i}\left(K^{A}\left(\frac{\mathfrak{H}_{n}}{\mathfrak{H}_{n}^{2}},M^{(s)}\right)\right) \arrow{r}{H_{i}\left(K^{A}\left(\frac{\mathfrak{H}_{n}}{\mathfrak{H}_{n}^{2}},\lambda_{s}\right)\right)} & H_{i}\left(K^{A}\left(\frac{\mathfrak{H}_{n}}{\mathfrak{H}_{n}^{2}},M^{(s-1)}\right)\right)
\end{tikzcd}
\end{equation*}
Let $i\geq 1$. We note that $H_{i}\left(K^{A}\left(\frac{\mathfrak{H}_{n}}{\mathfrak{H}_{n}^{2}},\lambda_{s}\right)\right)$ is induced by $\lambda_{s}$. But we saw above that $H_{i}\left(K^{A}\left(\frac{\mathfrak{H}_{n}}{\mathfrak{H}_{n}^{2}},M^{(s)}\right)\right)$ has no component in degrees less than $s-1$, so we must have $H_{i}\left(K^{A}\left(\frac{\mathfrak{H}_{n}}{\mathfrak{H}_{n}^{2}},\lambda_{s}\right)\right)=0$. Therefore, the above diagram implies that $\Tor_{i}^{A\langle X\rangle_{n}}\left(A,\kappa_{s}\right)=0$.

Secondly, we show that $\pi_{k}\left(\Tor_{i}^{A\langle X\rangle}\left(A,\frac{A\langle X\rangle}{\mathfrak{H}^{s}}\right)\right)=0$ for every $i,s\geq 1$ and $0\leq k\leq s-1$. Let $n\geq 0$ and $i,s\geq 1$. We argue by induction on $k$. Let $k=0$. Since $\mathfrak{H}_{n}$ is a projective, hence flat $A$-module and $\pi_{0}(\mathfrak{H})=0$, we infer from Corollary \ref{03.4} that $\pi_{0}\left(\mathfrak{H}\otimes_{A}\mathfrak{H}^{j-1}\right)\cong \pi_{0}(\mathfrak{H})\otimes_{A}\pi_{0}\left(\mathfrak{H}^{j-1}\right)=0$ for every $j\geq 1$. For any $j\geq 1$, the natural $A$-epimorphism $\theta_{n}:\mathfrak{H}_{n}\otimes_{A}\mathfrak{H}_{n}^{j-1}\rightarrow \mathfrak{H}_{n}^{j}$, given by $\theta_{n}(p\otimes q)=pq$ for every $p\in \mathfrak{H}_{n}$ and $q\in \mathfrak{H}_{n}^{j-1}$, provides the following natural short exact sequence of $A$-modules:
$$0\rightarrow \Ker(\theta_{n})\rightarrow \mathfrak{H}_{n}\otimes_{A}\mathfrak{H}_{n}^{j-1}\xrightarrow{\theta_{n}} \mathfrak{H}_{n}^{j} \rightarrow 0$$
Since this short exact sequence is natural, we obtain the following short exact sequence of simplicial $A$-modules for every $j\geq 1$:
$$0\rightarrow \Ker(\theta) \rightarrow \mathfrak{H}\otimes_{A}\mathfrak{H}^{j-1} \rightarrow \mathfrak{H}^{j} \rightarrow 0$$
Thus we get the exact sequence
$$0= \pi_{0}\left(\mathfrak{H}\otimes_{A}\mathfrak{H}^{j-1}\right) \rightarrow \pi_{0}\left(\mathfrak{H}^{j}\right) \rightarrow 0,$$
implying that $\pi_{0}\left(\mathfrak{H}^{j}\right)=0$ for every $j\geq 1$. For any $j\geq 1$, the short exact sequence
$$\textstyle 0 \rightarrow \mathfrak{H}^{j+1} \rightarrow \mathfrak{H}^{j} \rightarrow \frac{\mathfrak{H}^{j}}{\mathfrak{H}^{j+1}} \rightarrow 0$$
of simplicial $A$-modules gives the exact sequence
$$\textstyle 0= \pi_{0}\left(\mathfrak{H}^{j}\right) \rightarrow \pi_{0}\left(\frac{\mathfrak{H}^{j}}{\mathfrak{H}^{j+1}}\right) \rightarrow 0,$$
so $\pi_{0}\left(\frac{\mathfrak{H}^{j}}{\mathfrak{H}^{j+1}}\right)=0$. We note that by Remark \ref{03.1}, the functor $\pi_{0}$ has a right adjoint, so it preserves direct limits, hence in particular, it preserves direct sums. Therefore, we get:
$$\textstyle \pi_{0}\left(\frac{A\langle X\rangle}{\mathfrak{H}^{s}}\right) \cong \pi_{0}\left(\bigoplus_{j=0}^{s-1}\frac{\mathfrak{H}^{j}}{\mathfrak{H}^{j+1}}\right) \cong \bigoplus_{j=0}^{s-1}\pi_{0}\left(\frac{\mathfrak{H}^{j}}{\mathfrak{H}^{j+1}}\right)=0$$
We observe that $\frac{A\langle X\rangle_{n}}{\mathfrak{H}_{n}^{s}} \cong \bigoplus_{j=0}^{s-1}\Sym_{A}^{j}\left(\frac{\mathfrak{H}_{n}}{\mathfrak{H}_{n}^{2}}\right) \cong \bigoplus_{j=0}^{s-1}\frac{\mathfrak{H}^{j}}{\mathfrak{H}^{j+1}}$ is a free, hence flat $A$-module, so we have the following natural $A$-isomorphisms:
$$\textstyle \Tor_{i}^{A\langle X\rangle_{n}}\left(A,\frac{A\langle X\rangle_{n}}{\mathfrak{H}_{n}^{s}}\right) \cong \Tor_{i}^{A\langle X\rangle_{n}}\left(A,A\otimes_{A}\frac{A\langle X\rangle_{n}}{\mathfrak{H}_{n}^{s}}\right) \cong \Tor_{i}^{A\langle X\rangle_{n}}(A,A)\otimes_{A}\frac{A\langle X\rangle_{n}}{\mathfrak{H}_{n}^{s}}$$
Since these isomorphisms are natural, we get the isomorphism
$$\textstyle \Tor_{i}^{A\langle X\rangle}\left(A,\frac{A\langle X\rangle}{\mathfrak{H}^{s}}\right) \cong \Tor_{i}^{A\langle X\rangle}(A,A)\otimes_{A}\frac{A\langle X\rangle}{\mathfrak{H}^{s}}$$
of simplicial $A$-modules. Then by using Corollary \ref{03.4}, we get:
$$\textstyle \pi_{0}\left(\Tor_{i}^{A\langle X\rangle}\left(A,\frac{A\langle X\rangle}{\mathfrak{H}^{s}}\right)\right) \cong \pi_{0}\left(\Tor_{i}^{A\langle X\rangle}(A,A)\otimes_{A}\frac{A\langle X\rangle}{\mathfrak{H}^{s}}\right) \cong \pi_{0}\left(\Tor_{i}^{A\langle X\rangle}(A,A)\right) \otimes_{A} \pi_{0}\left(\frac{A\langle X\rangle}{\mathfrak{H}^{s}}\right) = 0$$
Now suppose that $k\geq 1$, and the result holds for less than $k$. As $\frac{\mathfrak{H}_{n}^{s-1}}{\mathfrak{H}_{n}^{s}}$ is a free, hence flat $A$-module, and $\frac{A\langle X\rangle_{n}}{\mathfrak{H}_{n}} \cong A$, we have the following natural $A$-isomorphisms:
$$\textstyle \Tor_{i}^{A\langle X\rangle_{n}}\left(A,\frac{\mathfrak{H}_{n}^{s-1}}{\mathfrak{H}_{n}^{s}}\right) \cong \Tor_{i}^{A\langle X\rangle_{n}}\left(A,A\otimes_{A}\frac{\mathfrak{H}_{n}^{s-1}}{\mathfrak{H}_{n}^{s}}\right) \cong \Tor_{i}^{A\langle X\rangle_{n}}(A,A)\otimes_{A}\frac{\mathfrak{H}_{n}^{s-1}}{\mathfrak{H}_{n}^{s}} \cong \Tor_{i}^{A\langle X\rangle_{n}}\left(A,\frac{A\langle X\rangle_{n}}{\mathfrak{H}_{n}}\right)\otimes_{A}\frac{\mathfrak{H}_{n}^{s-1}}{\mathfrak{H}_{n}^{s}}$$
Thus we get the isomorphism
$$\textstyle \Tor_{i}^{A\langle X\rangle}\left(A,\frac{\mathfrak{H}^{s-1}}{\mathfrak{H}^{s}}\right) \cong \Tor_{i}^{A\langle X\rangle}\left(A,\frac{A\langle X\rangle}{\mathfrak{H}}\right)\otimes_{A}\frac{\mathfrak{H}^{s-1}}{\mathfrak{H}^{s}}$$
of simplicial $A$-modules. By the base case, we have $\pi_{0}\left(\Tor_{i}^{A\langle X\rangle}\left(A,\frac{A\langle X\rangle}{\mathfrak{H}}\right)\right)=0$, i.e. $\pi_{j}\left(\Tor_{i}^{A\langle X\rangle}\left(A,\frac{A\langle X\rangle}{\mathfrak{H}}\right)\right)=0$ for every $j\leq 0$. On the other hand, we have the following natural $A$-isomorphisms:
$$\textstyle \frac{\mathfrak{H}_{n}^{s-1}}{\mathfrak{H}_{n}^{s}} = \frac{\mathfrak{H}_{n}\cap \mathfrak{H}_{n}^{s-1}}{\mathfrak{H}_{n}\mathfrak{H}_{n}^{s-1}} \cong \Tor_{1}^{A\langle X\rangle_{n}}\left(\frac{A\langle X\rangle_{n}}{\mathfrak{H}_{n}},\frac{A\langle X\rangle_{n}}{\mathfrak{H}_{n}^{s-1}}\right) \cong \Tor_{1}^{A\langle X\rangle_{n}}\left(A,\frac{A\langle X\rangle_{n}}{\mathfrak{H}_{n}^{s-1}}\right)$$
Thus we get the following isomorphism of simplicial $A$-modules:
$$\textstyle \frac{\mathfrak{H}^{s-1}}{\mathfrak{H}^{s}} \cong \Tor_{1}^{A\langle X\rangle}\left(A,\frac{A\langle X\rangle}{\mathfrak{H}^{s-1}}\right)$$
By the induction hypothesis, $\pi_{j}\left(\frac{\mathfrak{H}^{s-1}}{\mathfrak{H}^{s}}\right) \cong \pi_{j}\left(\Tor_{1}^{A\langle X\rangle}\left(A,\frac{A\langle X\rangle}{\mathfrak{H}^{s-1}}\right)\right)=0$ for every $j\leq k-1$. Since $\frac{\mathfrak{H}_{n}}{\mathfrak{H}_{n}^{2}}$ is a free, hence flat $A$-module, it follows from Corollary \ref{03.5} that
$$\textstyle \pi_{j}\left(\Tor_{i}^{A\langle X\rangle}\left(A,\frac{\mathfrak{H}^{s-1}}{\mathfrak{H}^{s}}\right)\right) \cong \pi_{j}\left(\Tor_{i}^{A\langle X\rangle}\left(A,\frac{A\langle X\rangle}{\mathfrak{H}}\right)\otimes_{A}\frac{\mathfrak{H}^{s-1}}{\mathfrak{H}^{s}}\right)=0$$
for every $j\leq 0+k-1+1=k$. In particular, $\pi_{k}\left(\Tor_{i}^{A\langle X\rangle}\left(A,\frac{\mathfrak{H}^{s-1}}{\mathfrak{H}^{s}}\right)\right)=0$. Now the natural short exact sequence
$$\textstyle 0 \rightarrow \frac{\mathfrak{H}_{n}^{s-1}}{\mathfrak{H}_{n}^{s}} \rightarrow \frac{A\langle X\rangle_{n}}{\mathfrak{H}_{n}^{s}} \xrightarrow{\kappa_{s}} \frac{A\langle X\rangle_{n}}{\mathfrak{H}_{n}^{s-1}} \rightarrow 0$$
yields the following natural exact sequence:
$$\textstyle \Tor_{i+1}^{A\langle X\rangle_{n}}\left(A,\frac{A\langle X\rangle_{n}}{\mathfrak{H}_{n}^{s}}\right) \xrightarrow{\Tor_{i+1}^{A\langle X\rangle_{n}}\left(A,\kappa_{s}\right)} \Tor_{i+1}^{A\langle X\rangle_{n}}\left(A,\frac{A\langle X\rangle_{n}}{\mathfrak{H}_{n}^{s-1}}\right) \rightarrow \Tor_{i}^{A\langle X\rangle_{n}}\left(A,\frac{\mathfrak{H}_{n}^{s-1}}{\mathfrak{H}_{n}^{s}}\right) $$$$ \textstyle \rightarrow \Tor_{i}^{A\langle X\rangle_{n}}\left(A,\frac{A\langle X\rangle_{n}}{\mathfrak{H}_{n}^{s}}\right) \xrightarrow{\Tor_{i}^{A\langle X\rangle_{n}}\left(A,\kappa_{s}\right)} \Tor_{i}^{A\langle X\rangle_{n}}\left(A,\frac{A\langle X\rangle_{n}}{\mathfrak{H}_{n}^{s-1}}\right)$$
By the previous paragraph, $\Tor_{i+1}^{A\langle X\rangle_{n}}\left(A,\kappa_{s}\right) = 0 = \Tor_{i}^{A\langle X\rangle_{n}}\left(A,\kappa_{s}\right)$, so we get the following natural short exact sequence of $A$-modules:
$$\textstyle 0 \rightarrow \Tor_{i+1}^{A\langle X\rangle_{n}}\left(A,\frac{A\langle X\rangle_{n}}{\mathfrak{H}_{n}^{s-1}}\right) \rightarrow \Tor_{i}^{A\langle X\rangle_{n}}\left(A,\frac{\mathfrak{H}_{n}^{s-1}}{\mathfrak{H}_{n}^{s}}\right) \rightarrow \Tor_{i}^{A\langle X\rangle_{n}}\left(A,\frac{A\langle X\rangle_{n}}{\mathfrak{H}_{n}^{s}}\right) \rightarrow 0$$
Thus we obtain the following short exact sequence of simplicial $A$-modules:
$$\textstyle 0 \rightarrow \Tor_{i+1}^{A\langle X\rangle}\left(A,\frac{A\langle X\rangle}{\mathfrak{H}^{s-1}}\right) \rightarrow \Tor_{i}^{A\langle X\rangle}\left(A,\frac{\mathfrak{H}^{s-1}}{\mathfrak{H}^{s}}\right) \rightarrow \Tor_{i}^{A\langle X\rangle}\left(A,\frac{A\langle X\rangle}{\mathfrak{H}^{s}}\right) \rightarrow 0$$
Using the induction hypothesis, we get the following exact sequence:
$$\textstyle 0 = \pi_{k}\left(\Tor_{i}^{A\langle X\rangle}\left(A,\frac{\mathfrak{H}^{s-1}}{\mathfrak{H}^{s}}\right)\right) \rightarrow \pi_{k}\left(\Tor_{i}^{A\langle X\rangle}\left(A,\frac{A\langle X\rangle}{\mathfrak{H}^{s}}\right)\right) \rightarrow \pi_{k-1}\left(\Tor_{i+1}^{A\langle X\rangle}\left(A,\frac{A\langle X\rangle}{\mathfrak{H}^{s-1}}\right)\right) = 0$$
It follows that $\pi_{k}\left(\Tor_{i}^{A\langle X\rangle}\left(A,\frac{A\langle X\rangle}{\mathfrak{H}^{s}}\right)\right)=0$.

Thirdly, we show that $\pi_{k}\left(\Tor_{i}^{A\langle X\rangle}\left(\mathfrak{H},\mathfrak{H}^{s}\right)\right) = 0 = \pi_{k}\left(\Tor_{i}^{A\langle X\rangle}\left(A,\mathfrak{H}^{s}\right)\right)$ for every $i,s\geq 1$ and $0\leq k\leq s-1$. Let $n\geq 0$ and $i,s\geq 1$. The natural short exact sequence $$0 \rightarrow \mathfrak{H}_{n} \rightarrow A\langle X\rangle_{n} \xrightarrow{\ev_{n}} A \rightarrow 0$$
yields the following natural exact sequence:
$$0= \Tor_{i+1}^{A\langle X\rangle_{n}}\left(A\langle X\rangle_{n},\mathfrak{H}_{n}^{s}\right) \rightarrow \Tor_{i+1}^{A\langle X\rangle_{n}}\left(A,\mathfrak{H}_{n}^{s}\right) \rightarrow \Tor_{i}^{A\langle X\rangle_{n}}\left(\mathfrak{H}_{n},\mathfrak{H}_{n}^{s}\right) \rightarrow \Tor_{i}^{A\langle X\rangle_{n}}\left(A\langle X\rangle_{n},\mathfrak{H}_{n}^{s}\right) =0$$
Hence we get the natural $A$-isomorphism $\Tor_{i+1}^{A\langle X\rangle_{n}}\left(A,\mathfrak{H}_{n}^{s}\right) \cong \Tor_{i}^{A\langle X\rangle_{n}}\left(\mathfrak{H}_{n},\mathfrak{H}_{n}^{s}\right)$. On the other hand, the natural short exact sequence
$$\textstyle 0\rightarrow \mathfrak{H}_{n}^{s} \rightarrow A\langle X\rangle_{n} \rightarrow \frac{A\langle X\rangle_{n}}{\mathfrak{H}_{n}^{s}} \rightarrow 0$$
gives the following natural exact sequence:
$$\textstyle 0= \Tor_{i+1}^{A\langle X\rangle_{n}}\left(A,A\langle X\rangle_{n}\right) \rightarrow \Tor_{i+1}^{A\langle X\rangle_{n}}\left(A,\frac{A\langle X\rangle_{n}}{\mathfrak{H}_{n}^{s}}\right) \rightarrow \Tor_{i}^{A\langle X\rangle_{n}}\left(A,\mathfrak{H}_{n}^{s}\right) \rightarrow \Tor_{i}^{A\langle X\rangle_{n}}\left(A,A\langle X\rangle_{n}\right) =0$$
Thus we get the natural $A$-isomorphism $\Tor_{i+1}^{A\langle X\rangle_{n}}\left(A,\frac{A\langle X\rangle_{n}}{\mathfrak{H}_{n}^{s}}\right) \cong \Tor_{i}^{A\langle X\rangle_{n}}\left(A,\mathfrak{H}_{n}^{s}\right)$. As a result, we get the following natural $A$-isomorphisms:
$$\textstyle \Tor_{i}^{A\langle X\rangle_{n}}\left(\mathfrak{H}_{n},\mathfrak{H}_{n}^{s}\right) \cong \Tor_{i+1}^{A\langle X\rangle_{n}}\left(A,\mathfrak{H}_{n}^{s}\right) \cong \Tor_{i+2}^{A\langle X\rangle_{n}}\left(A,\frac{A\langle X\rangle_{n}}{\mathfrak{H}_{n}^{s}}\right)$$
Therefore, we obtain the isomorphisms $\Tor_{i}^{A\langle X\rangle}\left(\mathfrak{H},\mathfrak{H}^{s}\right) \cong \Tor_{i+2}^{A\langle X\rangle}\left(A,\frac{A\langle X\rangle}{\mathfrak{H}^{s}}\right)$ and $\Tor_{i}^{A\langle X\rangle}\left(A,\mathfrak{H}^{s}\right) \cong \Tor_{i+1}^{A\langle X\rangle}\left(A,\frac{A\langle X\rangle}{\mathfrak{H}^{s}}\right)$ of simplicial $A$-modules. Therefore, by the previous paragraph, we get $\pi_{k}\left(\Tor_{i}^{A\langle X\rangle}\left(\mathfrak{H},\mathfrak{H}^{s}\right)\right) \cong \pi_{k}\left(\Tor_{i+2}^{A\langle X\rangle}\left(A,\frac{A\langle X\rangle}{\mathfrak{H}^{s}}\right)\right)= 0$ and $\pi_{k}\left(\Tor_{i}^{A\langle X\rangle}\left(A,\mathfrak{H}^{s}\right)\right) \cong \pi_{k}\left(\Tor_{i+1}^{A\langle X\rangle}\left(A,\frac{A\langle X\rangle}{\mathfrak{H}^{s}}\right)\right)= 0$ for every $0\leq k\leq s-1$.

We are now ready to proceed with the inductive step. Let $s\geq 2$, and assume that the result holds for $s-1$, i.e. $\pi_{k}\left(\mathfrak{H}^{s-1}\right)=0$ for every $k\leq s-2$. Consider the bar construction $\Barr_{A}\left(\mathfrak{H}_{n},A\langle X\rangle_{n},\mathfrak{H}_{n}^{s-1}\right)$ for every $n\geq 0$. Let $n\geq 0$. It is clear that $A\langle X\rangle_{n} = A[X_{n}]$ is a free, hence flat $A$-algebra. Moreover, $\mathfrak{H}_{n}$ is a projective, hence flat $A$-module. As a result, by Remark \ref{03.1}, we have the following isomorphisms for every $i\geq 0$:
$$\Tor_{i}^{A\langle X\rangle_{n}}\left(\mathfrak{H}_{n},\mathfrak{H}_{n}^{s-1}\right) \cong \pi_{i}\left(\Barr_{A}\left(\mathfrak{H}_{n},A\langle X\rangle_{n},\mathfrak{H}_{n}^{s-1}\right)\right) \cong H_{i}\left(C_{A}\left(\Barr_{A}\left(\mathfrak{H}_{n},A\langle X\rangle_{n},\mathfrak{H}_{n}^{s-1}\right)\right)\right)$$
As a result, we get the following short exact sequence of $A$-modules for every $i\geq 0$:
$$0 \rightarrow \im\left(\partial_{i+1}^{C_{A}\left(\Barr_{A}\left(\mathfrak{H}_{n},A\langle X\rangle_{n},\mathfrak{H}_{n}^{s-1}\right)\right)}\right) \rightarrow \Ker\left(\partial_{i}^{C_{A}\left(\Barr_{A}\left(\mathfrak{H}_{n},A\langle X\rangle_{n},\mathfrak{H}_{n}^{s-1}\right)\right)}\right) \rightarrow \Tor_{i}^{A\langle X\rangle_{n}}\left(\mathfrak{H}_{n},\mathfrak{H}_{n}^{s-1}\right) \rightarrow 0$$
Since this short exact sequence is natural, we get the following short exact sequence of simplicial $A$-modules for every $i\geq 0$:
$$0 \rightarrow \im\left(\partial_{i+1}^{C_{A}\left(\Barr_{A}\left(\mathfrak{H},A\langle X\rangle,\mathfrak{H}^{s-1}\right)\right)}\right) \rightarrow \Ker\left(\partial_{i}^{C_{A}\left(\Barr_{A}\left(\mathfrak{H},A\langle X\rangle,\mathfrak{H}^{s-1}\right)\right)}\right) \rightarrow \Tor_{i}^{A\langle X\rangle}\left(\mathfrak{H},\mathfrak{H}^{s-1}\right) \rightarrow 0$$
Using what was proved in the previous paragraph, we get the following exact sequence for every $i\geq 1$ and $0\leq k\leq s-2$:
\small
$$\pi_{k}\left(\im\left(\partial_{i+1}^{C_{A}\left(\Barr_{A}\left(\mathfrak{H},A\langle X\rangle,\mathfrak{H}^{s-1}\right)\right)}\right)\right) \rightarrow \pi_{k}\left(\Ker\left(\partial_{i}^{C_{A}\left(\Barr_{A}\left(\mathfrak{H},A\langle X\rangle,\mathfrak{H}^{s-1}\right)\right)}\right)\right) \rightarrow \pi_{k}\left(\Tor_{i}^{A\langle X\rangle}\left(\mathfrak{H},\mathfrak{H}^{s-1}\right)\right)=0$$
\normalsize
We have for every $i\geq 0$:
$$C_{A}\left(\Barr_{A}\left(\mathfrak{H},A\langle X\rangle,\mathfrak{H}^{s-1}\right)\right)_{i} = \Barr_{A}\left(\mathfrak{H},A\langle X\rangle,\mathfrak{H}^{s-1}\right)_{i} = \mathfrak{H}\otimes_{A}T^{i}_{A}\left(A\langle X\rangle\right)\otimes_{A}\mathfrak{H}^{s-1}$$
Since $\mathfrak{H}_{n}$ is a projective, hence flat $A$-module and $\pi_{0}(\mathfrak{H})=0$, by Corollary \ref{03.4}, we have $\pi_{0}\left(\mathfrak{H}\otimes_{A} T_{A}^{i}\left(A\langle X\rangle\right)\right) \cong \pi_{0}(\mathfrak{H}) \otimes_{A} \pi_{0}\left(T_{A}^{i}\left(A\langle X\rangle)\right)\right)=0$, i.e. $\pi_{j}\left(\mathfrak{H}\otimes_{A} T_{A}^{i}\left(A\langle X\rangle\right)\right)=0$ for every $j\leq 0$ and $i\geq 0$. On the other hand, $\pi_{j}(\mathfrak{H}^{s-1})=0$ for every $j\leq s-2$ by the induction hypothesis. Since $\mathfrak{H}_{n}^{s}$ is a projective, hence flat $A$-module, it follows from Corollary \ref{03.5} that
$$\pi_{j}\left(C_{A}\left(\Barr_{A}\left(\mathfrak{H},A\langle X\rangle,\mathfrak{H}^{s-1}\right)\right)_{i}\right) = \pi_{j}\left(\mathfrak{H}\otimes_{A}T^{i}_{A}\left(A\langle X\rangle\right)\otimes_{A}\mathfrak{H}^{s-1}\right) = 0$$
for every $j\leq 0+s-2+1=s-1$ and $i\geq 0$. Now consider the following short exact sequence of simplicial $A$-modules for every $i\geq 0$:
\small
$$0 \rightarrow \Ker\left(\partial_{i}^{C_{A}\left(\Barr_{A}\left(\mathfrak{H},A\langle X\rangle,\mathfrak{H}^{s-1}\right)\right)}\right) \rightarrow C_{A}\left(\Barr_{A}\left(\mathfrak{H},A\langle X\rangle,\mathfrak{H}^{s-1}\right)\right)_{i} \rightarrow \im\left(\partial_{i}^{C_{A}\left(\Barr_{A}\left(\mathfrak{H},A\langle X\rangle,\mathfrak{H}^{s-1}\right)\right)}\right) \rightarrow 0$$
\normalsize
We get the following exact sequence for every $i\geq 0$ and $k\leq s-2$:
$$0= \pi_{k+1}\left(C_{A}\left(\Barr_{A}\left(\mathfrak{H},A\langle X\rangle,\mathfrak{H}^{s-1}\right)\right)_{i}\right) \rightarrow \pi_{k+1}\left(\im\left(\partial_{i}^{C_{A}\left(\Barr_{A}\left(\mathfrak{H},A\langle X\rangle,\mathfrak{H}^{s-1}\right)\right)}\right)\right) $$$$ \rightarrow \pi_{k}\left(\Ker\left(\partial_{i}^{C_{A}\left(\Barr_{A}\left(\mathfrak{H},A\langle X\rangle,\mathfrak{H}^{s-1}\right)\right)}\right)\right)$$
So far, we have obtained the following two exact sequences for every $i\geq 1$ and $k\leq s-2$:
$$\pi_{k}\left(\im\left(\partial_{i+1}^{C_{A}\left(\Barr_{A}\left(\mathfrak{H},A\langle X\rangle,\mathfrak{H}^{s-1}\right)\right)}\right)\right) \rightarrow \pi_{k}\left(\Ker\left(\partial_{i}^{C_{A}\left(\Barr_{A}\left(\mathfrak{H},A\langle X\rangle,\mathfrak{H}^{s-1}\right)\right)}\right)\right) \rightarrow 0$$
and
$$0\rightarrow \pi_{k+1}\left(\im\left(\partial_{i}^{C_{A}\left(\Barr_{A}\left(\mathfrak{H},A\langle X\rangle,\mathfrak{H}^{s-1}\right)\right)}\right)\right) \rightarrow \pi_{k}\left(\Ker\left(\partial_{i}^{C_{A}\left(\Barr_{A}\left(\mathfrak{H},A\langle X\rangle,\mathfrak{H}^{s-1}\right)\right)}\right)\right)$$
These sequences conspire to show that if $i\geq 1$ and $k\leq s-2$, then the following implication holds:
$$\pi_{k}\left(\im\left(\partial_{i+1}^{C_{A}\left(\Barr_{A}\left(\mathfrak{H},A\langle X\rangle,\mathfrak{H}^{s-1}\right)\right)}\right)\right)=0 \Longrightarrow \pi_{k+1}\left(\im\left(\partial_{i}^{C_{A}\left(\Barr_{A}\left(\mathfrak{H},A\langle X\rangle,\mathfrak{H}^{s-1}\right)\right)}\right)\right)=0$$
If $i\geq 0$, then $\pi_{0}\left(C_{A}\left(\Barr_{A}\left(\mathfrak{H},A\langle X\rangle,\mathfrak{H}^{s-1}\right)\right)_{i+1}\right)=0$, so the short exact sequence
\small
$$0 \rightarrow \Ker\left(\partial_{i+1}^{C_{A}\left(\Barr_{A}\left(\mathfrak{H},A\langle X\rangle,\mathfrak{H}^{s-1}\right)\right)}\right) \rightarrow C_{A}\left(\Barr_{A}\left(\mathfrak{H},A\langle X\rangle,\mathfrak{H}^{s-1}\right)\right)_{i+1} \rightarrow \im\left(\partial_{i+1}^{C_{A}\left(\Barr_{A}\left(\mathfrak{H},A\langle X\rangle,\mathfrak{H}^{s-1}\right)\right)}\right) \rightarrow 0$$
\normalsize
gives the following exact sequence:
$$0= \pi_{0}\left(C_{A}\left(\Barr_{A}\left(\mathfrak{H},A\langle X\rangle,\mathfrak{H}^{s-1}\right)\right)_{i+1}\right) \rightarrow \pi_{0}\left(\im\left(\partial_{i+1}^{C_{A}\left(\Barr_{A}\left(\mathfrak{H},A\langle X\rangle,\mathfrak{H}^{s-1}\right)\right)}\right)\right) \rightarrow 0$$
Thus $\pi_{0}\left(\im\left(\partial_{i+1}^{C_{A}\left(\Barr_{A}\left(\mathfrak{H},A\langle X\rangle,\mathfrak{H}^{s-1}\right)\right)}\right)\right) = 0$ for every $i\geq 0$. Using the aforementioned implication repeatedly, we end up with $$\pi_{i}\left(\im\left(\partial_{1}^{C_{A}\left(\Barr_{A}\left(\mathfrak{H},A\langle X\rangle,\mathfrak{H}^{s-1}\right)\right)}\right)\right) = 0$$
for every $i\leq s-1$. Consider the following short exact sequence:
$$0 \rightarrow \im\left(\partial_{1}^{C_{A}\left(\Barr_{A}\left(\mathfrak{H},A\langle X\rangle,\mathfrak{H}^{s-1}\right)\right)}\right) \rightarrow \Ker\left(\partial_{0}^{C_{A}\left(\Barr_{A}\left(\mathfrak{H},A\langle X\rangle,\mathfrak{H}^{s-1}\right)\right)}\right) \rightarrow \Tor_{0}^{A\langle X\rangle}\left(\mathfrak{H},\mathfrak{H}^{s-1}\right) \rightarrow 0$$
We have
\begin{equation*}
\begin{split}
\Ker\left(\partial_{0}^{C_{A}\left(\Barr_{A}\left(\mathfrak{H},A\langle X\rangle,\mathfrak{H}^{s-1}\right)\right)}\right) & = C_{A}\left(\Barr_{A}\left(\mathfrak{H},A\langle X\rangle,\mathfrak{H}^{s-1}\right)\right)_{0} = \mathfrak{H}\otimes_{A}T^{0}_{A}\left(A\langle X\rangle\right)\otimes_{A}\mathfrak{H}^{s-1} \\
& = \mathfrak{H}\otimes_{A}A\otimes_{A}\mathfrak{H}^{s-1} \cong \mathfrak{H}\otimes_{A}\mathfrak{H}^{s-1}
\end{split}
\end{equation*}
and $\Tor_{0}^{A\langle X\rangle}\left(\mathfrak{H},\mathfrak{H}^{s-1}\right) \cong \mathfrak{H}\otimes_{A\langle X\rangle}\mathfrak{H}^{s-1}$, so we obtain the following short exact sequence:
$$0 \rightarrow \im\left(\partial_{1}^{C_{A}\left(\Barr_{A}\left(\mathfrak{H},A\langle X\rangle,\mathfrak{H}^{s-1}\right)\right)}\right) \rightarrow \mathfrak{H}\otimes_{A}\mathfrak{H}^{s-1} \rightarrow \mathfrak{H}\otimes_{A\langle X\rangle}\mathfrak{H}^{s-1} \rightarrow 0$$
By the base case, $\pi_{0}(\mathfrak{H})=0$, i.e. $\pi_{k}(\mathfrak{H})=0$ for every $k\leq 0$, and by the induction hypothesis, $\pi_{k}\left(\mathfrak{H}^{s-1}\right)=0$ for every $k\leq s-2$. As $\mathfrak{H}_{n}$ is a projective, hence flat $A$-module, it follows from Corollary \ref{03.5} that $\pi_{k}\left(\mathfrak{H}\otimes_{A}\mathfrak{H}^{s-1}\right)=0$ for every $k\leq 0+s-2+1=s-1$. Therefore, we get the following exact sequence for every $k\leq s-1$:
$$0 = \pi_{k}\left(\mathfrak{H}\otimes_{A}\mathfrak{H}^{s-1}\right) \rightarrow  \pi_{k}\left(\mathfrak{H}\otimes_{A\langle X\rangle}\mathfrak{H}^{s-1}\right) \rightarrow \pi_{k-1}\left(\im\left(\partial_{1}^{C_{A}\left(\Barr_{A}\left(\mathfrak{H},A\langle X\rangle,\mathfrak{H}^{s-1}\right)\right)}\right)\right)=0$$
Hence $\pi_{k}\left(\mathfrak{H}\otimes_{A\langle X\rangle}\mathfrak{H}^{s-1}\right)=0$ for every $k\leq s-1$. We have the following short exact sequence of $A$-modules:
$$\textstyle 0 \rightarrow \Tor_{1}^{A\langle X\rangle_{n}}\left(\frac{A\langle X\rangle_{n}}{\mathfrak{H}_{n}},\mathfrak{H}_{n}^{s-1}\right) \rightarrow \mathfrak{H}_{n}\otimes_{A\langle X\rangle_{n}}\mathfrak{H}_{n}^{s-1} \rightarrow \mathfrak{H}_{n}\mathfrak{H}_{n}^{s-1} \rightarrow 0$$
But $\frac{A\langle X\rangle_{n}}{\mathfrak{H}_{n}} \cong A$ and $\mathfrak{H}_{n}\mathfrak{H}_{n}^{s-1} = \mathfrak{H}_{n}^{s}$, so we get the following short exact sequence of $A$-modules:
$$0 \rightarrow \Tor_{1}^{A\langle X\rangle_{n}}\left(A,\mathfrak{H}_{n}^{s-1}\right) \rightarrow \mathfrak{H}_{n}\otimes_{A\langle X\rangle_{n}}\mathfrak{H}_{n}^{s-1} \rightarrow \mathfrak{H}_{n}^{s} \rightarrow 0$$
Since this short exact sequence is natural, we get the following short exact sequence of simplicial $A$-modules:
$$0 \rightarrow \Tor_{1}^{A\langle X\rangle}\left(A,\mathfrak{H}^{s-1}\right) \rightarrow \mathfrak{H}\otimes_{A\langle X\rangle}\mathfrak{H}^{s-1} \rightarrow \mathfrak{H}^{s} \rightarrow 0$$
Therefore, using what was proved in the previous paragraph, we get the following exact sequence for every $k\leq s-1$:
$$0 = \pi_{k}\left(\mathfrak{H}\otimes_{A\langle X\rangle}\mathfrak{H}^{s-1}\right) \rightarrow \pi_{k}\left(\mathfrak{H}^{s}\right) \rightarrow \pi_{k-1}\left(\Tor_{1}^{A\langle X\rangle}\left(A,\mathfrak{H}^{s-1}\right)\right) = 0$$
Hence $\pi_{k}\left(\mathfrak{H}^{s}\right)=0$ for every $k\leq s-1$.
\end{proof}

\section{Fundamental Spectral Sequences}

In this section, we provide complete proofs of Quillen's fundamental spectral sequences. We start with a remark on cotangent complex and Andr\'{e}-Quillen homology and cohomology.

\begin{remark} \label{04.1}
Let $R$ be a commutative ring, and $A$ a commutative $R$-algebra. Then there is a Quillen adjunction as follows:
$$\left(\Omega_{-|R}\otimes_{-}A,A\ltimes -\right):\mathpzc{s}\mathcal{C}\mathpzc{om}\mathcal{A}\mathpzc{lg}(R)/A \leftrightarrows \mathpzc{s}\mathcal{M}\mathpzc{od}(A)$$
See \cite[Proposition 1.7]{Qu1} and \cite[Proposition 4.27]{GS}. By the general theory of derived functors in model categories, the left derived functor of $\Omega_{-|R}\otimes_{-}A:\mathpzc{s}\mathcal{C}\mathpzc{om}\mathcal{A}\mathpzc{lg}(R)/A \rightarrow \mathpzc{s}\mathcal{M}\mathpzc{od}(A)$ exists; see \cite[Theorem 8.5.8]{Hi} and \cite[Lemma 1.3.10]{Ho}. Also, the right derived functor of $\Nor_{A}(-):\mathpzc{s}\mathcal{M}\mathpzc{od}(A)\rightarrow \mathcal{C}_{\geq 0}(A)$ exists. Let $\mathcal{I}: \mathcal{C}_{\geq 0}(A) \rightarrow \mathcal{C}(A)$ be the inclusion functor, and consider the following functors:
$$\Ho\left(\mathpzc{s}\mathcal{C}\mathpzc{om}\mathcal{A}\mathpzc{lg}(R)/A\right) \xrightarrow{\textrm{L}\left(\Omega_{-|R}\otimes_{-}A\right)} \Ho\left(\mathpzc{s}\mathcal{M}\mathpzc{od}(A)\right) \xrightarrow{\textrm{R}\Nor_{A}(-)} \Ho\left(\mathcal{C}_{\geq 0}(A)\right) \xrightarrow{\Ho(\mathcal{I})} \mathcal{D}(A)$$
Then the cotangent complex of $A$ over $R$ is defined as follows:
$$\mathbb{L}^{A|R}:= \Ho(\mathcal{I})\left(\textrm{R}\Nor_{A}\left(\textrm{L}\left(\Omega_{-|R}\otimes_{-}A\right)\left(1^{A}\right)\right)\right)$$
More concretely, if $\rho_{A}:A_{\textrm{c}}\rightarrow A$ is any simplicial algebra resolution of $A$, then we have:
$$\mathbb{L}^{A|R}= \Nor_{A}\left(\Omega_{A_{\textrm{c}}|R}\otimes_{A_{\textrm{c}}}A\right)$$
Now if $M$ is an $A$-module, then for any $i\in \mathbb{Z}$, the $i$th Andr\'{e}-Quillen homology of $A$ over $R$ with coefficients in $M$ is defined as
$$H_{i}^{AQ}(A|R;M):= \Tor_{i}^{A}\left(\mathbb{L}^{A|R},M\right)= H_{i}\left(\mathbb{L}^{A|R}\otimes_{A}M\right),$$
and the $i$th Andr\'{e}-Quillen cohomology of $A$ over $R$ with coefficients in $M$ is defined as
$$H_{AQ}^{i}(A|R;M):= \Ext_{A}^{i}\left(\mathbb{L}^{A|R},M\right)= H_{-i}\left(\Hom_{A}\left(\mathbb{L}^{A|R},M\right)\right).$$
\end{remark}

Now we are ready to present Quillen's fundamental spectral sequences. Note that a sketchy argument with gaps for a special case of the first spectral sequence appears in an unpublished manuscript of Quillen; see \cite[Theorem 6.8]{Qu1}. Also, the second spectral sequence is only stated without proof in another
paper of Quillen; see \cite[Theorem 6.8]{Qu3}.

\begin{theorem} \label{04.2}
Let $R$ be a ring, $A$ a commutative $R$-algebra, and $M$ an $A$-module. Let $\rho:R\langle X\rangle \rightarrow A$ be a simplicial algebra resolution of $A$. Suppose that the natural $R$-algebra epimorphism $\mu_{A}:A\otimes_{R}A\rightarrow A$, given by $\mu_{A} (a\otimes b)=ab$ for every $a,b\in A$, is an isomorphism. Then the following assertions hold:
\begin{enumerate}
\item[(i)] There is a first quadrant spectral sequence as follows:
$$E_{p,q}^{2}= H_{p+q}\left(\Nor_{A}\left(\Sym_{A}^{q}\left(\Omega_{R\langle X\rangle \mid R}\otimes_{R\langle X\rangle}A\right)\right)\otimes_{A}M\right) \Rightarrow \Tor_{p+q}^{R}(A,M)$$
\item[(ii)] There is a third quadrant spectral sequence as follows:
$$E_{p,q}^{2}= H_{p+q}\left(\Hom_{A}\left(\Nor_{A}\left(\Sym_{A}^{-q}\left(\Omega_{R\langle X\rangle \mid R}\otimes_{R\langle X\rangle}A\right)\right),M\right)\right) \Rightarrow \Ext_{R}^{-p-q}(A,M)$$
\end{enumerate}
\end{theorem}

\begin{proof}
We first note that the spectral sequences are independent of the simplicial algebra resolution chosen since $\Omega_{R\langle X\rangle |R} \otimes_{R\langle X\rangle}A = \textrm{L}\left(\Omega_{- |R} \otimes_{-}A\right)\left(1^{A}\right)$. Consider the morphism $\varphi:R\langle X\rangle \otimes_{R}A \rightarrow A$ of simplicial $A$-algebras, given by $\varphi_{n}(p\otimes a)= \rho_{n}(p)a$ for every $n\geq 0$, $p\in R\langle X\rangle_{n}$, and $a\in A$. Set $\mathfrak{K}= \Ker(\varphi)$ regarded as a simplicial ideal in the simplicial $A$-algebra $R\langle X\rangle \otimes_{R}A$. This gives a filtration
$$\cdots \subseteq \mathfrak{K}^{s+1} \subseteq \mathfrak{K}^{s} \subseteq \cdots \subseteq \mathfrak{K}^{2} \subseteq \mathfrak{K} \subseteq R\langle X\rangle \otimes_{R}A$$
of the simplicial $A$-module $R\langle X\rangle \otimes_{R}A$. Since the normalization functor is exact, after applying the normalization functor and making the required identifications, we can assume that there is a filtration
$$\cdots \subseteq \Nor_{A}(\mathfrak{K}^{s+1}) \subseteq \Nor_{A}(\mathfrak{K}^{s}) \subseteq \cdots \subseteq \Nor_{A}(\mathfrak{K}^{2}) \subseteq \Nor_{A}(\mathfrak{K}) \subseteq \Nor_{A}\left(R\langle X\rangle \otimes_{R}A\right)$$
of the $A$-complex $\Nor_{A}\left(R\langle X\rangle \otimes_{R}A\right)$. Furthermore, we observed in the proof of Theorem \ref{03.6} that for any $s,n\geq 0$, $\frac{\mathfrak{K}_{n}^{s}}{\mathfrak{K}_{n}^{s+1}}$ is a projective $A$-module, so its direct summand $\Nor_{A}\left(\frac{\mathfrak{K}^{s}}{\mathfrak{K}^{s+1}}\right)_{n}$ is a projective $A$-module as well; see \cite[Lemma 8.3.7]{We} or \cite[Remark 4.6]{Fa}. Thus for any $s\geq 0$, the short exact sequence $0\rightarrow \mathfrak{K}^{s+1}\rightarrow \mathfrak{K}^{s}\rightarrow \frac{\mathfrak{K}^{s}}{\mathfrak{K}^{s+1}}\rightarrow 0$ of simplicial $A$-modules yields the degreewise split short exact sequence
$$\textstyle 0\rightarrow \Nor_{A}(\mathfrak{K}^{s+1})\rightarrow \Nor_{A}(\mathfrak{K}^{s})\rightarrow \Nor_{A}\left(\frac{\mathfrak{K}^{s}}{\mathfrak{K}^{s+1}}\right)\rightarrow 0$$
of $A$-complexes. Then:

(i): For any $s\geq 0$, the degreewise split short exact sequence
$$\textstyle 0\rightarrow \Nor_{A}(\mathfrak{K}^{s+1})\rightarrow \Nor_{A}(\mathfrak{K}^{s})\rightarrow \Nor_{A}\left(\frac{\mathfrak{K}^{s}}{\mathfrak{K}^{s+1}}\right)\rightarrow 0$$
of $A$-complexes gives the degreewise split short exact sequence
$$\textstyle 0\rightarrow \Nor_{A}(\mathfrak{K}^{s+1})\otimes_{A}M \rightarrow \Nor_{A}(\mathfrak{K}^{s})\otimes_{A}M \rightarrow \Nor_{A}\left(\frac{\mathfrak{K}^{s}}{\mathfrak{K}^{s+1}}\right)\otimes_{A}M \rightarrow 0$$
of $A$-complexes. After performing the required identifications, we can assume that $\Nor_{A}(\mathfrak{K}^{s+1})\otimes_{A}M \subseteq \Nor_{A}(\mathfrak{K}^{s})\otimes_{A}M$ and $\frac{\Nor_{A}(\mathfrak{K}^{s})\otimes_{A}M}{\Nor_{A}(\mathfrak{K}^{s+1})\otimes_{A}M} \cong \Nor_{A}\left(\frac{\mathfrak{K}^{s}}{\mathfrak{K}^{s+1}}\right)\otimes_{A}M$. This gives a filtration
\small
$$\cdots \subseteq \Nor_{A}(\mathfrak{K}^{s+1})\otimes_{A}M \subseteq \Nor_{A}(\mathfrak{K}^{s})\otimes_{A}M \subseteq \cdots \subseteq \Nor_{A}(\mathfrak{K}^{2})\otimes_{A}M \subseteq \Nor_{A}(\mathfrak{K})\otimes_{A}M \subseteq \Nor_{A}\left(R\langle X\rangle \otimes_{R}A\right)\otimes_{A}M$$
\normalsize
of the $A$-complex $\Nor_{A}\left(R\langle X\rangle \otimes_{R}A\right)\otimes_{A}M$. Setting $F^{p}= \Nor_{A}(\mathfrak{K}^{-p})\otimes_{A}M$ for every $p\in \mathbb{Z}$, the above filtration takes the following standard form:
$$\cdots \subseteq F^{p-1} \subseteq F^{p} \subseteq \cdots \subseteq F^{-1} \subseteq F^{0} = \Nor_{A}\left(R\langle X\rangle \otimes_{R}A\right)\otimes_{A}M$$

We next show that $H_{k}\left(\Nor_{A}(\mathfrak{K}^{s})\otimes_{A}M\right)=0$ for every $s\geq 0$ and $k\leq s-1$. Fix $s\geq 0$. Following the ideas of \cite[Chapter II, Lemme 18]{An}, we show that $H_{k}\left(\Nor_{A}(\mathfrak{K}^{s})\otimes_{A}M\right) \cong H_{k}\left(\Nor_{A}(\mathfrak{K}^{s})\right)\otimes_{A}M$ for every $A$-module $M$ and $k\leq s-1$. For $k<0$, both sides are zero, so we may assume that $0\leq k\leq s-1$. We argue by induction on $k$. Let $k=0$. Then we have:
\begin{equation*}
\begin{split}
 H_{0}\left(\Nor_{A}(\mathfrak{K}^{s})\otimes_{A}M\right) & = \Coker\left(\partial_{1}^{\Nor_{A}(\mathfrak{K}^{s})\otimes_{A}M}\right) = \Coker\left(\partial_{1}^{\Nor_{A}(\mathfrak{K}^{s})}\otimes_{A}M\right) \cong \Coker\left(\partial_{1}^{\Nor_{A}(\mathfrak{K}^{s})}\right)\otimes_{A}M \\
 & = H_{0}\left(\Nor_{A}(\mathfrak{K}^{s})\right)\otimes_{A}M
\end{split}
\end{equation*}
Now suppose that $k\geq 1$, and the result holds for $k-1$. Let $0\rightarrow N' \rightarrow N \rightarrow N'' \rightarrow 0$ be a short exact sequence of $A$-modules. Since $\Nor_{A}(\mathfrak{K}^{s})$ is an $A$-complex of projective, hence flat modules, by \cite[Example 2.4.17]{CFH}, we get the short exact sequence
$$0\rightarrow \Nor_{A}(\mathfrak{K}^{s})\otimes_{A}N' \rightarrow \Nor_{A}(\mathfrak{K}^{s})\otimes_{A}N \rightarrow \Nor_{A}(\mathfrak{K}^{s})\otimes_{A}N'' \rightarrow 0$$
of $A$-complexes, which gives the following exact sequence:
$$H_{k}\left(\Nor_{A}(\mathfrak{K}^{s})\otimes_{A}N'\right) \rightarrow H_{k}\left(\Nor_{A}(\mathfrak{K}^{s})\otimes_{A}N\right) \rightarrow H_{k}\left(\Nor_{A}(\mathfrak{K}^{s})\otimes_{A}N''\right) \rightarrow H_{k-1}\left(\Nor_{A}(\mathfrak{K}^{s})\otimes_{A}N'\right)$$
By the induction hypothesis, $H_{k-1}\left(\Nor_{A}(\mathfrak{K}^{s})\otimes_{A}M\right) \cong H_{k-1}\left(\Nor_{A}(\mathfrak{K}^{s})\right)\otimes_{A}M$. Since $k-1< k \leq s-1$, Theorem \ref{03.6} implies that $H_{k-1}\left(\Nor_{A}(\mathfrak{K}^{s})\right)= \pi_{k-1}(\mathfrak{K}^{s})=0$, so $H_{k-1}\left(\Nor_{A}(\mathfrak{K}^{s})\otimes_{A}M\right)=0$. Therefore, the above exact sequence turns into the following exact sequence:
$$H_{k}\left(\Nor_{A}(\mathfrak{K}^{s})\otimes_{A}N'\right) \rightarrow H_{k}\left(\Nor_{A}(\mathfrak{K}^{s})\otimes_{A}N\right) \rightarrow H_{k}\left(\Nor_{A}(\mathfrak{K}^{s})\otimes_{A}N''\right) \rightarrow 0$$
This in particular shows that the functor $H_{k}\left(\Nor_{A}(\mathfrak{K}^{s})\otimes_{A}- \right):\mathcal{M}\mathpzc{od}(A) \rightarrow \mathcal{M}\mathpzc{od}(A)$ is right exact. On the other hand, the covariant functor $-\otimes_{A}M:\mathcal{M}\mathpzc{od}(A) \rightarrow \mathcal{M}\mathpzc{od}(A)$ is right exact, so there is a natural $A$-homomorphism $\chi_{M}:H_{k}\left(\Nor_{A}(\mathfrak{K}^{s})\right)\otimes_{A}M \rightarrow H_{k}\left(\Nor_{A}(\mathfrak{K}^{s})\otimes_{A}M\right)$. Moreover, if $F$ is a flat $A$-module, then the functor $-\otimes_{A}F:\mathcal{M}\mathpzc{od}(A) \rightarrow \mathcal{M}\mathpzc{od}(A)$ is exact, so we conclude that $\chi_{F}$ is an isomorphism; see \cite[2.2.19]{CFH}. Now let $F_{1}\rightarrow F_{0}\rightarrow M\rightarrow 0$ be an exact sequence in which $F_{0}$ and $F_{1}$ are flat $A$-modules. Then we get the following commutative diagram with exact rows:
\begin{equation*}
\begin{tikzcd}
  H_{k}\left(\Nor_{A}(\mathfrak{K}^{s})\right)\otimes_{A}F_{1} \arrow{r} \arrow{d}{\cong}[swap]{\chi_{F_{1}}} & H_{k}\left(\Nor_{A}(\mathfrak{K}^{s})\right)\otimes_{A}F_{0} \arrow{r} \arrow{d}{\cong}[swap]{\chi_{F_{0}}} & H_{k}\left(\Nor_{A}(\mathfrak{K}^{s})\right)\otimes_{A}M \arrow{d}[swap]{\chi_{M}} \arrow{r} & 0
  \\
  H_{k}\left(\Nor_{A}(\mathfrak{K}^{s})\otimes_{A}F_{1}\right) \arrow{r} &
  H_{k}\left(\Nor_{A}(\mathfrak{K}^{s})\otimes_{A}F_{0}\right) \arrow{r} &
  H_{k}\left(\Nor_{A}(\mathfrak{K}^{s})\otimes_{A}M\right) \arrow{r} & 0
\end{tikzcd}
\end{equation*}
It follows that $\chi_{M}$ is an isomorphism. Now since $k\leq s-1$, Theorem \ref{03.6} implies that $H_{k}\left(\Nor_{A}(\mathfrak{K}^{s})\right)= \pi_{k}(\mathfrak{K}^{s})=0$, so $H_{k}\left(\Nor_{A}(\mathfrak{K}^{s})\otimes_{A}M\right)=0$.

We showed that $H_{k}\left(\Nor_{A}(\mathfrak{K}^{s})\otimes_{A}M\right)=0$ for every $s\geq 0$ and $k\leq s-1$, which is equivalent to $H_{n}(F^{p})=0$ for every $p\leq -n-1$. Therefore, Corollary \ref{02.3} implies that there is a spectral sequence as follows:
$$\textstyle E_{p,q}^{1}= H_{p+q}\left(\frac{F^{p}}{F^{p-1}}\right) \Rightarrow H_{p+q}\left(\Nor_{A}\left(R\langle X \rangle \otimes_{R}A\right)\otimes_{A}M\right)$$
But $\frac{F^{p}}{F^{p-1}}= \frac{\Nor_{A}(\mathfrak{K}^{-p})\otimes_{A}M}{\Nor_{A}(\mathfrak{K}^{-p+1})\otimes_{A}M} \cong \Nor_{A}\left(\frac{\mathfrak{K}^{-p}}{\mathfrak{K}^{-p+1}}\right)\otimes_{A}M$, so we get the following spectral sequence:
$$\textstyle E_{p,q}^{1}= H_{p+q}\left(\Nor_{A}\left(\frac{\mathfrak{K}^{-p}}{\mathfrak{K}^{-p+1}}\right)\otimes_{A}M\right) \Rightarrow H_{p+q}\left(\Nor_{A}\left(R\langle X \rangle \otimes_{R}A\right)\otimes_{A}M\right)$$
We now reparametrize the spectral sequence as follows. Consider the following differential:
$$\textstyle H_{p+q}\left(\Nor_{A}\left(\frac{\mathfrak{K}^{-p}}{\mathfrak{K}^{-p+1}}\right)\otimes_{A}M\right)= E_{p,q}^{1} \xrightarrow{d_{p,q}^{1}} E_{p-1,q}^{1}= H_{p+q-1}\left(\Nor_{A}\left(\frac{\mathfrak{K}^{-p+1}}{\mathfrak{K}^{-p+2}}\right)\otimes_{A}M\right)$$
We find $p'$ and $q'$ such that $q'=-p$ and $p'+q'= p+q$. Thus $p'= 2p+q$ and $q'= -p$. Then $H_{p+q}\left(\Nor_{A}\left(\frac{\mathfrak{K}^{-p}}{\mathfrak{K}^{-p+1}}\right)\otimes_{A}M\right)$ is transformed into $H_{p'+q'}\left(\Nor_{A}\left(\frac{\mathfrak{K}^{q'}}{\mathfrak{K}^{q'+1}}\right)\otimes_{A}M\right)$, and $H_{p+q-1}\left(\Nor_{A}\left(\frac{\mathfrak{K}^{-p+1}}{\mathfrak{K}^{-p+2}}\right)\otimes_{A}M\right)$ is transformed into $H_{p'+q'-1}\left(\Nor_{A}\left(\frac{\mathfrak{K}^{q'+1}}{\mathfrak{K}^{q'+2}}\right)\otimes_{A}M\right)$. Since $d_{p,q}^{1}$ is induced by $\partial_{p+q}^{\Nor_{A}\left(R\langle X \rangle \otimes_{R}A\right)\otimes_{A}M}$, we see that $d_{p,q}^{1}$ is transformed into the following homomorphism which is induced by $\partial_{p'+q'}^{\Nor_{A}\left(R\langle X \rangle \otimes_{R}A\right)\otimes_{A}M}$:
$$\textstyle H_{p'+q'}\left(\Nor_{A}\left(\frac{\mathfrak{K}^{q'}}{\mathfrak{K}^{q'+1}}\right)\otimes_{A}M\right) \rightarrow H_{p'+q'-1}\left(\Nor_{A}\left(\frac{\mathfrak{K}^{q'+1}}{\mathfrak{K}^{q'+2}}\right)\otimes_{A}M\right)$$
Looking carefully at the indices, we notice that this homomorphism should be $d_{p',q'}^{2}$. In other words, we have:
$$\textstyle H_{p'+q'}\left(\Nor_{A}\left(\frac{\mathfrak{K}^{q'}}{\mathfrak{K}^{q'+1}}\right)\otimes_{A}M\right)=E_{p',q'}^{2} \xrightarrow{d_{p',q'}^{2}} E_{p'-2,q'+1}^{2}=H_{p'+q'-1}\left(\Nor_{A}\left(\frac{\mathfrak{K}^{q'+1}}{\mathfrak{K}^{q'+2}}\right)\otimes_{A}M\right)$$
Reverting the indices back to $p$ and $q$, we get the following spectral sequence:
$$\textstyle E_{p,q}^{2}= H_{p+q}\left(\Nor_{A}\left(\frac{\mathfrak{K}^{q}}{\mathfrak{K}^{q+1}}\right)\otimes_{A}M\right) \Rightarrow H_{p+q}\left(\Nor_{A}\left(R\langle X \rangle \otimes_{R}A\right)\otimes_{A}M\right)$$
By \cite[Exercise 5.5]{Iy} or \cite[paragraph before Lemma 6.5]{Qu1}, there is an isomorphism $\Omega_{R\langle X\rangle \mid R}\otimes_{R\langle X\rangle}A \cong \frac{\mathfrak{K}}{\mathfrak{K}^{2}}$ of simplicial $A$-modules, so for any $q\geq 0$, there is an isomorphism
$$\textstyle \frac{\mathfrak{K}^{q}}{\mathfrak{K}^{q+1}} \cong \Sym_{A}^{q}\left(\frac{\mathfrak{K}}{\mathfrak{K}^{2}}\right) \cong \Sym_{A}^{q}\left(\Omega_{R\langle X\rangle \mid R}\otimes_{R\langle X\rangle}A \right)$$
of simplicial $A$-modules which gives the following isomorphism of $A$-complexes:
$$\textstyle \Nor_{A}\left(\frac{\mathfrak{K}^{q}}{\mathfrak{K}^{q+1}}\right)\otimes_{A}M \cong \Nor_{A}\left(\Sym_{A}^{q}\left(\Omega_{R\langle X\rangle \mid R}\otimes_{R\langle X\rangle}A \right)\right)\otimes_{A}M$$
Moreover, $\Nor_{R}(\rho):\Nor_{R}\left(R\langle X\rangle\right) \rightarrow \Nor_{R}(A)=A$ is a projective resolution of $A$ as an $R$-module, so we have
\begin{equation*}
\begin{split}
 H_{p+q}\left(\Nor_{A}\left(R\langle X \rangle \otimes_{R}A\right)\otimes_{A}M\right) & \cong H_{p+q}\left(\left(\Nor_{R}\left(R\langle X \rangle\right)\otimes_{R}A\right)\otimes_{A}M\right) \\
 & \cong H_{p+q}\left(\Nor_{R}\left(R\langle X \rangle\right)\otimes_{R}(A\otimes_{A}M)\right) \\
 & \cong H_{p+q}\left(\Nor_{R}\left(R\langle X \rangle\right)\otimes_{R}M\right) \\
 & = \Tor_{p+q}^{R}(A,M)
\end{split}
\end{equation*}
for every $p,q\in \mathbb{Z}$ with $p+q\geq 0$. Thus we get the following spectral sequence:
$$E_{p,q}^{2}= H_{p+q}\left(\Nor_{A}\left(\Sym_{A}^{q}\left(\Omega_{R\langle X\rangle \mid R}\otimes_{R\langle X\rangle}A\right)\right)\otimes_{A}M\right) \Rightarrow \Tor_{p+q}^{R}(A,M)$$

We finally show that the spectral sequence is first quadrant. If $q<0$, then $\Sym_{A}^{q}\left(\Omega_{R\langle X\rangle \mid R}\otimes_{R\langle X\rangle}A\right)=0$, so $E_{p,q}^{2}=0$. For any $q\geq 0$, the short exact sequence
$$\textstyle 0\rightarrow \Nor_{A}(\mathfrak{K}^{q+1})\otimes_{A}M \rightarrow \Nor_{A}(\mathfrak{K}^{q})\otimes_{A}M \rightarrow \Nor_{A}\left(\frac{\mathfrak{K}^{q}}{\mathfrak{K}^{q+1}}\right)\otimes_{A}M \rightarrow 0$$
yields the following exact sequence:
$$\textstyle H_{p+q}\left(\Nor_{A}(\mathfrak{K}^{q})\otimes_{A}M\right) \rightarrow H_{p+q}\left(\Nor_{A}\left(\frac{\mathfrak{K}^{q}}{\mathfrak{K}^{q+1}}\right)\otimes_{A}M\right) \rightarrow H_{p+q-1}\left(\Nor_{A}(\mathfrak{K}^{q+1})\otimes_{A}M\right)$$
If $p<0$, then $p+q\leq q-1$ and $p+q-1\leq q$, so $H_{p+q}\left(\Nor_{A}(\mathfrak{K}^{q})\otimes_{A}M\right) =0= H_{p+q-1}\left(\Nor_{A}(\mathfrak{K}^{q+1})\otimes_{A}M\right)$, whence the above exact sequence implies that:
$$\textstyle E_{p,q}^{2}= H_{p+q}\left(\Nor_{A}\left(\Sym_{A}^{q}\left(\Omega_{R\langle X\rangle \mid R}\otimes_{R\langle X\rangle}A\right)\right)\otimes_{A}M\right) \cong H_{p+q}\left(\Nor_{A}\left(\frac{\mathfrak{K}^{q}}{\mathfrak{K}^{q+1}}\right)\otimes_{A}M\right) =0$$
Hence the spectral sequence is first quadrant.

(ii): Consider the filtration
$$\cdots \subseteq \Nor_{A}(\mathfrak{K}^{s+1}) \subseteq \Nor_{A}(\mathfrak{K}^{s}) \subseteq \cdots \subseteq \Nor_{A}(\mathfrak{K}^{2}) \subseteq \Nor_{A}(\mathfrak{K}) \subseteq \Nor_{A}\left(R\langle X\rangle \otimes_{R}A\right)$$
of the $A$-complex $\Nor_{A}\left(R\langle X\rangle \otimes_{R}A\right)$. For any $p,n\in \mathbb{Z}$, set:
$$F_{n}^{p}:= \left\{f\in \Hom_{A}\left(\Nor_{A}\left(R\langle X \rangle\otimes_{R}A\right),M\right)_{n}= \Hom_{A}\left(\Nor_{A}\left(R\langle X \rangle\otimes_{R}A\right)_{-n},M\right) \suchthat f|_{\Nor_{A}(\mathfrak{K}^{p})_{-n}}=0 \right\}$$
Moreover, define $\partial_{n}^{F^{p}}:F_{n}^{p}\rightarrow F_{n-1}^{p}$ by setting
\small
$$\partial_{n}^{F^{p}}(f):= \partial_{n}^{\Hom_{A}\left(\Nor_{A}\left(R\langle X \rangle\otimes_{R}A\right),M\right)}(f)= (-1)^{n+1}f\partial_{-n+1}^{\Nor_{A}\left(R\langle X \rangle\otimes_{R}A\right)}= (-1)^{n+1}\Hom_{A}\left(\partial_{-n+1}^{\Nor_{A}\left(R\langle X \rangle\otimes_{R}A\right)},M\right)(f)$$
\normalsize
for every $f\in F_{n}^{p}$. Then it is easily seen that $F^{p}$ is a subcomplex of the $A$-complex $\Hom_{A}\left(\Nor_{A}\left(R\langle X \rangle\otimes_{R}A\right),M\right)$. This gives a filtration
$$0= F^{0} \subseteq F^{1} \subseteq \cdots \subseteq F^{p-1} \subseteq F^{p} \subseteq \cdots \subseteq \Hom_{A}\left(\Nor_{A}\left(R\langle X \rangle\otimes_{R}A\right),M\right)$$
of the $A$-complex $\Hom_{A}\left(\Nor_{A}\left(R\langle X \rangle\otimes_{R}A\right),M\right)$.

We next show that $H_{-k}\left(\Hom_{A}\left(\Nor_{A}(\mathfrak{K}^{s}),M\right)\right)=0$ for every $s\geq 0$ and $k\leq s-1$. Fix $s\geq 0$. We show that $H_{-k}\left(\Hom_{A}\left(\Nor_{A}(\mathfrak{K}^{s}),M\right)\right) \cong \Hom_{A}\left(H_{k}\left(\Nor_{A}(\mathfrak{K}^{s})\right),M\right)$ for every $A$-module $M$ and $k\leq s-1$; see \cite[Chapter II, Lemme 19]{An}. For $k<0$, both sides are zero, so we may assume that $0\leq k\leq s-1$. We argue by induction on $k$. Let $k=0$. Then we have:
\begin{equation*}
\begin{split}
 H_{0}\left(\Hom_{A}\left(\Nor_{A}(\mathfrak{K}^{s}),M\right)\right) & = \Ker\left(\partial_{-1}^{\Hom_{A}\left(\Nor_{A}(\mathfrak{K}^{s}),M\right)}\right) = \Ker\left(\Hom_{A}\left(\partial_{1}^{\Nor_{A}(\mathfrak{K}^{s})},M\right)\right) \\
 & \cong \Hom_{A}\left(\Coker\left(\partial_{1}^{\Nor_{A}(\mathfrak{K}^{s})}\right),M\right) = \Hom_{A}\left(H_{0}\left(\Nor_{A}(\mathfrak{K}^{s})\right),M\right)
\end{split}
\end{equation*}
Now suppose that $k\geq 1$, and the result holds for $k-1$. Let $0\rightarrow N' \rightarrow N \rightarrow N'' \rightarrow 0$ be a short exact sequence of $A$-modules. Since $\Nor_{A}(\mathfrak{K}^{s})$ is an $A$-complex of projective modules, by \cite[Example 2.3.18]{CFH}, we get the short exact sequence
$$0\rightarrow \Hom_{A}\left(\Nor_{A}(\mathfrak{K}^{s}),N'\right) \rightarrow \Hom_{A}\left(\Nor_{A}(\mathfrak{K}^{s}),N\right) \rightarrow \Hom_{A}\left(\Nor_{A}(\mathfrak{K}^{s}),N''\right) \rightarrow 0$$
of $A$-complexes, which gives the following exact sequence:
$$H_{-k+1}\left(\Hom_{A}\left(\Nor_{A}(\mathfrak{K}^{s}),N''\right)\right) \rightarrow H_{-k}\left(\Hom_{A}\left(\Nor_{A}(\mathfrak{K}^{s}),N'\right)\right) \rightarrow H_{-k}\left(\Hom_{A}\left(\Nor_{A}(\mathfrak{K}^{s}),N\right)\right) $$$$ \rightarrow H_{-k}\left(\Hom_{A}\left(\Nor_{A}(\mathfrak{K}^{s}),N''\right)\right)$$
By the induction hypothesis, $H_{-k+1}\left(\Hom_{A}\left(\Nor_{A}(\mathfrak{K}^{s}),N''\right)\right) \cong \Hom_{A}\left(H_{k-1}\left(\Nor_{A}(\mathfrak{K}^{s})\right),N''\right)$. Since $k-1< k \leq s-1$, Theorem \ref{03.6} implies that $H_{k-1}\left(\Nor_{A}(\mathfrak{K}^{s})\right)= \pi_{k-1}(\mathfrak{K}^{s})=0$, so $H_{-k+1}\left(\Hom_{A}\left(\Nor_{A}(\mathfrak{K}^{s}),N''\right)\right)=0$. Therefore, the above exact sequence turns into the following exact sequence:
$$0 \rightarrow H_{-k}\left(\Hom_{A}\left(\Nor_{A}(\mathfrak{K}^{s}),N'\right)\right) \rightarrow H_{-k}\left(\Hom_{A}\left(\Nor_{A}(\mathfrak{K}^{s}),N\right)\right) \rightarrow H_{-k}\left(\Hom_{A}\left(\Nor_{A}(\mathfrak{K}^{s}),N''\right)\right)$$
This in particular shows that the functor $H_{-k}\left(\Hom_{A}\left(\Nor_{A}(\mathfrak{K}^{s}),-\right)\right):\mathcal{M}\mathpzc{od}(A) \rightarrow \mathcal{M}\mathpzc{od}(A)$ is left exact. On the other hand, the contravariant functor $\Hom_{A}(-,M):\mathcal{M}\mathpzc{od}(A) \rightarrow \mathcal{M}\mathpzc{od}(A)$ is left exact, so there is a natural $A$-homomorphism $\chi_{M}:H_{-k}\left(\Hom_{A}\left(\Nor_{A}(\mathfrak{K}^{s}),M\right)\right) \rightarrow \Hom_{A}\left(H_{k}\left(\Nor_{A}(\mathfrak{K}^{s})\right),M\right)$. Moreover, if $I$ is an injective $A$-module, then the functor $\Hom_{A}(-,I):\mathcal{M}\mathpzc{od}(A) \rightarrow \mathcal{M}\mathpzc{od}(A)$ is exact, so we conclude that $\chi_{I}$ is an isomorphism; see \cite[2.2.19]{CFH}. Now let $0\rightarrow M \rightarrow I_{0}\rightarrow I_{-1}$ be an exact sequence in which $I_{0}$ and $I_{-1}$ are injective $A$-modules. Then we get the following commutative diagram with exact rows:
\small
\begin{equation*}
\begin{tikzcd}
  0 \arrow{r} &
  H_{-k}\left(\Hom_{A}\left(\Nor_{A}(\mathfrak{K}^{s}),M\right)\right) \arrow{r} \arrow{d}[swap]{\chi_{M}} & H_{-k}\left(\Hom_{A}\left(\Nor_{A}(\mathfrak{K}^{s}),I_{0}\right)\right) \arrow{r} \arrow{d}{\cong}[swap]{\chi_{I_{0}}} & H_{-k}\left(\Hom_{A}\left(\Nor_{A}(\mathfrak{K}^{s}),I_{-1}\right)\right) \arrow{d}[swap]{\chi_{I_{-1}}}
  \\
  0 \arrow{r} &
  \Hom_{A}\left(H_{k}\left(\Nor_{A}(\mathfrak{K}^{s})\right),M\right) \arrow{r} &
  \Hom_{A}\left(H_{k}\left(\Nor_{A}(\mathfrak{K}^{s})\right),I_{0}\right) \arrow{r} &
  \Hom_{A}\left(H_{k}\left(\Nor_{A}(\mathfrak{K}^{s})\right),I_{-1}\right)
\end{tikzcd}
\end{equation*}
\normalsize
It follows that $\chi_{M}$ is an isomorphism. Now since $k\leq s-1$, Theorem \ref{03.6} implies that $H_{k}\left(\Nor_{A}(\mathfrak{K}^{s})\right)= \pi_{k}(\mathfrak{K}^{s})=0$, so $H_{-k}\left(\Hom_{A}\left(\Nor_{A}(\mathfrak{K}^{s}),M\right)\right)=0$.

We have observed before that for any $s\geq 0$, $\frac{\mathfrak{K}^{s}}{\mathfrak{K}^{s+1}}$  is a simplicial $A$-module consisting of projective modules, so we conclude that $\Nor_{A}\left(\frac{\mathfrak{K}^{s}}{\mathfrak{K}^{s+1}}\right)$ is an $A$-complex of projective modules. In particular, $\Nor_{A}\left(\frac{\mathfrak{K}}{\mathfrak{K}^{2}}\right)$ and $\Nor_{A}\left(\frac{R\langle X\rangle\otimes_{R}A}{\mathfrak{K}}\right)$ are $A$-complexes of projective modules. Therefore, the short exact sequence
$$\textstyle 0\rightarrow \Nor_{A}\left(\frac{\mathfrak{K}}{\mathfrak{K}^{2}}\right) \rightarrow \Nor_{A}\left(\frac{R\langle X\rangle\otimes_{R}A}{\mathfrak{K}^{2}}\right) \rightarrow \Nor_{A}\left(\frac{R\langle X\rangle\otimes_{R}A}{\mathfrak{K}}\right) \rightarrow 0$$
of $A$-complexes is degreewise split, implying that $\Nor_{A}\left(\frac{R\langle X\rangle\otimes_{R}A}{\mathfrak{K}^{2}}\right)$ is an $A$-complexes of projective module. Similarly, considering the degreewise split short exact sequence
$$\textstyle 0\rightarrow \Nor_{A}\left(\frac{\mathfrak{K}^{2}}{\mathfrak{K}^{3}}\right) \rightarrow \Nor_{A}\left(\frac{R\langle X\rangle\otimes_{R}A}{\mathfrak{K}^{3}}\right) \rightarrow \Nor_{A}\left(\frac{R\langle X\rangle\otimes_{R}A}{\mathfrak{K}^{2}}\right) \rightarrow 0$$
of $A$-complexes, we deduce that $\Nor_{A}\left(\frac{R\langle X\rangle\otimes_{R}A}{\mathfrak{K}^{3}}\right)$ is an $A$-complexes of projective module. Continuing in this manner, we infer that $\frac{\Nor_{A}\left(R\langle X\rangle\otimes_{R}A\right)}{\Nor_{A}(\mathfrak{K}^{p})} \cong \Nor_{A}\left(\frac{R\langle X\rangle\otimes_{R}A}{\mathfrak{K}^{p}}\right)$ is an $A$-complex of projective module for every $p\geq 0$. Now let $p\geq 0$. The short exact sequence
$$\textstyle 0\rightarrow \Nor_{A}(\mathfrak{K}^{p}) \xrightarrow{\kappa^{p}} \Nor_{A}\left(R\langle X\rangle\otimes_{R}A\right) \xrightarrow{\varpi^{p}} \frac{\Nor_{A}\left(R\langle X\rangle\otimes_{R}A\right)}{\Nor_{A}(\mathfrak{K}^{p})} \rightarrow 0$$
of $A$-complexes, in which $\kappa^{p}$ is the inclusion morphism and $\varpi^{p}$ is the quotient morphism, is degreewise split. Therefore, it yields the following degreewise split short exact sequence of $A$-complexes:
\footnotesize
$$\textstyle 0\rightarrow \Hom_{A}\left(\frac{\Nor_{A}\left(R\langle X\rangle\otimes_{R}A\right)}{\Nor_{A}(\mathfrak{K}^{p})},M\right) \xrightarrow{\Hom_{A}(\varpi^{p},M)} \Hom_{A}\left(\Nor_{A}\left(R\langle X\rangle\otimes_{R}A\right),M\right) \xrightarrow{\Hom_{A}(\kappa^{p},M)} \Hom_{A}\left(\Nor_{A}(\mathfrak{K}^{p}),M\right) \rightarrow 0$$
\normalsize
In particular, the morphism $\Hom_{A}(\kappa^{p},M)$ of $A$-complexes is surjective. Furthermore, we have for every $n\in \mathbb{Z}$:
\begin{equation*}
\begin{split}
 \Ker\left(\Hom_{A}(\kappa^{p},M)\right)_{n} & = \Ker\left(\Hom_{A}(\kappa^{p},M)_{n}\right) \\
 & = \Ker\left(\Hom_{A}(\kappa^{p}_{-n},M)\right) \\
 & = \left\{f\in \Hom_{A}\left(\Nor_{A}(R\langle X\rangle\otimes_{R}A)_{-n},M\right) \suchthat f\kappa^{p}_{-n}= \Hom_{A}(\kappa^{p}_{-n},M)(f)=0 \right\} \\
 & = \left\{f\in \Hom_{A}\left(\Nor_{A}(R\langle X\rangle\otimes_{R}A)_{-n},M\right) \suchthat f|_{\Nor(\mathfrak{K}^{p})_{-n}}=0 \right\} \\
 & = F_{n}^{p}
\end{split}
\end{equation*}
Thus $\Ker\left(\Hom_{A}(\kappa^{p},M)\right)= F^{p}$. Besides, the above short exact sequence implies that the induced morphism
$$\textstyle \overline{\Hom_{A}(\varpi^{p},M)}: \Hom_{A}\left(\frac{\Nor_{A}\left(R\langle X\rangle\otimes_{R}A\right)}{\Nor_{A}(\mathfrak{K}^{p})},M\right) \rightarrow \im\left(\Hom_{A}(\varpi^{p},M)\right) = \Ker\left(\Hom_{A}(\kappa^{p},M)\right) = F^{p}$$
is an isomorphism. Now consider the short exact sequence
$$0\rightarrow F^{p} \xrightarrow{\iota^{p}} \Hom_{A}\left(\Nor_{A}\left(R\langle X\rangle\otimes_{R}A\right),M\right) \xrightarrow{\Hom_{A}(\kappa^{p},M)} \Hom_{A}\left(\Nor_{A}(\mathfrak{K}^{p}),M\right) \rightarrow 0$$
of $A$-complexes in which $\iota^{p}$ is the inclusion morphism. We showed that $H_{-k}\left(\Hom_{A}\left(\Nor_{A}(\mathfrak{K}^{s}),M\right)\right)=0$ for every $s\geq 0$ and $k\leq s-1$. Thus we get the following exact sequence for every $p\geq -n+1$:
$$0= H_{n+1}\left(\Hom_{A}\left(\Nor_{A}(\mathfrak{K}^{p}),M\right)\right) \rightarrow H_{n}(F^{p}) \xrightarrow{H_{n}(\iota^{p})} H_{n}\left(\Hom_{A}\left(\Nor_{A}(R\langle X\rangle\otimes_{R}A),M\right)\right) $$$$ \rightarrow H_{n}\left(\Hom_{A}\left(\Nor_{A}(\mathfrak{K}^{p}),M\right)\right)=0$$
As a result, $H_{n}(\iota^{p})$ is an isomorphism for every $p\geq -n+1$. Thus Corollary \ref{02.3} implies that there is a spectral sequence as follows:
$$\textstyle E_{p,q}^{1}=H_{p+q}\left(\frac{F^{p}}{F^{p-1}}\right) \Rightarrow H_{p+q}\left(\Hom_{A}\left(\Nor_{A}\left(R\langle X\rangle\otimes_{R}A\right),M\right)\right)$$
The short exact sequence
$$\textstyle 0\rightarrow \frac{\Nor_{A}(\mathfrak{K}^{p-1})}{\Nor_{A}(\mathfrak{K}^{p})} \rightarrow \frac{\Nor_{A}(R\langle X\rangle\otimes_{R}A)}{\Nor_{A}(\mathfrak{K}^{p})} \xrightarrow{\pi^{p}} \frac{\Nor_{A}(R\langle X\rangle\otimes_{R}A)}{\Nor_{A}(\mathfrak{K}^{p-1})} \rightarrow 0$$
of $A$-complexes is degreewise split. Therefore, we obtain the following degreewise split short exact sequence of $A$-complexes:
\small
$$\textstyle 0\rightarrow \Hom_{A}\left(\frac{\Nor_{A}(R\langle X\rangle\otimes_{R}A)}{\Nor_{A}(\mathfrak{K}^{p-1})},M\right) \xrightarrow{\Hom_{A}(\pi^{p},M)} \Hom_{A}\left(\frac{\Nor_{A}(R\langle X\rangle\otimes_{R}A)}{\Nor_{A}(\mathfrak{K}^{p})},M\right) \rightarrow \Hom_{A}\left(\frac{\Nor_{A}(\mathfrak{K}^{p-1})}{\Nor_{A}(\mathfrak{K}^{p})},M\right) \rightarrow 0$$
\normalsize
Let $\lambda^{p}:F^{p-1} \rightarrow F^{p}$ be the inclusion morphism. Then the following diagram is commutative:
\footnotesize
\begin{equation*}
\begin{tikzcd}
  0 \arrow{r} &
  \Hom_{A}\left(\frac{\Nor_{A}(R\langle X\rangle\otimes_{R}A)}{\Nor_{A}(\mathfrak{K}^{p-1})},M\right) \arrow{r}{\Hom_{A}(\pi^{p},M)} \arrow{d}{\cong}[swap]{\overline{\Hom_{A}(\varpi^{p-1},M)}} & [3em] \Hom_{A}\left(\frac{\Nor_{A}(R\langle X\rangle\otimes_{R}A)}{\Nor_{A}(\mathfrak{K}^{p})},M\right) \arrow{r} \arrow{d}{\cong}[swap]{\overline{\Hom_{A}(\varpi^{p},M)}} & \Hom_{A}\left(\frac{\Nor_{A}(\mathfrak{K}^{p-1})}{\Nor_{A}(\mathfrak{K}^{p})},M\right) \arrow{r} & 0
  \\
  0 \arrow{r} &
  F^{p-1} \arrow{r}{\lambda^{p}} &
  F^{p} \arrow{r} &
  \frac{F^{p}}{F^{p-1}} \arrow{r} & 0
\end{tikzcd}
\end{equation*}
\normalsize
As a consequence, there is an isomorphism $\Hom_{A}\left(\frac{\Nor_{A}(\mathfrak{K}^{p-1})}{\Nor_{A}(\mathfrak{K}^{p})},M\right) \rightarrow \frac{F^{p}}{F^{p-1}}$ that makes the above diagram commutative. In particular, $\frac{F^{p}}{F^{p-1}} \cong \Hom_{A}\left(\frac{\Nor_{A}(\mathfrak{K}^{p-1})}{\Nor_{A}(\mathfrak{K}^{p})},M\right) \cong \Hom_{A}\left(\Nor_{A}\left(\frac{\mathfrak{K}^{p-1}}{\mathfrak{K}^{p}}\right),M\right)$. Hence we have the following spectral sequence:
$$\textstyle E_{p,q}^{1}= H_{p+q}\left(\Hom_{A}\left(\Nor_{A}\left(\frac{\mathfrak{K}^{p-1}}{\mathfrak{K}^{p}}\right),M\right)\right) \Rightarrow H_{p+q}\left(\Hom_{A}\left(\Nor_{A}\left(R\langle X\rangle\otimes_{R}A\right),M\right)\right)$$
We now reparametrize the spectral sequence as follows. Consider the following differential:
$$\textstyle H_{p+q}\left(\Hom_{A}\left(\Nor_{A}\left(\frac{\mathfrak{K}^{p-1}}{\mathfrak{K}^{p}}\right),M\right)\right)= E_{p,q}^{1} \xrightarrow{d_{p,q}^{1}} E_{p-1,q}^{1}= H_{p+q-1}\left(\Hom_{A}\left(\Nor_{A}\left(\frac{\mathfrak{K}^{p-2}}{\mathfrak{K}^{p-1}}\right),M\right)\right)$$
We find $p'$ and $q'$ such that $q'=-p+1$ and $p'+q'= p+q$. Thus $p'= 2p+q-1$ and $q'= -p+1$. Then $H_{p+q}\left(\Hom_{A}\left(\Nor_{A}\left(\frac{\mathfrak{K}^{p-1}}{\mathfrak{K}^{p}}\right),M\right)\right)$ is transformed into $H_{p'+q'}\left(\Hom_{A}\left(\Nor_{A}\left(\frac{\mathfrak{K}^{-q'}}{\mathfrak{K}^{-q'+1}}\right),M\right)\right)$, and $H_{p+q-1}\left(\Hom_{A}\left(\Nor_{A}\left(\frac{\mathfrak{K}^{p-2}}{\mathfrak{K}^{p-1}}\right),M\right)\right)$ is transformed into $H_{p'+q'-1}\left(\Hom_{A}\left(\Nor_{A}\left(\frac{\mathfrak{K}^{-q'-1}}{\mathfrak{K}^{-q'}}\right),M\right)\right)$. Since $d_{p,q}^{1}$ is induced by $\partial_{p+q}^{\Hom_{A}\left(\Nor_{A}\left(R\langle X \rangle \otimes_{R}A\right),M\right)}$, we see that $d_{p,q}^{1}$ is transformed into the following homomorphism which is induced by $\partial_{p'+q'}^{\Hom_{A}\left(\Nor_{A}\left(R\langle X \rangle \otimes_{R}A\right),M\right)}$:
$$\textstyle H_{p'+q'}\left(\Hom_{A}\left(\Nor_{A}\left(\frac{\mathfrak{K}^{-q'}}{\mathfrak{K}^{-q'+1}}\right),M\right)\right) \rightarrow H_{p'+q'-1}\left(\Hom_{A}\left(\Nor_{A}\left(\frac{\mathfrak{K}^{-q'-1}}{\mathfrak{K}^{-q'}}\right),M\right)\right)$$
Looking carefully at the indices, we notice that this homomorphism should be $d_{p',q'}^{2}$. In other words, we have:
\small
$$\textstyle H_{p'+q'}\left(\Hom_{A}\left(\Nor_{A}\left(\frac{\mathfrak{K}^{-q'}}{\mathfrak{K}^{-q'+1}}\right),M\right)\right)=E_{p',q'}^{2} \xrightarrow{d_{p',q'}^{2}} E_{p'-2,q'+1}^{2}=H_{p'+q'-1}\left(\Hom_{A}\left(\Nor_{A}\left(\frac{\mathfrak{K}^{-q'-1}}{\mathfrak{K}^{-q'}}\right),M\right)\right)$$
\normalsize
Reverting the indices back to $p$ and $q$, we get the following spectral sequence:
$$\textstyle E_{p,q}^{2}= H_{p+q}\left(\Hom_{A}\left(\Nor_{A}\left(\frac{\mathfrak{K}^{-q}}{\mathfrak{K}^{-q+1}}\right),M\right)\right) \Rightarrow H_{p+q}\left(\Hom_{A}\left(\Nor_{A}\left(R\langle X\rangle\otimes_{R}A\right),M\right)\right)$$
For any $q\leq 0$, there is an isomorphism $\frac{\mathfrak{K}^{-q}}{\mathfrak{K}^{-q+1}} \cong \Sym_{A}^{-q}\left(\frac{\mathfrak{K}}{\mathfrak{K}^{2}}\right) \cong \Sym_{A}^{-q}\left(\Omega_{R\langle X\rangle \mid R}\otimes_{R\langle X\rangle}A \right)$ of simplicial $A$-modules which gives the following isomorphism of $A$-complexes:
$$\textstyle \Hom_{A}\left(\Nor_{A}\left(\frac{\mathfrak{K}^{-q}}{\mathfrak{K}^{-q+1}}\right),M\right) \cong \Hom_{A}\left(\Nor_{A}\left(\Sym_{A}^{-q}\left(\Omega_{R\langle X\rangle \mid R}\otimes_{R\langle X\rangle}A \right)\right),M\right)$$
Moreover, $\Nor_{R}(\rho):\Nor_{R}\left(R\langle X\rangle\right) \rightarrow \Nor_{R}(A)=A$ is a projective resolution of $A$ as an $R$-module, so we have
\begin{equation*}
\begin{split}
 H_{p+q}\left(\Hom_{A}\left(\Nor_{A}\left(R\langle X\rangle\otimes_{R}A\right),M\right)\right) & \cong H_{p+q}\left(\Hom_{A}\left(\Nor_{R}\left(R\langle X\rangle\right)\otimes_{R}A,M\right)\right) \\
 & \cong H_{p+q}\left(\Hom_{R}\left(\Nor_{R}\left(R\langle X\rangle\right),\Hom_{A}(A,M)\right)\right) \\
 & \cong H_{p+q}\left(\Hom_{R}\left(\Nor_{R}\left(R\langle X\rangle\right),M\right)\right) \\
 & = \Ext_{R}^{-p-q}(A,M)
\end{split}
\end{equation*}
for every $p,q\in \mathbb{Z}$ with $p+q\leq 0$. Thus we get the following spectral sequence:
$$E_{p,q}^{2}= H_{p+q}\left(\Hom_{A}\left(\Nor_{A}\left(\Sym_{A}^{-q}\left(\Omega_{R\langle X\rangle \mid R}\otimes_{R\langle X\rangle}A\right)\right),M\right)\right) \Rightarrow \Ext_{R}^{-p-q}(A,M)$$

We finally show that the spectral sequence is third quadrant. If $q>0$, then $\Sym_{A}^{-q}\left(\Omega_{R\langle X\rangle \mid R}\otimes_{R\langle X\rangle}A\right)=0$, so $E_{p,q}^{2}=0$. For any $q\leq 0$, the degreewise split short exact sequence
$$\textstyle 0\rightarrow \Nor_{A}(\mathfrak{K}^{-q+1}) \rightarrow \Nor_{A}(\mathfrak{K}^{-q}) \rightarrow \Nor_{A}\left(\frac{\mathfrak{K}^{-q}}{\mathfrak{K}^{-q+1}}\right) \rightarrow 0$$
of $A$-complexes gives the degreewise split short exact sequence
$$\textstyle 0\rightarrow \Hom_{A}\left(\Nor_{A}\left(\frac{\mathfrak{K}^{-q}}{\mathfrak{K}^{-q+1}}\right),M\right) \rightarrow \Hom_{A}\left(\Nor_{A}(\mathfrak{K}^{-q}),M\right) \rightarrow \Hom_{A}\left(\Nor_{A}(\mathfrak{K}^{-q+1}),M\right) \rightarrow 0$$
of $A$-complexes which in turn yields the following exact sequence:
$$\textstyle H_{p+q+1}\left(\Hom_{A}\left(\Nor_{A}(\mathfrak{K}^{-q+1}),M\right)\right) \rightarrow H_{p+q}\left(\Hom_{A}\left(\Nor_{A}\left(\frac{\mathfrak{K}^{-q}}{\mathfrak{K}^{-q+1}}\right),M\right)\right) $$$$ \rightarrow H_{p+q}\left(\Hom_{A}\left(\Nor_{A}(\mathfrak{K}^{-q}),M\right)\right)$$
If $p>0$, then $-p-q-1 \leq -q$ and $-p-q \leq -q-1$, so $H_{p+q+1}\left(\Hom_{A}\left(\Nor_{A}(\mathfrak{K}^{-q+1}),M\right)\right) =0= H_{p+q}\left(\Hom_{A}\left(\Nor_{A}(\mathfrak{K}^{-q}),M\right)\right)$, whence the above exact sequence implies that:
\small
$$\textstyle E_{p,q}^{2}= H_{p+q}\left(\Hom_{A}\left(\Nor_{A}\left(\Sym_{A}^{-q}\left(\Omega_{R\langle X\rangle \mid R}\otimes_{R\langle X\rangle}A\right)\right),M\right)\right) \cong H_{p+q}\left(\Hom_{A}\left(\Nor_{A}\left(\frac{\mathfrak{K}^{-q}}{\mathfrak{K}^{-q+1}}\right),M\right)\right) =0$$
Hence the spectral sequence is third quadrant.
\normalsize
\end{proof}

\section{Five-Term Exact Sequences}

In this section, we present the five-term exact sequences resulting from the Quillen's fundamental spectral sequences. To this end, we need some generalizations of the five-term exact sequence of a general spectral sequence that seem to appear in \cite[Corollaries 10.32 and 10.34]{Ro}. However, the statements there are incorrect and no proof is presented. The correct statements with proofs are offered in \cite[Lemma 5.4]{DFT}. We include an even stronger statement here for the convenience of the reader.

\begin{lemma} \label{05.1}
Let $E^{2}_{p,q} \Rightarrow H_{p+q}$ be a spectral sequence. Then the following assertions hold:
\begin{enumerate}
\item[(i)] If the spectral sequence is first quadrant and there is an integer $n \geq 1$ such that $E^{2}_{p,q}=0$ for every $p\geq 2$ and $q \leq n-2$, then there is a five-term exact sequence as follows:
$$H_{n+1} \rightarrow E^{2}_{2,n-1} \rightarrow E^{2}_{0,n} \rightarrow H_{n} \rightarrow E^{2}_{1,n-1} \rightarrow 0$$
\item[(ii)] If the spectral sequence is third quadrant and there is an integer $n \geq 1$ such that $E^{2}_{p,q}=0$ for every $p\leq -2$ and $q \geq 2-n$, then there is a five-term exact sequence as follows:
$$0 \rightarrow E^{2}_{-1,1-n} \rightarrow H_{-n} \rightarrow E^{2}_{0,-n} \rightarrow E^{2}_{-2,1-n} \rightarrow H_{-n-1}$$
\end{enumerate}
\end{lemma}

\begin{proof}
For a proof, refer to \cite[Lemma 5.4]{DFT}. However, note that in the statement of \cite[Lemma 5.4 (i)]{DFT}, the vanishing $E^{2}_{p,q}=0$ is assumed for every $q \leq n-2$. But the proof shows that the weaker assumption that $E^{2}_{p,q}=0$ for every $p\geq 2$ and $q \leq n-2$ suffices. The same holds for \cite[Lemma 5.4 (ii)]{DFT}.
\end{proof}

The following lemma is well-known; see for example \cite[1.11 and 1.15]{DP} or \cite[Lemma 4.3 and Theorem 5.6]{Do}. However, we include a proof here based on the ideas of the latter reference for the sake of completeness.

\begin{lemma} \label{05.2}
Let $R$ and $S$ be two rings, and $\mathcal{F}:\mathpzc{s}\mathcal{M}\mathpzc{od}(R)\rightarrow \mathpzc{s}\mathcal{M}\mathpzc{od}(S)$ a covariant functor that is extended from a preadditive covariant functor $\mathcal{M}\mathpzc{od}(R)\rightarrow \mathcal{M}\mathpzc{od}(S)$. If $f,g:M\rightarrow N$ are two morphisms in $\mathpzc{s}\mathcal{M}\mathpzc{od}(R)$ such that $f\sim g$, $M$ is cofibrant in $\mathpzc{s}\mathcal{M}\mathpzc{od}(R)$, and $\mathcal{F}(M)$ is cofibrant in $\mathpzc{s}\mathcal{M}\mathpzc{od}(S)$, then $\mathcal{F}(f)\sim \mathcal{F}(g)$.
\end{lemma}

\begin{proof}
We consider the standard model structure on simplicial modules; see \cite[Theorem 4.15]{Fa}. For the notions and notations of homotopy used here, refer to \cite[Definitions 2.2 and 2.3]{Fa}. Since $f\sim g$, we have $f\sim_{r}g$. But every object of $\mathpzc{s}\mathcal{M}\mathpzc{od}(R)$ including $N$ is fibrant, so by \cite[Proposition 7.4.7]{Hi}, we conclude that $f\sim_{l}g$ using any cylinder object for $M$. In particular, as $M$ is cofibrant, we can use the cylinder object $M^{(\Delta^{1})}$ for $M$; see \cite[Proposition 6.5]{Fa}. Thus there is a commutative diagram
\begin{equation*}
  \begin{tikzcd}
  M\oplus M \arrow{r}{\nabla^{M}}\arrow{d}[swap]{\kappa^{M}}
  & M
  \\
  M^{(\Delta^{1})} \arrow{ru}[swap]{\xi^{M}}
\end{tikzcd}
\end{equation*}

\noindent
in which $\kappa^{M}$ is a cofibration and $\xi^{M}$ is a weak equivalence, such that $f$ and $g$ factor through $M^{(\Delta^{1})}$ as
\begin{equation*}
  \begin{tikzcd}
  M \arrow{r}{f} \arrow{rd}[swap]{\kappa^{M}\iota_{1}^{M}}
  & N
  & M \arrow{l}[swap]{g} \arrow{ld}{\kappa^{M}\iota_{2}^{M}}
  \\
  & M^{(\Delta^{1})} \arrow{u}{\phi}
\end{tikzcd}
\end{equation*}

\noindent
in which $\iota_{1}^{M}:M\rightarrow M\oplus M$ and $\iota_{2}^{M}:M\rightarrow M\oplus M$ are canonical injections. Now let $n\geq 0$. Then we have:
$$\Delta_{n}^{1}= \left\{(\alpha_{0},\alpha_{1},\ldots , \alpha_{n})\in \mathbb{N}^{n+1} \suchthat 0\leq \alpha_{0}\leq \alpha_{1}\leq \cdots \leq \alpha_{n}\leq 1 \right\}$$
We note that $|\Delta_{n}^{1}|= \binom{1+n+1}{1}= n+2$, so $\left(M^{(\Delta^{1})}\right)_{n}= M_{n}^{(\Delta_{n}^{1})}= M_{n}^{n+2}$; see \cite[Example 4.3]{Fa}. A careful inspection reveals that $\kappa_{n}^{M}(x,y)=(x,\underbrace{0,\ldots,0}_{n \textrm{ times}},y)$ for every $x,y\in M_{n}$, $\xi_{n}^{M}(x_{1},\ldots,x_{n+2})= x_{1}+\cdots +x_{n+2}$ for every $x_{1},\ldots ,x_{n+2}\in M_{n}$, and $\nabla_{n}^{M}(x,y)=x+y$ for every $x,y\in M_{n}$. Applying the functor $\mathcal{F}$ to the above diagram, we get the following commutative diagram in $\mathpzc{s}\mathcal{M}\mathpzc{od}(S)$:
\begin{equation*}
  \begin{tikzcd}
  \mathcal{F}(M) \arrow{r}{\mathcal{F}(f)} \arrow{rd}[swap]{\mathcal{F}\left(\kappa^{M}\iota_{1}^{M}\right)}
  & \mathcal{F}(N)
  & \mathcal{F}(M) \arrow{l}[swap]{\mathcal{F}(g)} \arrow{ld}{\mathcal{F}\left(\kappa^{M}\iota_{2}^{M}\right)}
  \\
  & \mathcal{F}\left(M^{(\Delta^{1})}\right) \arrow{u}{\mathcal{F}(\phi)}
\end{tikzcd}
\end{equation*}

\noindent
We next define a morphism $\theta:\mathcal{F}(M)^{(\Delta^{1})}\rightarrow \mathcal{F}\left(M^{(\Delta^{1})}\right)$ of simplicial $S$-modules as follows. Let $n\geq 0$. Let $\lambda_{n,j}:M_{n}\rightarrow M_{n}^{n+2}$ be the canonical injection into the $j$th component for every $0\leq j\leq n+2$. We define an $S$-homomorphism
$$\theta_{n}:\mathcal{F}(M_{n})^{(\Delta^{1}_{n})}=\mathcal{F}(M_{n})^{n+2} \rightarrow \mathcal{F}\left(M_{n}^{n+2}\right)= \mathcal{F}\left(M_{n}^{(\Delta^{1}_{n})}\right)$$
by setting
$$\theta_{n}(z_{1},...,z_{n+2})= \sum_{j=1}^{n+2}\mathcal{F}(\lambda_{n,j})(z_{j})$$
for every $z_{1},...,z_{n+2}\in \mathcal{F}(M_{n})$. Then the following diagram is commutative for every $0\leq i\leq n$:
\begin{equation*}
  \begin{tikzcd}[column sep=4.5em,row sep=2em]
  \mathcal{F}(M_{n-1})^{n+1} \arrow{r}{s_{n-1,i}^{\mathcal{F}(M)^{(\Delta^{1})}}} \arrow{d}{\theta_{n-1}}
  & \mathcal{F}(M_{n})^{n+2} \arrow{r}{d_{n,i}^{\mathcal{F}(M)^{(\Delta^{1})}}} \arrow{d}{\theta_{n}}
  & \mathcal{F}(M_{n-1})^{n+1} \arrow{d}{\theta_{n-1}}
  \\
  \mathcal{F}(M_{n-1}^{n+1}) \arrow{r}{s_{n-1,i}^{\mathcal{F}\left(M^{(\Delta^{1})}\right)}}
  & \mathcal{F}(M_{n}^{n+2}) \arrow{r}{d_{n,i}^{\mathcal{F}\left(M^{(\Delta^{1})}\right)}}
  & \mathcal{F}(M_{n-1}^{n+1})
\end{tikzcd}
\end{equation*}

\noindent
Indeed, we have for every $0\leq i\leq n$ and $z_{1},...,z_{n+2}\in \mathcal{F}(M_{n})$:
\begin{equation*}
\begin{split}
 d_{n,i}^{\mathcal{F}\left(M^{(\Delta^{1})}\right)}\left(\theta_{n}(z_{1},...,z_{n+2})\right)
 & = \mathcal{F}\left(d_{n,i}^{M^{(\Delta^{1})}}\right)\left(\theta_{n}(z_{1},...,z_{n+2})\right) = \mathcal{F}\left(\left(d_{n,i}^{M}\right)^{\left(d_{n,i}^{\Delta^{1}}\right)}\right)\left(\sum_{j=1}^{n+2}\mathcal{F}\left(\lambda_{n,j}\right)(z_{j})\right) \\
 & = \sum_{j=1}^{n+2}\mathcal{F}\left(\left(d_{n,i}^{M}\right)^{\left(d_{n,i}^{\Delta^{1}}\right)}\right)\left(\mathcal{F}(\lambda_{n,j})(z_{j})\right) = \sum_{j=1}^{n+2}\mathcal{F}\left(\left(d_{n,i}^{M}\right)^{\left(d_{n,i}^{\Delta^{1}}\right)}\lambda_{n,j}\right)(z_{j})
\end{split}
\end{equation*}

\noindent
Let $1\leq j\leq n+2$, and $x\in M_{n}$. Then $\lambda_{n,j}$ embeds $x$ into the $j$th component, and $M_{n}^{\left(d_{n,i}^{\Delta^{1}}\right)}$ deletes the $i$th component. As a result, we see that:
\small
\label{eqn:damage piecewise}
 $$\left(d_{n,i}^{M}\right)^{\left(d_{n,i}^{\Delta^{1}}\right)}\left(\lambda_{n,j}(x)\right)= \left(d_{n,i}^{M}\right)^{n+1}\left(M_{n}^{\left(d_{n,i}^{\Delta^{1}}\right)}\left(\lambda_{n,j}(x)\right)\right)=
      \begin{cases}
        0 & \text{if } j=i \\
        (0,...,\underbrace{d_{n,i}^{M}(x)}_{j \textrm{th spot}},...,0)= \lambda_{n-1,j}\left(d_{n,i}^{M}(x)\right) & \text{if } j\neq i
      \end{cases}$$

\normalsize
\noindent
In other words, we have:
$$\left(d_{n,i}^{M}\right)^{\left(d_{n,i}^{\Delta^{1}}\right)}\lambda_{n,j}=
      \begin{cases}
        0 & \text{if } j=i \\
        \lambda_{n-1,j}d_{n,i}^{M} & \text{if } j\neq i
      \end{cases}$$

\noindent
Therefore, we get:
\small
$$d_{n,i}^{\mathcal{F}\left(M^{(\Delta^{1})}\right)}\left(\theta_{n}(z_{1},...,z_{n+2})\right)= \sum_{j=1}^{n+2}\mathcal{F}\left(\left(d_{n,i}^{M}\right)^{\left(d_{n,i}^{\Delta^{1}}\right)}\lambda_{n,j}\right)(z_{j})= \sum_{j=1}^{i-1}\mathcal{F}\left(\lambda_{n-1,j}d_{n,i}^{M}\right)(z_{j})+ \sum_{j=i+1}^{n+2}\mathcal{F}\left(\lambda_{n-1,j}d_{n,i}^{M}\right)(z_{j})$$
\normalsize
On the other hand, we have:
\begin{equation*}
\begin{split}
 \theta_{n-1}\left(d_{n,i}^{\mathcal{F}(M)^{(\Delta^{1})}}(z_{1},...,z_{n+2})\right)
 & = \theta_{n-1}\left(\mathcal{F}\left(d_{n,i}^{M}\right)^{\left(d_{n,i}^{\Delta^{1}}\right)}(z_{1},...,z_{n+2})\right) \\
 & = \theta_{n-1}\left(\mathcal{F}\left(d_{n,i}^{M}\right)^{n+1}\left(\mathcal{F}(M_{n})^{\left(d_{n,i}^{\Delta^{1}}\right)}(z_{1},...,z_{n+2})\right)\right) \\
 & = \theta_{n-1}\left(\mathcal{F}\left(d_{n,i}^{M}\right)^{n+1}\left(z_{1},...,z_{i-1},z_{i+1},...,z_{n+2}\right)\right) \\
 & = \theta_{n-1}\left(\mathcal{F}\left(d_{n,i}^{M}\right)(z_{1}),...,\mathcal{F}\left(d_{n,i}^{M}\right)(z_{i-1}), \mathcal{F}\left(d_{n,i}^{M}\right)(z_{i+1}),...,\mathcal{F}\left(d_{n,i}^{M}\right)(z_{n+2})\right) \\
 & = \sum_{j=1}^{i-1}\mathcal{F}\left(\lambda_{n-1,j}\right)\left(\mathcal{F}\left(d_{n,i}^{M}\right)(z_{j})\right) + \sum_{j=i+1}^{n+2}\mathcal{F}\left(\lambda_{n-1,j}\right)\left(\mathcal{F}\left(d_{n,i}^{M}\right)(z_{j})\right) \\
 & = \sum_{j=1}^{i-1}\mathcal{F}\left(\lambda_{n-1,j}d_{n,i}^{M}\right)(z_{j})+ \sum_{j=i+1}^{n+2}\mathcal{F}\left(\lambda_{n-1,j}d_{n,i}^{M}\right)(z_{j})
\end{split}
\end{equation*}

\noindent
As a consequence, we have $d_{n,i}^{\mathcal{F}\left(M^{(\Delta^{1})}\right)}\left(\theta_{n}(z_{1},...,z_{n+2})\right)= \theta_{n-1}\left(d_{n,i}^{\mathcal{F}(M)^{(\Delta^{1})}}(z_{1},...,z_{n+2})\right)$, so $d_{n,i}^{\mathcal{F}\left(M^{(\Delta^{1})}\right)}\theta_{n}= \theta_{n-1}d_{n,i}^{\mathcal{F}(M)^{(\Delta^{1})}}$. Similarly, one can verify that $s_{n-1,i}^{\mathcal{F}\left(M^{(\Delta^{1})}\right)}\theta_{n-1}= \theta_{n}s_{n-1,i}^{\mathcal{F}(M)^{(\Delta^{1})}}$. This shows that $\theta=(\theta_{n})_{n\geq 0}:\mathcal{F}(M)^{(\Delta^{1})}\rightarrow \mathcal{F}\left(M^{(\Delta^{1})}\right)$ is a morphism of simplicial $S$-module.

Considering the commutative diagram
\begin{equation*}
\begin{tikzcd}[column sep=4.5em,row sep=2.5em]
  \mathcal{F}(M)\oplus \mathcal{F}(M) \arrow{r}{\nabla^{\mathcal{F}(M)}} \arrow{d}[swap]{\kappa^{\mathcal{F}(M)}}
  & \mathcal{F}(M)
  \\
  \mathcal{F}(M)^{(\Delta^{1})} \arrow{ru}[swap]{\xi^{\mathcal{F}(M)}}
  &
\end{tikzcd}
\end{equation*}

\noindent
we get the following commutative diagram:
\begin{equation*}
  \begin{tikzcd}[column sep=7em, row sep=4em]
  \mathcal{F}(M) \arrow{r}{\mathcal{F}(f)} \arrow{rd}[swap]{\mathcal{F}\left(\kappa^{M}\iota_{1}^{M}\right)} \arrow[bend right, swap]{ddr}{\kappa^{\mathcal{F}(M)}\iota_{1}^{\mathcal{F}(M)}}
  & \mathcal{F}(N)
  & \mathcal{F}(M) \arrow{l}[swap]{\mathcal{F}(g)} \arrow{ld}{\mathcal{F}\left(\kappa^{M}\iota_{2}^{M}\right)} \arrow[bend left]{ddl}{\kappa^{\mathcal{F}(M)}\iota_{2}^{\mathcal{F}(M)}}
  \\
  & \mathcal{F}\left(M^{(\Delta^{1})}\right) \arrow{u}{\mathcal{F}(\phi)} \\
  & \mathcal{F}(M)^{(\Delta^{1})} \arrow{u}{\theta}
\end{tikzcd}
\end{equation*}

\noindent
To see this, let $n\geq 0$. Then we have for every $z\in \mathcal{F}(M_{n})$:
$$\theta_{n}\left(\kappa_{n}^{\mathcal{F}(M)}\left(\left(\iota_{1}^{\mathcal{F}(M)}\right)_{n}(z)\right)\right)= \theta_{n}\left(\kappa_{n}^{\mathcal{F}(M)}(z,0)\right)= \theta_{n}(z,\underbrace{0,...,0}_{n+1 \textrm{ times}})= \mathcal{F}\left(\lambda_{n,1}\right)(z)$$
It follows that $\theta_{n}\kappa_{n}^{\mathcal{F}(M)}\left(\iota_{1}^{\mathcal{F}(M)}\right)_{n}= \mathcal{F}\left(\lambda_{n,1}\right)$. However, we have for every $x\in M_{n}$:
$$\lambda_{n,1}(x)= (x,\underbrace{0,...,0}_{n+1 \textrm{ times}})= \kappa_{n}^{M}(x,0)= \kappa_{n}^{M}\left((\iota_{1}^{M})_{n}(x)\right)$$
Hence $\lambda_{n,1}=\kappa_{n}^{M}(\iota_{1}^{M})_{n}$, so we get $\theta_{n}\kappa_{n}^{\mathcal{F}(M)}\left(\iota_{1}^{\mathcal{F}(M)}\right)_{n}= \mathcal{F}\left(\kappa_{n}^{M}(\iota_{1}^{M})_{n}\right)$. Therefore, $\theta\kappa^{\mathcal{F}(M)}\iota_{1}^{\mathcal{F}(M)}= \mathcal{F}\left(\kappa^{M}\iota_{1}^{M}\right)$. Similarly, one can verify that $\theta\kappa^{\mathcal{F}(M)}\iota_{2}^{\mathcal{F}(M)}= \mathcal{F}\left(\kappa^{M}\iota_{2}^{M}\right)$. As a result, we get the following commutative diagram:
\begin{equation*}
  \begin{tikzcd}[row sep=3em, column sep=3em]
  \mathcal{F}(M) \arrow{r}{\mathcal{F}(f)} \arrow[swap]{dr}{\kappa^{\mathcal{F}(M)}\iota_{1}^{\mathcal{F}(M)}}
  & \mathcal{F}(N)
  & \mathcal{F}(M) \arrow{l}[swap]{\mathcal{F}(g)} \arrow{dl}{\kappa^{\mathcal{F}(M)}\iota_{2}^{\mathcal{F}(M)}}
  \\
  & \mathcal{F}(M)^{(\Delta^{1})} \arrow{u}{\mathcal{F}(\phi)\theta}
\end{tikzcd}
\end{equation*}

\noindent
This shows that $\mathcal{F}(f) \sim_{l} \mathcal{F}(g)$. But $\mathcal{F}(M)$ is cofibrant, so we deduce that $\mathcal{F}(f) \sim_{r} \mathcal{F}(g)$; see \cite[Proposition 7.4.5]{Hi}. Therefore, $\mathcal{F}(f) \sim \mathcal{F}(g)$.
\end{proof}

\begin{corollary} \label{05.3}
Let $R$ and $S$ be two rings, and $\mathcal{F}:\mathpzc{s}\mathcal{M}\mathpzc{od}(R)\rightarrow \mathpzc{s}\mathcal{M}\mathpzc{od}(S)$ a covariant functor that is extended from a preadditive covariant functor $\mathcal{M}\mathpzc{od}(R)\rightarrow \mathcal{M}\mathpzc{od}(S)$. If $f:M \rightarrow N$ is a homotopy equivalence in $\mathpzc{s}\mathcal{M}\mathpzc{od}(R)$, $M$ and $N$ are cofibrant in $\mathpzc{s}\mathcal{M}\mathpzc{od}(R)$, and $\mathcal{F}(M)$ and $\mathcal{F}(N)$ are cofibrant in $\mathpzc{s}\mathcal{M}\mathpzc{od}(S)$, then $\mathcal{F}(f):\mathcal{F}(M) \rightarrow \mathcal{F}(N)$ is a homotopy equivalence in $\mathpzc{s}\mathcal{M}\mathpzc{od}(S)$.
\end{corollary}

\begin{proof}
Since $f:M\rightarrow N$ is a homotopy equivalence, there is a morphism $g:N\rightarrow M$ such that $gf \sim 1^{M}$ and $fg \sim 1^{N}$. Now the previous proposition implies that $\mathcal{F}(g)\mathcal{F}(f)=\mathcal{F}(gf) \sim \mathcal{F}(1^{M})=1^{\mathcal{F}(M)}$ and $\mathcal{F}(f)\mathcal{F}(g)=\mathcal{F}(fg) \sim \mathcal{F}(1^{N})=1^{\mathcal{F}(N)}$. Hence $\mathcal{F}(f)$ is a homotopy equivalence.
\end{proof}

Now we are ready to handle the five-term exact sequences of the Quillen's fundamental spectral sequences.

\begin{theorem} \label{05.4}
Let $R$ be a ring, $A$ an $R$-algebra, and $M$ an $A$-module. Suppose that the natural $R$-algebra epimorphism $\mu_{A}:A\otimes_{R}A\rightarrow A$, given by $\mu_{A} (a\otimes b)=ab$ for every $a,b\in A$, is an isomorphism. Then the following assertions hold:
\begin{enumerate}
\item[(i)] $H_{0}^{AQ}(A|R;M) = 0 = H_{AQ}^{0}(A|R;M)$.

\item[(ii)] $H_{1}^{AQ}(A|R;M) \cong \Tor_{1}^{R}(A,M)$ and $H_{AQ}^{1}(A|R;M) \cong \Ext_{R}^{1}(A,M)$.

\item[(iii)] There are natural exact sequences of $A$-modules as follows:
$$\Tor_{3}^{R}(A,M)\rightarrow H_{3}^{AQ}(A|R;M) \rightarrow \textstyle\bigwedge_{A}^{2}\Tor_{1}^{R}(A,A)\otimes_{A}M \rightarrow \Tor_{2}^{R}(A,M) \rightarrow H_{2}^{AQ}(A|R;M) \rightarrow 0$$
and
$$0\rightarrow H_{AQ}^{2}(A|R;M)\rightarrow \Ext_{R}^{2}(A,M)\rightarrow \Hom_{A}\left(\textstyle\bigwedge_{A}^{2}\Tor_{1}^{R}(A,A),M\right)\rightarrow H_{AQ}^{3}(A|R;M)\rightarrow \Ext_{R}^{3}(A,M)$$
\end{enumerate}
\end{theorem}

\begin{proof}
(i): By the hypothesis, $\mu_{A}$ is an isomorphism, so $\mathfrak{k}= \Ker(\mu_{A})=0$, whence $\Omega_{A|R}= \mathfrak{k}/\mathfrak{k}^{2} =0$. Using \cite[Proposition 3.8]{Qu1}, we get $H_{0}^{AQ}(A|R;M)\cong \Omega_{A|R}\otimes_{A}M=0$ and $H_{AQ}^{0}(A|R;M)\cong \Der_{R}(A,M) \cong \Hom_{A}\left(\Omega_{A|R},M\right)=0$.

(ii): Consider the following natural first quadrant spectral sequence from Theorem \ref{04.2}:
$$E_{p,q}^{2}= H_{p+q}\left(\Nor_{A}\left(\Sym_{A}^{q}\left(\Omega_{R\langle X\rangle \mid R}\otimes_{R\langle X\rangle}A\right)\right)\otimes_{A}M\right) \Rightarrow \Tor_{p+q}^{R}(A,M)$$
By \cite[Theorem 10.31]{Ro}, there is a natural five-term exact sequence as follows:
$$H_{2}\rightarrow E_{2,0}^{2}\rightarrow E_{0,1}^{2}\rightarrow H_{1}\rightarrow E_{1,0}^{2}\rightarrow 0$$
But we have:
\small
$$E_{1,0}^{2} = H_{1}\left(\Nor_{A}\left(\Sym_{A}^{0}\left(\Omega_{R\langle X\rangle \mid R}\otimes_{R\langle X\rangle}A\right)\right)\otimes_{A}M\right) = H_{1}\left(\Nor_{A}(A)\otimes_{A}M\right) = H_{1}\left(A\otimes_{A}M\right) \cong H_{1}(M) = 0$$
\normalsize
Similarly, we see that $E_{2,0}^{2}=0$. Thus from the above sequence, we get the following natural isomorphism:
\begin{equation*}
\begin{split}
 H_{1}^{AQ}(A|R;M) & = H_{1}\left(\mathbb{L}^{A|R}\otimes_{A}M\right) = H_{1}\left(\Nor_{A}\left(\Omega_{R\langle X\rangle \mid R}\otimes_{R\langle X\rangle}A\right)\otimes_{A}M\right) \\
 & = H_{1}\left(\Nor_{A}\left(\Sym_{A}^{1}\left(\Omega_{R\langle X\rangle \mid R}\otimes_{R\langle X\rangle}A\right)\right)\otimes_{A}M\right) = E_{0,1}^{2} \cong H_{1} = \Tor_{1}^{R}(A,M)
\end{split}
\end{equation*}

For the second part, consider the following natural third quadrant spectral sequence from Theorem \ref{04.2}:
$$E_{p,q}^{2}= H_{p+q}\left(\Hom_{A}\left(\Nor_{A}\left(\Sym_{A}^{-q}\left(\Omega_{R\langle X\rangle \mid R}\otimes_{R\langle X\rangle}A\right)\right),M\right)\right) \Rightarrow \Ext_{R}^{-p-q}(A,M)$$
By \cite[Theorem 10.33]{Ro}, there is a natural five-term exact sequence as follows:
$$0\rightarrow E_{-1,0}^{2}\rightarrow H_{-1}\rightarrow E_{0,-1}^{2}\rightarrow E_{-2,0}^{2}\rightarrow H_{-2}$$
But we have:
\begin{equation*}
\begin{split}
 E_{-1,0}^{2} & = H_{-1}\left(\Hom_{A}\left(\Nor_{A}\left(\Sym_{A}^{0}\left(\Omega_{R\langle X\rangle \mid R}\otimes_{R\langle X\rangle}A\right)\right),M\right)\right) = H_{-1}\left(\Hom_{A}\left(\Nor_{A}(A),M\right)\right) \\
 & = H_{-1}\left(\Hom_{A}\left(A,M\right)\right) \cong H_{-1}(M) = 0
\end{split}
\end{equation*}
Similarly, we see that $E_{-2,0}^{2}=0$. Thus from the above sequence, we get the following natural isomorphism:
\begin{equation*}
\begin{split}
 H_{AQ}^{1}(A|R;M) & = H_{-1}\left(\Hom_{A}\left(\mathbb{L}^{A|R},M\right)\right) = H_{-1}\left(\Hom_{A}\left(\Nor_{A}\left(\Omega_{R\langle X\rangle \mid R}\otimes_{R\langle X\rangle}A\right),M\right)\right) \\
 & = H_{-1}\left(\Hom_{A}\left(\Nor_{A}\left(\Sym_{A}^{1}\left(\Omega_{R\langle X\rangle \mid R}\otimes_{R\langle X\rangle}A\right)\right),M\right)\right) = E_{0,-1}^{2} \cong H_{-1} = \Ext_{R}^{1}(A,M)
\end{split}
\end{equation*}

(iii): Consider the following natural first quadrant spectral sequence from Theorem \ref{04.2}:
$$E_{p,q}^{2}= H_{p+q}\left(\Nor_{A}\left(\Sym_{A}^{q}\left(\Omega_{R\langle X\rangle \mid R}\otimes_{R\langle X\rangle}A\right)\right)\otimes_{A}M\right) \Rightarrow \Tor_{p+q}^{R}(A,M)$$
We note that $E_{p,q}^{2}=0$ for every $p\geq 2$ and $q\leq 0$. Indeed, if $q<0$, then $\Sym_{A}^{q}\left(\Omega_{R\langle X\rangle \mid R}\otimes_{R\langle X\rangle}A\right)=0$, so $E_{p,q}^{2}=0$. Moreover, if $q=0$, then since $p\geq 2$, we have:
\small
$$E_{p,0}^{2}= H_{p}\left(\Nor_{A}\left(\Sym_{A}^{0}\left(\Omega_{R\langle X\rangle \mid R}\otimes_{R\langle X\rangle}A\right)\right)\otimes_{A}M\right) = H_{p}\left(\Nor_{A}(A)\otimes_{A}M\right)= H_{p}(A\otimes_{A}M)\cong H_{p}(M)=0$$
\normalsize
Therefore, it follows from Lemma \ref{05.1} that there is a natural five-term exact sequence as follows:
$$H_{3}\rightarrow E_{2,1}^{2} \rightarrow E_{0,2}^{2} \rightarrow H_{2} \rightarrow E_{1,1}^{2} \rightarrow 0$$
We note that:
\begin{equation*}
\begin{split}
 E_{1,1}^{2} & = H_{2}\left(\Nor_{A}\left(\Sym_{A}^{1}\left(\Omega_{R\langle X\rangle \mid R}\otimes_{R\langle X\rangle}A\right)\right)\otimes_{A}M\right) = H_{2}\left(\Nor_{A}\left(\Omega_{R\langle X\rangle \mid R}\otimes_{R\langle X\rangle}A\right)\otimes_{A}M\right) \\
 & = H_{2}\left(\mathbb{L}^{A|R}\otimes_{A}M\right) = H_{2}^{AQ}(A|R;M)
\end{split}
\end{equation*}
Similarly, we have $E_{2,1}^{2}= H_{3}^{AQ}(A|R;M)$. Moreover, $H_{2}= \Tor_{2}^{R}(A,M)$ and $H_{3}= \Tor_{3}^{R}(A,M)$. Therefore, we are left to determine the following term:
$$E_{0,2}^{2}= H_{2}\left(\Nor_{A}\left(\Sym_{A}^{2}\left(\Omega_{R\langle X\rangle \mid R}\otimes_{R\langle X\rangle}A\right)\right)\otimes_{A}M\right)$$
We express this term in terms of exterior powers. By \cite[Theorems 5.1.12 and 5.2.15]{CFH}, there is a semi-projective resolution $P\rightarrow \mathbb{L}^{A|R}$ such that $P_{n}$ is a free $A$-module for every $n\in \mathbb{Z}$ and $P_{n}=0$ for every $n< \inf\left(\mathbb{L}^{A|R}\right)$. Since $\mu_{A}$ is an isomorphism, $\mathfrak{k}= \Ker(\mu_{A})=0$, so $H_{0}\left(\mathbb{L}^{A|R}\right) \cong \Omega_{A|R}= \mathfrak{k}/ \mathfrak{k}^{2}=0$, whence $\inf\left(\mathbb{L}^{A|R}\right) \geq 1$. Hence $P_{n}=0$ for every $n<1$. Therefore, we have the following degreewise split short exact sequence of connective $A$-complexes:
$$0\rightarrow \Sigma^{-1}P \rightarrow \Cone\left(1^{\Sigma^{-1}P}\right) \rightarrow P \rightarrow 0$$
Since the Dold-Kan functor $\DK:\mathcal{C}_{\geq 0}(A) \rightarrow \mathpzc{s}\mathcal{M}\mathpzc{od}(A)$ is an equivalence of categories, it is a fortiori exact. Moreover, it is clear from the definition that it preserves degreewise split short exact sequences. As a result, we get the following degreewise split short exact sequence of simplicial $A$-modules:
$$0\rightarrow \DK\left(\Sigma^{-1}P\right) \rightarrow \DK\left(\Cone\left(1^{\Sigma^{-1}P}\right)\right) \rightarrow \DK(P) \rightarrow 0$$
Setting $U= \DK\left(\Sigma^{-1}P\right)$, $V= \DK\left(\Cone\left(1^{\Sigma^{-1}P}\right)\right)$, and $W= \DK(P)$ for convenience, the above short exact sequence takes the following form:
$$0\rightarrow U \rightarrow V \rightarrow W \rightarrow 0$$
We further note that for any $n\geq 0$, $P_{n}$ is a free $A$-module, so it is obvious that $U_{n}$, $V_{n}$, and $W_{n}$ are free $A$-modules. Now fix $n\geq 0$. Since $V_{n}= U_{n}\oplus W_{n}$, we get:
$$\Sym_{A}(V_{n})= \Sym_{A}(U_{n}\oplus W_{n})\cong \Sym_{A}(U_{n}) \otimes_{A} \Sym_{A}(W_{n})$$
Thus $\Sym_{A}(V_{n})$ is a $\Sym_{A}(U_{n})$-module via the canonical injection $\Sym_{A}(U_{n}) \rightarrow \Sym_{A}(V_{n})$. Therefore, as in the proof of Theorem \ref{03.6}, we have the following natural isomorphisms for every $i\geq 0$:
\begin{equation*}
\begin{split}
 H_{i}\left(K^{A}\left(U_{n},\Sym_{A}(V_{n})\right)\right) & \cong \Tor_{i}^{\Sym_{A}(U_{n})}\left(A,\Sym_{A}(V_{n})\right) \cong \Tor_{i}^{\Sym_{A}(U_{n})}\left(A,\Sym_{A}(U_{n}) \otimes_{A} \Sym_{A}(W_{n})\right) \\
 & \cong \Tor_{i}^{\Sym_{A}(U_{n})}\left(A,\Sym_{A}(U_{n})\right)\otimes_{A}\Sym_{A}(W_{n})
\end{split}
\end{equation*}
It follows that:
\begin{equation*}
 \label{eqn:damage piecewise}
 H_{i}\left(K^{A}\left(U_{n},\Sym_{A}(V_{n})\right)\right) \cong
 \begin{dcases}
  \Sym_{A}(W_{n}) & \textrm{if } i=0 \\
  0 & \textrm{if } i>0
 \end{dcases}
\end{equation*}
This means that the following augmented sequence is exact:
$$K^{A}\left(U_{n},\Sym_{A}(V_{n})\right)^{+}: \cdots \rightarrow K^{A}\left(U_{n},\Sym_{A}(V_{n})\right)_{2} \rightarrow K^{A}\left(U_{n},\Sym_{A}(V_{n})\right)_{1} \rightarrow K^{A}\left(U_{n},\Sym_{A}(V_{n})\right)_{0} $$$$ \rightarrow \Sym_{A}(W_{n}) \rightarrow 0$$
Writing $\Sym_{A}(W_{n})= \bigoplus_{j=0}^{\infty}\Sym_{A}^{j}(W_{n})$, we have
$$\bigoplus_{j=0}^{\infty} H_{i,j}\left(K^{A}\left(U_{n},\Sym_{A}(V_{n})\right)^{+}\right) \cong H_{i}\left(K^{A}\left(U_{n},\Sym_{A}(V_{n})\right)^{+}\right)= 0$$
for every $i\geq 0$, so $H_{i,j}\left(K^{A}\left(U_{n},\Sym_{A}(V_{n})\right)^{+}\right)=0$ for every $i,j \geq 0$. In particular, $H_{i,2}\left(K^{A}\left(U_{n},\Sym_{A}(V_{n})\right)^{+}\right)=0$ for every $i\geq 0$, so we get the following natural exact sequence of $A$-modules:
$$0 \rightarrow \textstyle\bigwedge_{A}^{2}U_{n}\otimes_{A}\Sym_{A}^{0}(V_{n}) \rightarrow \textstyle\bigwedge_{A}^{1}U_{n}\otimes_{A}\Sym_{A}^{1}(V_{n}) \rightarrow \textstyle\bigwedge_{A}^{0}U_{n}\otimes_{A}\Sym_{A}^{2}(V_{n}) \rightarrow \Sym_{A}^{2}(W_{n})\rightarrow 0$$
But this sequence is isomorphic to the following natural exact sequence of $A$-modules:
$$0 \rightarrow \textstyle\bigwedge_{A}^{2}U_{n} \rightarrow U_{n}\otimes_{A}V_{n} \rightarrow \Sym_{A}^{2}(V_{n}) \rightarrow \Sym_{A}^{2}(W_{n})\rightarrow 0$$
As a result, we get the following exact sequence of simplicial $A$-modules:
$$0 \rightarrow \textstyle\bigwedge_{A}^{2}U \rightarrow U\otimes_{A}V \rightarrow \Sym_{A}^{2}(V) \rightarrow \Sym_{A}^{2}(W) \rightarrow 0$$
Applying the normalization functor, we obtain the following exact sequence of $A$-complexes:
$$0 \rightarrow \Nor_{A}\left(\textstyle\bigwedge_{A}^{2}U\right) \rightarrow \Nor_{A}\left(U\otimes_{A}V\right) \rightarrow \Nor_{A}\left(\Sym_{A}^{2}(V)\right) \rightarrow \Nor_{A}\left(\Sym_{A}^{2}(W)\right) \rightarrow 0$$
For any $n\geq 0$, the sequence
$$0 \rightarrow \Nor_{A}\left(\textstyle\bigwedge_{A}^{2}U\right)_{n} \rightarrow \Nor_{A}\left(U\otimes_{A}V\right)_{n} \rightarrow \Nor_{A}\left(\Sym_{A}^{2}(V)\right)_{n} \rightarrow \Nor_{A}\left(\Sym_{A}^{2}(W)\right)_{n} \rightarrow 0$$
is a right-bounded exact $A$-complex of projective modules, so it is split exact, thereby the sequence
\small
$$0 \rightarrow \Nor_{A}\left(\textstyle\bigwedge_{A}^{2}U\right)_{n}\otimes_{A}M \rightarrow \Nor_{A}\left(U\otimes_{A}V\right)_{n}\otimes_{A}M \rightarrow \Nor_{A}\left(\Sym_{A}^{2}(V)\right)_{n}\otimes_{A}M \rightarrow \Nor_{A}\left(\Sym_{A}^{2}(W)\right)_{n}\otimes_{A}M \rightarrow 0$$
\normalsize
is split exact. Hence we get the following exact sequence of $A$-complexes:
$$0 \rightarrow \Nor_{A}\left(\textstyle\bigwedge_{A}^{2}U\right)\otimes_{A}M \rightarrow \Nor_{A}\left(U\otimes_{A}V\right)\otimes_{A}M \rightarrow \Nor_{A}\left(\Sym_{A}^{2}(V)\right)\otimes_{A}M \rightarrow \Nor_{A}\left(\Sym_{A}^{2}(W)\right)\otimes_{A}M \rightarrow 0$$
We now note that by \cite[Theorem 4.3.30]{CFH}, $\Cone\left(1^{\Sigma^{-1}P}\right)$ is a contractible $A$-complex, so we have
$$H_{i}\left(\Nor_{A}(V)\right)= H_{i}\left(\Nor_{A}\left(\DK\left(\Cone\left(1^{\Sigma^{-1}P}\right)\right)\right)\right)\cong H_{i}\left(\Cone\left(1^{\Sigma^{-1}P}\right)\right)=0$$
for every $i\geq 0$, i.e. $\Nor_{A}(V)$ is an exact $A$-complex. Since $\Nor_{A}(U)$ is an $A$-complex of projective, hence flat modules, it follows from \cite[Theorem 10.90]{Ro} that $\Nor_{A}(U)\otimes_{A} \Nor_{A}(V)$ is an exact $A$-complex. But then by Lemma \ref{03.3},
$$H_{i}\left(\Nor_{A}(U\otimes_{A}V)\right)= \pi_{i}(U\otimes_{A}V) \cong H_{i}\left(\Nor_{A}(U) \otimes_{A} \Nor_{A}(V)\right)=0$$
for every $i\geq 0$, so $\Nor_{A}(U\otimes_{A}V)$ is an exact $A$-complex. As $\Nor_{A}(U\otimes_{A}V)$ is an $A$-complex of projective, hence flat modules, we conclude from \cite[Theorem 10.90]{Ro} that $\Nor_{A}(U\otimes_{A}V)\otimes_{A}M$ is an exact $A$-complex. On the other hand, since $\pi_{i}(V)= H_{i}\left(\Nor_{A}(V)\right)=0$ for every $i\geq 0$, we conclude that the unique morphism $0\rightarrow V$ is a weak equivalence in $\mathpzc{s}\mathcal{M}\mathpzc{od}(A)$. As the source and the target of the latter morphism consists of free, hence projective modules, they are both cofibrant and thus bifibrant in $\mathpzc{s}\mathcal{M}\mathpzc{od}(A)$, so by Whitehead Homotopy Lemma (see \cite[Theorem 7.5.10]{Hi}), $0\rightarrow V$ is a homotopy equivalence in $\mathpzc{s}\mathcal{M}\mathpzc{od}(A)$. Since the functor $\Sym_{A}^{2}(-):\mathpzc{s}\mathcal{M}\mathpzc{od}(A) \rightarrow \mathpzc{s}\mathcal{M}\mathpzc{od}(A)$ is extended from a preadditive covariant functor $\Sym_{A}^{2}(-):\mathcal{M}\mathpzc{od}(A) \rightarrow \mathcal{M}\mathpzc{od}(A)$, we conclude from Corollary \ref{05.3} that $0=\Sym_{A}^{2}(0)\rightarrow \Sym_{A}^{2}(V)$ is a homotopy equivalence in $\mathpzc{s}\mathcal{M}\mathpzc{od}(A)$. But the source and the target of the latter morphism consists of free, hence projective modules, so they are both cofibrant in $\mathpzc{s}\mathcal{M}\mathpzc{od}(A)$, whence $0 \rightarrow \Sym_{A}^{2}(V)$ is a weak equivalence in $\mathpzc{s}\mathcal{M}\mathpzc{od}(A)$. As a result, $0=\Nor_{A}(0) \rightarrow \Nor_{A}\left(\Sym_{A}^{2}(V)\right)$ is a weak equivalence in $\mathcal{C}_{\geq 0}(A)$, i.e. a quasi-isomorphism. Again the source and the target of the latter morphism consists of projective modules, so they are both cofibrant in $\mathcal{C}_{\geq 0}(A)$, thereby $0 \rightarrow \Nor_{A}\left(\Sym_{A}^{2}(V)\right)$ is a homotopy equivalence in $\mathcal{C}_{\geq 0}(A)$. Consequently, $0=0 \otimes_{A}M \rightarrow \Nor_{A}\left(\Sym_{A}^{2}(V)\right)\otimes_{A}M$ is a homotopy equivalence in $\mathcal{C}_{\geq 0}(A)$, hence a quasi-isomorphism. In particular, we have $H_{i}\left(\Nor_{A}\left(\Sym_{A}^{2}(V)\right)\otimes_{A}M\right)=0$ for every $i\geq 0$, i.e. $\Nor_{A}\left(\Sym_{A}^{2}(V)\right)\otimes_{A}M$ is an exact $A$-complex. Now consider the following exact sequence of $A$-complexes:
$$0 \rightarrow \Nor_{A}\left(\textstyle\bigwedge_{A}^{2}U\right)\otimes_{A}M \rightarrow \Nor_{A}\left(U\otimes_{A}V\right)\otimes_{A}M \rightarrow \Nor_{A}\left(\Sym_{A}^{2}(V)\right)\otimes_{A}M \rightarrow \Nor_{A}\left(\Sym_{A}^{2}(W)\right)\otimes_{A}M \rightarrow 0$$
We just observed that $\Nor_{A}(U\otimes_{A}V)\otimes_{A}M$ and $\Nor_{A}\left(\Sym_{A}^{2}(V)\right)\otimes_{A}M$ are exact $A$-complexes, so we infer that:
$$H_{2}\left(\Nor_{A}\left(\Sym_{A}^{2}(W)\right)\otimes_{A}M\right)\cong H_{0}\left(\Nor_{A}\left(\textstyle\bigwedge_{A}^{2}U\right)\otimes_{A}M\right)$$
Since $P\rightarrow \mathbb{L}^{A|R}$ is a quasi-isomorphism, $W= \DK(P)\rightarrow \DK\left(\mathbb{L}^{A|R}\right)$ is a weak equivalence in $\mathpzc{s}\mathcal{M}\mathpzc{od}(A)$. But
$$\DK\left(\mathbb{L}^{A|R}\right) = \DK\left(\Nor_{A}\left(\Omega_{R\langle X\rangle |R}\otimes_{R\langle X\rangle}A\right)\right)\cong \Omega_{R\langle X\rangle |R}\otimes_{R\langle X\rangle}A,$$
so there is a weak equivalence $W\rightarrow \Omega_{R\langle X\rangle |R}\otimes_{R\langle X\rangle}A$ in $\mathpzc{s}\mathcal{M}\mathpzc{od}(A)$. As the source and the target of the latter morphism consists of free, hence projective modules, they are both cofibrant in $\mathpzc{s}\mathcal{M}\mathpzc{od}(A)$, so $W\rightarrow \Omega_{R\langle X\rangle |R}\otimes_{R\langle X\rangle}A$ is a homotopy equivalence in $\mathpzc{s}\mathcal{M}\mathpzc{od}(A)$, whence $\Sym_{A}^{2}(W)\rightarrow \Sym_{A}^{2}\left(\Omega_{R\langle X\rangle |R}\otimes_{R\langle X\rangle}A\right)$ is a homotopy equivalence in $\mathpzc{s}\mathcal{M}\mathpzc{od}(A)$. But the source and the target of the latter morphism consists of free, hence projective modules, so they are both cofibrant in $\mathpzc{s}\mathcal{M}\mathpzc{od}(A)$, whence $\Sym_{A}^{2}(W)\rightarrow \Sym_{A}^{2}\left(\Omega_{R\langle X\rangle |R}\otimes_{R\langle X\rangle}A\right)$ is a weak equivalence in $\mathpzc{s}\mathcal{M}\mathpzc{od}(A)$. As a result, $\Nor_{A}\left(\Sym_{A}^{2}(W)\right) \rightarrow \Nor_{A}\left(\Sym_{A}^{2}\left(\Omega_{R\langle X\rangle |R}\otimes_{R\langle X\rangle}A\right)\right)$ is a weak equivalence in $\mathcal{C}_{\geq 0}(A)$, i.e. a quasi-isomorphism. Again the source and the target of latter morphism consists of projective modules, so they are both cofibrant in $\mathcal{C}_{\geq 0}(A)$, thereby $\Nor_{A}\left(\Sym_{A}^{2}(W)\right) \rightarrow \Nor_{A}\left(\Sym_{A}^{2}\left(\Omega_{R\langle X\rangle |R}\otimes_{R\langle X\rangle}A\right)\right)$ is a homotopy equivalence in $\mathcal{C}_{\geq 0}(A)$. As a consequence,
$$\Nor_{A}\left(\Sym_{A}^{2}(W)\right)\otimes_{A}M \rightarrow \Nor_{A}\left(\Sym_{A}^{2}\left(\Omega_{R\langle X\rangle |R}\otimes_{R\langle X\rangle}A\right)\right)\otimes_{A}M$$
is a homotopy equivalence in $\mathcal{C}_{\geq 0}(A)$, hence a quasi-isomorphism. In particular, we have:
$$E_{0,2}^{2}= H_{2}\left(\Nor_{A}\left(\Sym_{A}^{2}\left(\Omega_{R\langle X\rangle \mid R}\otimes_{R\langle X\rangle}A\right)\right)\otimes_{A}M\right) \cong H_{2}\left(\Nor_{A}\left(\Sym_{A}^{2}\left(W\right)\right)\otimes_{A}M\right)$$
Now let $\mathcal{S}\mathpzc{kw}\mathcal{A}\mathpzc{lg}(A)$ denote the category of skew-commutative $A$-algebras, and $\sharp_{1}:\mathcal{S}\mathpzc{kw}\mathcal{A}\mathpzc{lg}(A)\rightarrow \mathcal{M}\mathpzc{od}(A)$ be the the projection functor onto degree one summand. Deploying the adjoint pairs of functors $\left(\bigwedge_{A}-,\sharp_{1}\right):\mathcal{M}\mathpzc{od}(A) \leftrightarrows \mathcal{S}\mathpzc{kw}\mathcal{A}\mathpzc{lg}(A)$, $\left(\pi_{0},\Diag\right):\mathpzc{s}\mathcal{M}\mathpzc{od}(A) \leftrightarrows \mathcal{M}\mathpzc{od}(A)$, $\left(\bigwedge_{A}-,\sharp_{1}\right):\mathpzc{s}\mathcal{M}\mathpzc{od}(A) \leftrightarrows \mathpzc{s}\mathcal{S}\mathpzc{kw}\mathcal{A}\mathpzc{lg}(A)$, and $\left(\pi_{0},\Diag\right):\mathpzc{s}\mathcal{S}\mathpzc{kw}\mathcal{A}\mathpzc{lg}(A) \leftrightarrows \mathcal{S}\mathpzc{kw}\mathcal{A}\mathpzc{lg}(A)$ successively, we get the following natural isomorphisms for every skew-commutative $A$-algebra $B$:
\begin{equation*}
\begin{split}
 \Mor_{\mathcal{S}\mathpzc{kw}\mathcal{A}\mathpzc{lg}(A)}\left(\textstyle\bigwedge_{A}\pi_{0}(U),B\right) & \cong \Mor_{\mathcal{M}\mathpzc{od}(A)}\left(\pi_{0}(U),B_{1}\right) \cong \Mor_{\mathpzc{s}\mathcal{M}\mathpzc{od}(A)}\left(U,B_{1}\right) \cong \Mor_{\mathpzc{s}\mathcal{S}\mathpzc{kw}\mathcal{A}\mathpzc{lg}(A)}\left(\textstyle\bigwedge_{A}U,B\right) \\
 & \cong \Mor_{\mathcal{S}\mathpzc{kw}\mathcal{A}\mathpzc{lg}(A)}\left(\pi_{0}\left(\textstyle\bigwedge_{A} U\right),B\right)
\end{split}
\end{equation*}
An application of Yoneda Lemma (see \cite[Corollary 8.5]{Aw}) implies that $\bigwedge_{A}H_{0}\left(\Nor_{A}(U)\right)= \bigwedge_{A}\pi_{0}(U) \cong \pi_{0}\left(\bigwedge_{A}U\right)= H_{0}\left(\Nor_{A}\left(\bigwedge_{A}U\right)\right)$, so in particular, $\bigwedge_{A}^{2}H_{0}\left(\Nor_{A}(U)\right) \cong H_{0}\left(\Nor_{A}\left(\bigwedge_{A}^{2}U\right)\right)$. Thus we get:
\begin{equation*}
\begin{split}
 H_{0}\left(\Nor_{A}\left(\textstyle\bigwedge_{A}^{2}U\right)\otimes_{A}M\right) & = \Coker\left(\partial_{1}^{\Nor_{A}\left(\bigwedge_{A}^{2}U\right)\otimes_{A}M}\right) = \Coker\left(\partial_{1}^{\Nor_{A}\left(\bigwedge_{A}^{2}U\right)}\otimes_{A}M\right) \\
 & \cong \Coker\left(\partial_{1}^{\Nor_{A}\left(\bigwedge_{A}^{2}U\right)}\right)\otimes_{A}M = H_{0}\left(\Nor_{A}\left(\textstyle\bigwedge_{A}^{2}U\right)\right)\otimes_{A}M \\
 & \cong \textstyle\bigwedge_{A}^{2}H_{0}\left(\Nor_{A}(U)\right)\otimes_{A}M
\end{split}
\end{equation*}
Also, using (ii), we have:
\begin{equation*}
\begin{split}
 H_{0}\left(\Nor_{A}(U)\right) & \cong H_{0}\left(\Nor_{A}\left(\DK\left(\Sigma^{-1}P\right)\right)\right) \cong H_{0}\left(\Sigma^{-1}P\right) = H_{1}(P) \cong H_{1}\left(\mathbb{L}^{A|R}\right) \cong H_{1}\left(\mathbb{L}^{A|R}\otimes_{A}A\right) \\
 & = H_{1}^{AQ}(A|R;A) \cong \Tor_{1}^{R}(A,A)
\end{split}
\end{equation*}
Putting everything together, we get:
\begin{equation*}
\begin{split}
 E_{0,2}^{2} & = H_{2}\left(\Nor_{A}\left(\Sym_{A}^{2}\left(\Omega_{R\langle X\rangle \mid R}\otimes_{R\langle X\rangle}A\right)\right)\otimes_{A}M\right) \cong H_{2}\left(\Nor_{A}\left(\Sym_{A}^{2}(W)\right)\otimes_{A}M\right) \\
 & \cong H_{0}\left(\Nor_{A}\left(\textstyle\bigwedge_{A}^{2}U\right)\otimes_{A}M\right) \cong \textstyle\bigwedge_{A}^{2}H_{0}\left(\Nor_{A}(U)\right)\otimes_{A}M \cong \textstyle\bigwedge_{A}^{2}\Tor_{1}^{R}(A,A)\otimes_{A}M
\end{split}
\end{equation*}
All in all, we get the following exact sequence:
$$\Tor_{3}^{R}(A,M)\rightarrow H_{3}^{AQ}(A|R;M) \rightarrow \textstyle\bigwedge_{A}^{2}\Tor_{1}^{R}(A,A)\otimes_{A}M \rightarrow \Tor_{2}^{R}(A,M) \rightarrow H_{2}^{AQ}(A|R;M) \rightarrow 0$$

For the second part, consider the following third quadrant spectral sequence from Theorem \ref{04.2}:
$$E_{p,q}^{2}= H_{p+q}\left(\Hom_{A}\left(\Nor_{A}\left(\Sym_{A}^{-q}\left(\Omega_{R\langle X\rangle \mid R}\otimes_{R\langle X\rangle}A\right)\right),M\right)\right) \Rightarrow \Ext_{R}^{-p-q}(A,M)$$
We note that $E_{p,q}^{2}=0$ for every $p\leq -2$ and $q\geq 0$. Indeed, if $q>0$, then $\Sym_{A}^{-q}\left(\Omega_{R\langle X\rangle |R}\otimes_{R\langle X\rangle}A\right)=0$, so $E_{p,q}^{2}=0$. Moreover, if $q=0$, then since $p\leq -2$, we have:
\begin{equation*}
\begin{split}
 E_{p,0}^{2} & = H_{p}\left(\Hom_{A}\left(\Nor_{A}\left(\Sym_{A}^{0}\left(\Omega_{R\langle X\rangle \mid R}\otimes_{R\langle X\rangle}A\right)\right),M\right)\right) = H_{p}\left(\Hom_{A}\left(\Nor_{A}(A),M\right)\right) \\
 & = H_{p}\left(\Hom_{A}(A,M)\right) \cong H_{p}(M)  = 0
\end{split}
\end{equation*}
Therefore, it follows from Lemma \ref{05.1} that there is a natural five-term exact sequence as follows:
$$0\rightarrow E_{-1,-1}^{2}\rightarrow H_{-2}\rightarrow E_{0,-2}^{2}\rightarrow E_{-2,-1}^{2}\rightarrow H_{-3}$$
We note that:
\begin{equation*}
\begin{split}
 E_{-1,-1}^{2} & = H_{-2}\left(\Hom_{A}\left(\Nor_{A}\left(\Sym_{A}^{1}\left(\Omega_{R\langle X\rangle \mid R}\otimes_{R\langle X\rangle}A\right)\right),M\right)\right) \\
 & = H_{-2}\left(\Hom_{A}\left(\Nor_{A}\left(\Omega_{R\langle X\rangle \mid R}\otimes_{R\langle X\rangle}A\right),M\right)\right) \\
 & = H_{-2}\left(\Hom_{A}\left(\mathbb{L}^{A|R},M\right)\right) \\
 & = H^{2}_{AQ}(A|R;M)
\end{split}
\end{equation*}
Similarly, we have $E_{-2,-1}^{2}= H_{AQ}^{3}(A|R;M)$. Moreover, $H_{-2}= \Ext_{R}^{2}(A,M)$ and $H_{-3}= \Ext_{R}^{3}(A,M)$. Therefore, we are left to determine the following term:
$$E_{0,-2}^{2} = H_{-2}\left(\Hom_{A}\left(\Nor_{A}\left(\Sym_{A}^{2}\left(\Omega_{R\langle X\rangle \mid R}\otimes_{R\langle X\rangle}A\right)\right),M\right)\right)$$
We express this term in terms of exterior powers. Consider the exact sequence
$$0\rightarrow \Nor_{A}\left(\textstyle\bigwedge_{A}^{2}U\right) \rightarrow \Nor_{A}\left(U\otimes_{A}V\right) \rightarrow \Nor_{A}\left(\Sym_{A}^{2}(V)\right) \rightarrow \Nor_{A}\left(\Sym_{A}^{2}(W)\right) \rightarrow 0$$
of $A$-complexes obtained above. For any $n\geq 0$, the sequence
$$0\rightarrow \Nor_{A}\left(\textstyle\bigwedge_{A}^{2}U\right)_{n} \rightarrow \Nor_{A}\left(U\otimes_{A}V\right)_{n} \rightarrow \Nor_{A}\left(\Sym_{A}^{2}(V)\right)_{n} \rightarrow \Nor_{A}\left(\Sym_{A}^{2}(W)\right)_{n} \rightarrow 0$$
is split exact, so the sequence
$$0\rightarrow \Hom_{A}\left(\Nor_{A}\left(\Sym_{A}^{2}(W)\right)_{n},M\right) \rightarrow \Hom_{A}\left(\Nor_{A}\left(\Sym_{A}^{2}(V)\right)_{n},M\right) \rightarrow \Hom_{A}\left(\Nor_{A}\left(U\otimes_{A}V\right)_{n},M\right) $$$$ \rightarrow \Hom_{A}\left(\Nor_{A}\left(\textstyle\bigwedge_{A}^{2}U\right)_{n},M\right) \rightarrow 0$$
is split exact. Hence we get the following exact sequence of $A$-complexes:
$$0\rightarrow \Hom_{A}\left(\Nor_{A}\left(\Sym_{A}^{2}(W)\right),M\right) \rightarrow \Hom_{A}\left(\Nor_{A}\left(\Sym_{A}^{2}(V)\right),M\right) \rightarrow \Hom_{A}\left(\Nor_{A}\left(U\otimes_{A}V\right),M\right) $$$$ \rightarrow \Hom_{A}\left(\Nor_{A}\left(\textstyle\bigwedge_{A}^{2}U\right),M\right) \rightarrow 0$$
As $\Nor_{A}\left(U\otimes_{A}V\right)$ is an exact $A$-complex of projective modules, we conclude from a dual version of [Ro, Theorem 10.90] that $\Hom_{A}\left(\Nor_{A}\left(U\otimes_{A}V\right),M\right)$ is an exact $A$-complex. On the other hand, we observed that $0\rightarrow \Nor_{A}\left(\Sym_{A}^{2}(V)\right)$ is a homotopy equivalence in $\mathcal{C}_{\geq 0}(A)$, so $\Hom_{A}\left(\Nor_{A}\left(\Sym_{A}^{2}(V)\right),M\right) \rightarrow \Hom_{A}(0,M)=0$ is a homotopy equivalence in $\mathcal{C}_{\geq 0}(A)$, hence a quasi-isomorphism. In particular, we have $H_{i}\left(\Hom_{A}\left(\Nor_{A}\left(\Sym_{A}^{2}(V)\right),M\right)\right)=0$ for every $i\geq 0$, i.e. $\Hom_{A}\left(\Nor_{A}\left(\Sym_{A}^{2}(V)\right),M\right)$ is an exact $A$-complex. Now it follows from the above exact sequence that:
$$H_{-2}\left(\Hom_{A}\left(\Nor_{A}\left(\Sym_{A}^{2}(W)\right),M\right)\right) \cong H_{0}\left(\Hom_{A}\left(\Nor_{A}\left(\textstyle\bigwedge_{A}^{2}U\right),M\right)\right)$$
We observed that $\Nor_{A}\left(\Sym_{A}^{2}(W)\right) \rightarrow \Nor_{A}\left(\Sym_{A}^{2}\left(\Omega_{R\langle X\rangle \mid R}\otimes_{R\langle X\rangle}A\right)\right)$ is a homotopy equivalence in $\mathcal{C}_{\geq 0}(A)$, so
$$\Hom_{A}\left(\Nor_{A}\left(\Sym_{A}^{2}\left(\Omega_{R\langle X\rangle \mid R}\otimes_{R\langle X\rangle}A\right)\right),M\right) \rightarrow \Hom_{A}\left(\Nor_{A}\left(\Sym_{A}^{2}(W)\right),M\right)$$
is a homotopy equivalence in $\mathcal{C}_{\geq 0}(A)$, hence a quasi-isomorphism. In particular, we have:
\small
$$E_{0,-2}^{2} = H_{-2}\left(\Hom_{A}\left(\Nor_{A}\left(\Sym_{A}^{2}\left(\Omega_{R\langle X\rangle \mid R}\otimes_{R\langle X\rangle}A\right)\right),M\right)\right) \cong H_{-2}\left(\Hom_{A}\left(\Nor_{A}\left(\Sym_{A}^{2}(W)\right),M\right)\right)$$
\normalsize
We observed that $\bigwedge_{A}^{2}H_{0}\left(\Nor_{A}(U)\right) \cong H_{0}\left(\Nor_{A}\left(\bigwedge_{A}^{2}U\right)\right)$, so we get:
\small
\begin{equation*}
\begin{split}
 H_{0}\left(\Hom_{A}\left(\Nor_{A}\left(\textstyle\bigwedge_{A}^{2}U\right),M\right)\right) & = \Ker\left(\partial_{0}^{\Hom_{A}\left(\Nor_{A}\left(\textstyle\bigwedge_{A}^{2}U\right),M\right)}\right) = \Ker\left(\Hom_{A}\left(\partial_{1}^{\Nor_{A}\left(\textstyle\bigwedge_{A}^{2}U\right)},M\right)\right) \\
 & \cong \Hom_{A}\left(\Coker\left(\partial_{1}^{\Nor_{A}\left(\textstyle\bigwedge_{A}^{2}U\right)}\right),M\right) = \Hom_{A}\left(H_{0}\left(\Nor_{A}\left(\textstyle\bigwedge_{A}^{2}U\right)\right),M\right) \\
 & \cong \Hom_{A}\left(\textstyle\bigwedge_{A}^{2}H_{0}\left(\Nor_{A}\left(U\right)\right),M\right)
\end{split}
\end{equation*}
\normalsize
Also, we observed that $H_{0}\left(\Nor_{A}(U)\right) \cong \Tor_{1}^{R}(A,A)$. Putting everything together, we get:
\small
\begin{equation*}
\begin{split}
 E_{0,-2}^{2} & = H_{-2}\left(\Hom_{A}\left(\Nor_{A}\left(\Sym_{A}^{2}\left(\Omega_{R\langle X\rangle \mid R}\otimes_{R\langle X\rangle}A\right)\right),M\right)\right) \cong H_{-2}\left(\Hom_{A}\left(\Nor_{A}\left(\Sym_{A}^{2}(W)\right),M\right)\right) \\
 & \cong H_{0}\left(\Hom_{A}\left(\Nor_{A}\left(\textstyle\bigwedge_{A}^{2}U\right),M\right)\right) \cong \Hom_{A}\left(\textstyle\bigwedge_{A}^{2}H_{0}\left(\Nor_{A}(U)\right),M\right) \cong \Hom_{A}\left(\textstyle\bigwedge_{A}^{2}\Tor_{1}^{R}(A,A),M\right)
\end{split}
\end{equation*}
\normalsize
All in all, we get the following exact sequence:
$$0\rightarrow H_{AQ}^{2}(A|R;M)\rightarrow \Ext_{R}^{2}(A,M)\rightarrow \Hom_{A}\left(\textstyle\bigwedge_{A}^{2}\Tor_{1}^{R}(A,A),M\right)\rightarrow H_{AQ}^{3}(A|R;M)\rightarrow \Ext_{R}^{3}(A,M)\qedhere$$
\end{proof}



\begin{thebibliography}{99}

\bibitem[An]{An}{M. Andr\'{e}}, {\it Homologie des Alg\`{e}bres Commutatives}, Grundlehren Math. Wiss. 206, Springer-Verlag, Berlin, 1974.

\bibitem[Av1]{Av1}{L. L. Avramov}, {\it Flat Morphisms of Complete Intersections}, Soviet Math. Dokl., Vol. 16 (1975), No. 6.

\bibitem[Av2]{Av2}{L. L. Avramov}, {\it Locally Complete Intersection Homomorphisms and a Conjecture of Quillen on the Vanishing of Cotangent Homology}, Ann. Math. 150(2) (1999), 455-487.

\bibitem[Av3]{Av3}{L. L. Avramov}, {\it Local Rings of Finite Simplicial Dimension}, Bulletin of the American Mathematical Society, Vol. 10, No. 2, Apr. 1984.

\bibitem[Aw]{Aw}{S. Awodey}, {\it Category Theory}, Oxford Logic Guides 52, 2nd Edition, Oxford University Press, 2010.

\bibitem[BH]{BH}{W. Bruns and J. Herzog}, {\it Cohen-Macaulay Rings}, Cambridge Studies in Advanced Mathematics, 39, Cambridge University Press, Cambridge, 1993.

\bibitem[BI]{BI}{B. Briggs and S. B. Iyengar}, {\it Rigidity Properties of the Cotangent Complex}, Journal of the American Mathematical Society, 2022.

\bibitem[CFH]{CFH}{L. W. Christesen, H-B Foxby, and H. Holm}, {\it Derived Category Methods in Commutative Algebra}, \url{https://www.math.ttu.edu/~lchriste/download/dcmca.pdf}.

\bibitem[Do]{Do}{A. Dold}, {\it Homology of Symmetric Products and Other Functors of Complexes}, Ann. Math., Vol. 68, No. 1, 1958.

\bibitem[DFT]{DFT}{K. Divaani-Aazar, H. Faridian and M. Tousi}, {\it Local Homology, Koszul Homology and Serre Classes}, Rocky Mountain Journal of Mathematics, Vol. 48, No. 6, 2018.

\bibitem[DP]{DP}{A. Dold and D. Puppe}, {\it Homologie nicht-additiver Funktoren}, Ann. Inst. Fourier 11 (1961), 201-312.

\bibitem[Fa]{Fa}{H. Faridian}, {\it Model Category Structure on Simplicial Algebras via Dold-Kan Correspondence}, arXiv:2405.01752.

\bibitem[Gr]{Gr}{A. Grothendieck}, {\it EGA IV}, Publications Math\'{e}matiques de L'Institut des Hautes Scientifiques 20, 5-251 (1964).

\bibitem[GS]{GS}{P. G. Goerss and K. Schemmerhorn}, {\it Model Categories and Simplicial Methods}, Interactions between Homotopy Theory and Algebra, Contemporary Mathematics 436, AMS 2007.

\bibitem[Hi]{Hi}{P. S. Hirschhorn}, {\it Model Categories and Their Localizations}, American Mathematical Society, Mathematical Surveys and Monographs 99, 2002, ISBN 0-8218-3279-4.

\bibitem[Ho]{Ho}{M. Hovey}, {\it Model Categories}, Mathematical Surveys and Monographs, vol. 63, Amer. Math. Soc., Providence, RI, 1999.

\bibitem[Il]{Il}{L. Illusie}, {\it Complexe Cotangent et D\'{e}formations. I}, Lecture Notes in Mathematics, Vol. 239, Springer-Verlag, Berlin, 1971.

\bibitem[Iy]{Iy}{S. B. Iyengar}, {\it Andr\`{e}-Quillen Homology of Commutative Algebras}, in Interactions between Homotopy Theory and Algebra, Contemp. Math., Vol. 436 (American Mathematical Society, Providence, RI, 2007), 203-234.

\bibitem[LS]{LS}{S. Lichtenbaum and M. Schlessinger}, {\it The Cotangent Complex of a Morphism}, Trans. Amer. Math. Soc. 128 (1967), 41-70.

\bibitem[Mc]{Mc}{J. McCleary}, {\it A User's Guide to Spectral Sequences}, Cambr. Stud. Adv. Math. 58, Cambridge University Press, Cambridge, 2001.

\bibitem[MR]{MR}{J. Majadas and A. G. Rodicio}, {\it Smoothness, Regularity and Complete Intersection}, London Mathematical Society Lecture Note Series, vol. 373. Cambridge University Press, Cambridge, 2010.

\bibitem[Qu1]{Qu1}{D. G. Quillen}, {\it Homology of Commutative Rings}, mimeographed notes, MIT 1968.

\bibitem[Qu2]{Qu2}{D. G. Quillen}, {\it Homotopical Algebra}, Lecture Notes in Mathematics, vol. 43, Springer-Verlag, 1967.

\bibitem[Qu3]{Qu3}{D. G. Quillen}, {\it On the (Co-)homology of Commutative Rings}, in: Applications of Categorical Algebra; New York, 1968, Proc. Symp. Pure Math. 17, Amer. Math. Soc., Providence, RI, 1970; pp. 65-87.

\bibitem[Ro]{Ro}{J. J. Rotman}, {\it An Introduction to Homological Algebra}, Second Edition, Universitext, Springer-Verlag, 2009.

\bibitem[St]{St} {\it The Stacks Project}, \url{https://stacks.math.columbia.edu/tag/012A} and \url{https://stacks.math.columbia.edu/tag/012K}.

\bibitem[We]{We}{C. Weibel}, {\it An Introduction to Homological Algebra}, Cambr. Stud. Adv. Math. 38, Cambridge University Press, Cambridge, 1994.

\end{thebibliography}
\end{document}